\documentclass[11pt,a4paper,reqno]{amsart}

\usepackage{amsfonts}
\usepackage{amsmath}
\usepackage{amssymb}
\usepackage{hyperref}
\usepackage{graphicx}
\usepackage{epstopdf}

\usepackage{amsthm}

\usepackage{mathrsfs}
\usepackage{dutchcal}

\newcommand{\Div}{\operatorname{div}}
\newcommand{\capacity}{\operatorname{Cap}}

\newcommand{\supp}{\operatorname{supp}}
\newcommand{\tail}{\operatorname{tail}}
\newcommand{\head}{\operatorname{head}}

\numberwithin{equation}{section}
\newtheorem{theorem}{Theorem}[section]
\newtheorem{lemma}[theorem]{Lemma}
\newtheorem{remark}[theorem]{Remark}
\newtheorem{corollary}[theorem]{Corollary}
\newtheorem{proposition}[theorem]{Proposition}

\newtheorem{example}[theorem]{Example}

\newcommand{\HH}{H}

\allowdisplaybreaks

\begin{document}
\author{Qingsong Gu}
\address{School of Mathematics, Nanjing University, Nanjing 210093, P. R. China} \email{qingsonggu@nju.edu.cn}

\author{Lu Hao}
\address{Universit\"{a}t Bielefeld, Fakult\"{a}t f\"{u}r Mathematik, Postfach 100131, D-33501, Bielefeld, Germany}
\email{lhao@math.uni-bielefeld.de}

\author{Xueping Huang}
\address{School of Mathematics and Statistics, Nanjing University of Information Science and Technology,
	Nanjing 210044, P. R. China}
\email{hxp@nuist.edu.cn}
	
\author{Yuhua Sun}
\address{School of Mathematical Sciences and LPMC, Nankai University, 300071
		Tianjin, P. R. China}
\email{sunyuhua@nankai.edu.cn}
\title[Lane--Emden Inequalities]{Flow Decomposition, Green Testing, and Lane--Emden Inequalities on Weighted Graphs}

\thanks{\noindent
 Q. Gu was supported by the National Natural Science Foundation of China (Grant No. 12101303 and 12171354).
 L. Hao was funded by the Deutsche Forschungsgemeinschaft (DFG, German Research Foundation), Project-ID 317210226, SFB 1283.
 X. Huang was supported by
	the National Natural Science Foundation of China (Grant No. 11601238).
 Y. Sun was funded by the National Natural Science Foundation of
	China (Grant No. 12371206) and the Fundamental Research Funds for the Central Universities, No. 050-63263078.}


\begin{abstract}
We study positive solutions of the superlinear Lane--Emden inequality
\[
-\Delta u\ge \sigma u^q,\qquad q>1,
\]
on infinite locally finite weighted graphs and connected domains.  When the
Dirichlet Green function is finite, the existence of a positive solution is equivalent to
\[
G_\Omega\bigl(\sigma g_\Omega(o,\cdot)^q\bigr)(x)
\le C g_\Omega(o,x)
\]
for some pole \(o\in\Omega\).  Under Green function estimates, this yields
sharp existence criteria and the Serrin-type exponents on \(\mathbb Z^d\) and
orthant domains.

For nonexistence, the principal method is flow decomposition.  Its basic
estimate bounds Green energy from below in terms of the relative capacities
of intrinsic balls.  
For \(\sigma>0\), set
\(\nu=\sigma\mu\).  We show that if \(d_\rho\) is a complete \(\nu\)-adapted path metric
and
\[
\int_1^\infty
\frac{r^{2q-1}}{\nu(B_{d_\rho}(o,r))^{q-1}}\,dr=\infty,
\]
then every nonnegative solution is identically zero. The proof combines a flow decomposition of the acyclic Green current, a
pathwise Hardy estimate, and a relative capacity estimate.  It requires none
of \textnormal{(VD)}, \textnormal{(PI)}, \textnormal{(P$_0$)}, or the
\textnormal{(3G)} condition.
\end{abstract}

\maketitle
\tableofcontents

\section{Introduction}

Liouville-type theorems for semilinear elliptic inequalities ask how the
large-scale geometry of an underlying space restricts positive supersolutions.
In Euclidean spaces this principle is expressed through critical exponents,
whereas on manifolds and graphs the role of dimension is played by volume
growth and the behavior of Green functions.

The present paper studies this question for the Lane--Emden-type inequality
\[
-\Delta u\ge \sigma u^q,\qquad q>1,
\]
on infinite, locally finite weighted graphs.  We use two methods with
distinct roles.  Green testing characterizes the existence of positive
solutions.  Flow decomposition is the principal method for nonexistence: it
gives a lower bound for Green energy in terms of relative capacities and
leads to the sharp volume-growth Liouville theorem stated below.  The theorem
is formulated using a complete intrinsic metric adapted to
\(\nu=\sigma\mu\).  When \(\sigma\equiv1\), it immediately implies that
\[
\sum_{n\ge1}\frac{n^{2q-1}}{\mu(B(o,n))^{q-1}}=\infty
\]
forces every nonnegative solution of \(-\Delta u\ge u^q\) to vanish.


The main new method is flow decomposition.  In a finite intrinsic ball, the
Green function is interpreted as a voltage, and the associated unit current
is decomposed into directed paths.  A one-dimensional Hardy estimate controls
the first-exit voltages along these paths, while Thomson's principle gives the
required relative capacity bound.  An intrinsic radial cutoff then yields the
volume-growth criterion.  This argument avoids \textnormal{(VD)},
\textnormal{(PI)}, \textnormal{(P$_0$)}, and \textnormal{(3G)}-type Green
function assumptions.


This result places the discrete problem within the broader Liouville theory for
elliptic inequalities.  For the classical inequality
\[
-\Delta u\ge u^q,
\]
the interaction between diffusion, nonlinearity, and geometry has been studied
extensively in Euclidean spaces and in related problems for equations,
inequalities, systems, higher order operators, and more general nonlinearities;
see, for example,
\cite{BVP01,BM98,AM19,AMP06,GS81,MP01,SZ02,WangXiao16}.  Boundary geometry may
change the critical threshold.  For instance, in the half-space
\(\mathbb R^d_+\), the inequality with Dirichlet boundary condition
\[
-\Delta u\ge u^q,\qquad u|_{\partial\mathbb R^d_+}=0,
\]
has the sharp threshold \(q=(d+1)/(d-1)\); see
\cite{BVP01,BM98,AM25,AM25-arx}.  On manifolds and graphs, the same question
naturally leads to volume-growth criteria and Green-function estimates, which
are the main themes of this paper.

A central tool in this circle of problems is the passage from a differential
inequality to a potential inequality.  In a typical form, one replaces
\[
-Lu\ge \nu
\]
with $\nu$ a Radon measure by
\[
u\ge K\nu,
\]
where \(K\) is the fundamental solution associated with the differential operator
\(L\).  This representation-formula viewpoint, together with a priori
estimates, provides a common language for many existence and nonexistence
results; see \cite{AM19,AMP06,KV99,MP01} and the references therein.

This potential-theoretic viewpoint provides the bridge from the Euclidean
case to geometric settings.  Once the inequality is expressed through the
Green operator, the decisive information is encoded in the global behavior of
Green functions and in the volume growth of metric balls.

On complete Riemannian manifolds,  for the inequality \(-\Delta u\ge u^q\), Grigor'yan and Sun \cite{GS14} proved a
volume-growth nonexistence criterion with sharp  logarithmic exponents.  Later, for the more general inequality with potential
\[
-\Delta u\ge u^q \sigma
\]
where  $\sigma$ is a Radon measure,
 Grigor'yan, Sun and Verbitsky \cite{GSV20} obtained
sharp integral-type nonexistence criteria assuming \textnormal{(VD)} and
\textnormal{(PI)} conditions.
  In the unweighted case, their nonexistence criterion reduces to
  \begin{equation}\label{eq:GSV-conjecture}
 \int^\infty
 \frac{r^{2q-1}}{\mu(B(o,r))^{q-1}}\,dr=\infty.
  \end{equation}
They conjectured that this divergence condition is sufficient for the nonexistence of positive solutions to \(-\Delta u\ge u^q\), even without \textnormal{(VD)} and 	\textnormal{(PI)} conditions.
This conjecture is a useful benchmark for the present work. 

The work \cite{GSV20} is built upon the potential-theoretic tools developed in
\cite{GV20,KV99}.  In this approach, nonlinear inequalities of the form
\[
u\ge K(u^q \sigma)
\]
are studied through weighted norm inequalities. Typical hypotheses include a weak maximum principle and a quasi-metric
condition, the latter being equivalent to a \textnormal{(3G)}-type condition
on the kernel. Under
such hypotheses one obtains sharp tests involving growth of Green functions. We also mention the differential approach of \cite{GV19}, which does not rely on such assumptions and partly motivates the present work.

In this paper we study this circle of questions on infinite locally finite weighted graphs.  We use the graph setting not merely as a discrete analogue of a manifold problem, but as a natural geometric framework in which potential theory, random walks, electrical networks, and nonlinear elliptic inequalities meet. This setting allows us to avoid regularity issues and to focus on the difficulties arising in estimates of growth-type quantities. The graph setting also has the advantage that it presents both local and nonlocal features. Methods developed in this discrete setting may also shed light on the corresponding continuous problems.
 For recent related work on semilinear elliptic equations and inequalities on
graphs, we refer to \cite{CHZ25,HS23,HLY20,MPS25} and the references therein.


We now state the main results.  Only the notation needed for the statements is recalled here; the background on random walks, killed kernels, Green functions, and intrinsic metrics is collected in Section~\ref{sec2}.  Let \((V,\mu)\) be an infinite, connected, locally finite
weighted graph.  Thus \(\mu_{xy}=\mu_{yx}>0\) if and only if \(x\sim y\), and
\[
\mu(x):=\sum_{y\sim x}\mu_{xy}.
\]
We use the normalized graph Laplacian
\[
\Delta f(x)
=
\frac1{\mu(x)}
\sum_{y\sim x}\mu_{xy}\bigl(f(y)-f(x)\bigr).
\]
Let \(\Omega\subset V\) be connected.  Whenever the full graph Laplacian is
applied to a function on \(\Omega\), the function is extended by zero to
\(V\setminus\Omega\).  Therefore the Dirichlet problem on \(\Omega\) can be
written as
\begin{equation}\label{di-1}
	\begin{cases}
		-\Delta u\ge \sigma u^q & \text{in }\Omega,\\
		u=0 & \text{on }\Omega^c,
	\end{cases}
\end{equation}
where \(q>1\) and $\sigma$ is a nonnegative nonzero function on $\Omega$.  If
\(\Omega=V\), this becomes
\begin{equation}\label{di-2}
	-\Delta u\ge \sigma u^q\qquad\text{in }V.
\end{equation}
If \((V,\mu)\) is parabolic, then \eqref{di-2} has no positive solution:
such a solution would be nonnegative and superharmonic, hence constant, while
the inequality at a vertex where \(\sigma>0\) rules out a positive constant.
Accordingly, the criteria below involving a finite whole-graph Green function
are stated in the non-parabolic case.

We write
\[
\nu(x):=\sigma(x)\mu(x),
\]
so that \(\nu\) is the measure naturally associated with the potential
\(\sigma\).

Let $g$ and \(g_\Omega\) denote the whole-graph Green function and the Dirichlet Green function on $\Omega$, respectively.  The Green
function \(g_\Omega\) is finite whenever either \(\Omega\ne V\), or \(\Omega=V\) and the
graph is non-parabolic; see Section~\ref{sec2}.

Our first result gives an exact Green-kernel characterization of positive
solutions.  In addition to the differential inequality, we consider the
associated integral inequality
\begin{equation}\label{ii}
	u(x)\ge G_\Omega(\sigma u^q)(x),
	\qquad x\in\Omega.
\end{equation}

\begin{theorem}\label{thm:green-testing-equivalence}
	Let \((V,\mu)\) be an infinite, connected, locally finite weighted graph, and
	let \(\Omega\subset V\) be connected.  Assume that either \(\Omega\ne V\), or
	\(\Omega=V\) and \((V,\mu)\) is non-parabolic.  Let
	\(0\not\equiv\sigma\in\ell^+(\Omega)\) and \(1<q<\infty\).  Then the following
	are equivalent:
	\begin{enumerate}
		\item[\textnormal{(I)}] The differential inequality \eqref{di-1}, with \eqref{di-2} when
		\(\Omega=V\), admits a positive solution.
		\item[\textnormal{(II)}] The integral inequality \eqref{ii} admits a positive
		solution.
	    \item[\textnormal{(III)}] For some, equivalently for every, \(o\in\Omega\),
		there exists \(C>0\) such that
		\begin{equation}\label{e1}
			G_\Omega\bigl(\sigma\,g_\Omega(o,\cdot)^q\bigr)(x)
			\le C g_\Omega(o,x),
			\qquad x\in\Omega.
		\end{equation}
	\end{enumerate}
\end{theorem}
\begin{remark}\label{rem:positive-nonnegative}
By a positive solution we mean a function that is positive at every vertex of the domain.  In the following we also consider nonnegative solutions.  On a connected domain, if a nonnegative solution vanishes at one vertex, then it vanishes identically by the elementary maximum principle; hence the two formulations differ only by the zero solution.
\end{remark}

Theorem~\ref{thm:green-testing-equivalence} follows the potential-theoretic
philosophy of \cite{KV99} and \cite{GV19, GV20},
 but it also uses a
feature specific to graphs.  In the general kernel theory, weak maximum
principle or quasi-metric/\((3G)\)-type assumptions are often imposed in order
to pass from the nonlinear integral inequality to a Green-kernel testing
condition.  Here no such assumption is needed for
Theorem~\ref{thm:green-testing-equivalence}.
In this sense \eqref{e1} is the basic existence test in the graph setting. As applications, we determine the critical exponents on the integer lattice $\mathbb{Z}^d$ and on its $k$-orthant domains; see Section \ref{sec7} and Section \ref{sec-zk}.

The next result explains how \eqref{e1} (in the case $\Omega=V$) is related to the level-set criteria
appearing in \cite{GSV20, GV20}.

For \(r>0\), set
\[
A_r(o):=\{y\in V:g(o,y)>r^{-1}\}.
\]

\begin{theorem}\label{thm:green-level-criterion}
	Let \(1<q<\infty\), let \(0\not\equiv\sigma\in\ell^+(V)\), set
	\(\nu=\sigma\mu\), and assume that \((V,\mu)\) is non-parabolic.  If
	\eqref{di-2} admits a positive solution, then, for every
	\(o\in V\), \begin{equation}\label{et1}
		\sum_{y\in V}g(o,y)^q\,\nu(y)<\infty,
	\end{equation}
 and there exist
	\(r_0>0\) and \(C>0\) such that
	\begin{equation}\label{et2}
		\sup_{x\in V}
		\sum_{y\in A_r(o)} g(x,y)\,\nu(y)
		\le C r^{q-1},
		\qquad r>r_0.
	\end{equation}
	Conversely, if the Green function satisfies the \textnormal{\((3G)\)} inequality (see Subsection \ref{subsection:3G})
	and \eqref{et1}, \eqref{et2} hold for some \(o\in V\), then \eqref{di-2}
	admits a positive solution.
\end{theorem}
\begin{remark}\label{rem:relation-Green-energy}
Integrability \eqref{et1} is a Green energy
condition, and \eqref{et2} is a uniform localization condition. Condition \eqref{et1} does not imply \eqref{et2} in general, and does not control how \(\nu\) may
concentrate on sparse Green level sets. This distinction is also useful conceptually: diagonal Green energy is related to \(L^q\)-Liouville properties, whereas positive solvability of the nonlinear inequality is governed by the stronger pointwise testing condition. See Section \ref{sec-example} for further discussion.
\end{remark}

Our second goal is to obtain sharp volume-growth criteria: to determine when volume growth alone rules out positive supersolutions.

In \cite{GHS23}, under a uniform ellipticity assumption
\textnormal{(P$_0$)} (see Subsection~\ref{subsection:conditions}), it was proved that the logarithmic volume bound
\[
\mu(B(o,n))
\le C
n^{\frac{2q}{q-1}}(\log n)^{\frac1{q-1}}
\]
implies nonexistence of positive solutions to the superlinear inequality
\[
-\Delta u\ge u^q.
\]
Motivated by the conjecture in \cite{GSV20} for the manifold case, in \cite{GHS23} it was conjectured that on an arbitrary weighted graph the divergence
condition
\[
\sum_{n=1}^{\infty}
\frac{n^{2q-1}}{\mu(B(o,n))^{q-1}}
=\infty
\]
implies the nonexistence of positive solutions.

One of the main contributions of the present paper is the following sharp
volume-growth Liouville theorem.  Its unweighted specialization resolves the
graph conjecture of \cite{GHS23} stated above.
For a general positive potential \(\sigma\), the graph distance does not
reflect the growth of \(\sigma\).  The appropriate replacement is an intrinsic
path metric adapted to the measure \(\nu=\sigma\mu\).  The relevant definitions
are recalled in Subsection~\ref{subsection:intrinsic-metrics}.

\begin{theorem}\label{thm:weighted-liouville}
Let \(q>1\), let \(0<\sigma\in\ell^+(V)\), and set \(\nu=\sigma\mu\).  Let
\(\rho\) be a \(\nu\)-adapted edge-length function and let \(d_\rho\) be the
associated intrinsic path metric.  Assume that \((V,d_\rho)\) is complete.  If,
for some \(o\in V\),
\begin{equation}\label{eq:weighted-volume-divergence}
\int_1^\infty
\frac{r^{2q-1}}{\nu(B_{d_\rho}(o,r))^{q-1}}\,dr
=\infty,
\end{equation}
then every nonnegative solution of \eqref{di-2} is identically zero.
\end{theorem}
The divergence condition
\eqref{eq:weighted-volume-divergence} is independent of the choice of the
base point; see Remark~\ref{rem:base-point-independence}.
\begin{remark}
In this intrinsic-metric result we assume that \(\sigma\) is strictly positive;
this is stronger than the standing nonnegative nonzero assumption in
Theorem~\ref{thm:green-testing-equivalence}. 
This restriction can be removed by using oriented adapted
edge lengths and the associated forward path distance. Under finiteness of
the forward balls, it gives the corresponding conductance and volume criteria
for arbitrary nonnegative, nonzero potentials; see
Remark~\ref{rem:degenerate-potentials}.
\end{remark}


\begin{corollary}\label{cor:unweighted-liouville}
Let \((V,\mu)\) be an infinite, connected, locally finite weighted graph, and
let \(q>1\).  If, for some \(o\in V\),
\begin{equation}\label{eq:unweighted-volume-divergence}
\sum_{n=1}^{\infty}
\frac{n^{2q-1}}{\mu(B(o,n))^{q-1}}
=\infty,
\end{equation}
then every nonnegative solution of
\begin{equation}\label{sdi}
-\Delta u\ge u^q\qquad\text{in }V
\end{equation}
is identically zero.
\end{corollary}
\begin{remark}
	Corollary~\ref{cor:unweighted-liouville} is obtained from
	Theorem~\ref{thm:weighted-liouville} by taking \(\sigma\equiv1\) and the unit
	edge length \(\rho\equiv1\); the details are given in
	Section~\ref{sec:intrinsic-flow}.
	
Theorem~\ref{thm:weighted-liouville} and
Corollary~\ref{cor:unweighted-liouville} are directly related to \eqref{et1}, as
will be clear from the proofs below.  This is related to the so-called
\(L^q\)-Liouville property studied in \cite{GPS,HS2025}.  A more detailed
discussion is given in Section~\ref{sec-example}.
\end{remark}


The proof of Theorem~\ref{thm:weighted-liouville} begins with a relative
capacity estimate from which the volume condition follows. Fix the base point \(o\), and write
\[
B_r:=B_{d_\rho}(o,r),\qquad
g_R:=g_{B_R}(o,\cdot),\qquad
L_R^\nu:=\sum_{x\in B_R}g_R(x)^q\nu(x).
\]
For \(0<r<R\), let \(\capacity_{B_R}(B_r)\) denote the capacity of \(B_r\)
relative to the killed domain \(B_R\).  The central finite-domain estimate in
Section~\ref{sec:intrinsic-flow} is
\begin{equation}\label{eq:intro-relative-capacity-estimate}
L_R^\nu
\ge
c_q\int_0^R r\,\capacity_{B_R}(B_r)^{1-q}\,dr,
\qquad R>0.
\end{equation}
It follows from a flow decomposition of the Green current, the pathwise Hardy
estimate, and a capacity bound. The relative-capacity estimate implies the volume criterion and may retain information about geometric bottlenecks that ball volume does not record.

Theorem~\ref{thm:weighted-liouville} and
Corollary~\ref{cor:unweighted-liouville} give
nonexistence without geometric regularity assumptions.  The complementary
existence theory requires more precise information about Green functions, and
this is where Green testing is effective.

Under the standard assumptions
\textnormal{(VD)}, \textnormal{(PI)}, and \textnormal{(P$_0$)}, recalled in Section~\ref{sec2}, Green-function estimates \cite{Delmotte} allow us to transform the pointwise
Green test \eqref{e1} into explicit volume and potential growth conditions.

\begin{theorem}
	\label{thm:vd-pi-p0-potential-criterion}
	Let \(1<q<\infty\), let \(0\not\equiv\sigma\in\ell^+(V)\), and set
	\(\nu=\sigma\mu\).  Assume that \((V,\mu)\) satisfies \textnormal{(VD)},
	\textnormal{(PI)}, and \textnormal{(P$_0$)}.  Then \eqref{di-2} admits a
	positive solution if and only if there exists \(o\in V\)
	such that
	\begin{equation}\label{thm_e1}
		\sum_{n=1}^{\infty}
		\left[
		\sum_{m=n}^{\infty}
		\frac{m}{\mu(B(o,m))}
		\right]^{q-1}
		\frac{\nu(B(o,n))}{\mu(B(o,n))}\,n
		<\infty ,
	\end{equation}
	and there exists \(C>0\) such that, for all \(n\ge 1\),
	\begin{equation}\label{thm_e2}
		\sup_{x\in V}
		\left[
		\sum_{m=1}^{\infty}
		m\,\frac{\nu(B(x,m)\cap B(o,n))}{\mu(B(x,m))}
		\right]
		\left[
		\sum_{i=n}^{\infty}
		\frac{i}{\mu(B(o,i))}
		\right]^{q-1}
		\le C.
	\end{equation}
\end{theorem}

The two conditions in Theorem~\ref{thm:vd-pi-p0-potential-criterion} are the
graph analogues of the two potential-growth conditions in the manifold setting \cite{GSV20}.  They correspond to conditions \eqref{et1} and \eqref{et2}, respectively.  When
\(\sigma\equiv1\), the second condition follows from the first, and the
criterion collapses to a single volume condition.

\begin{corollary}\label{cor:unweighted-one-condition}
	Under the assumptions of
	Theorem~\ref{thm:vd-pi-p0-potential-criterion}, \eqref{sdi} admits a positive
	solution if and only if, for some \(o\in V\),
	\begin{equation*}
		\sum_{n=1}^{\infty}
		\frac{n^{2q-1}}{\mu(B(o,n))^{q-1}}
		<\infty.
	\end{equation*}
\end{corollary}

The proof is given at the end of Section~\ref{sec5}.

We close the introduction by emphasizing the distinct roles of the two
methods.  Flow decomposition is the principal method for nonexistence.  It
gives the relative capacity estimate
\eqref{eq:intro-relative-capacity-estimate} and the resulting volume-growth
criteria without requiring \textnormal{(VD)}, \textnormal{(PI)},
\textnormal{(P$_0$)}, or \textnormal{(3G)}.  Green testing is a
separate potential-theoretic method: it proves
Theorems~\ref{thm:green-testing-equivalence} and
\ref{thm:green-level-criterion} and, when the Green function can be estimated,
yields the existence criterion in
Theorem~\ref{thm:vd-pi-p0-potential-criterion}.


 Although flow decomposition is a classical tool in network theory \cite{FFbook}
and has many applications in probability \cite{LPbook}, its use in nonlinear
Liouville problems appears, to the best of our knowledge, to be new.
The method also seems flexible beyond the present semilinear elliptic setting. In the recent preprint \cite{GHHS26}, we use flow decomposition to obtain sharp nonexistence
criteria for quasilinear inequalities on weighted graphs.
In a forthcoming preprint
\cite{GHHS26-parabolic}, we prove sharp integral volume-growth criteria for
semilinear parabolic inequalities/equations on weighted graphs. 
In
another forthcoming preprint \cite{GHHS26-stochastic}, 
we obtain a proof by flow decomposition of the sharp volume-growth criterion for stochastic
completeness of general weighted graphs (cf. \cite{Folz14, Huang14, HKM20}).    
The new proof requires neither passage to metric graphs nor subdivision of edges.
 Together they
indicate that flow decomposition and capacitary estimates may apply broadly
to quasilinear elliptic, semilinear parabolic, and probabilistic problems.  We
also expect that this approach may be useful for the conjectural volume-growth criterion
\eqref{eq:GSV-conjecture} in the manifold setting \cite{GSV20}.

The paper is organized as follows.  Section~\ref{sec2} collects the necessary
preliminaries on weighted graphs. 
Section~\ref{sec3} proves
Theorem~\ref{thm:green-testing-equivalence}, the Green-testing criterion.  In
Section~\ref{sec4} we prove Theorem~\ref{thm:green-level-criterion}.  Section~\ref{sec5} derives the sharp volume and potential
criteria under \textnormal{(VD)}, \textnormal{(PI)}, and
\textnormal{(P$_0$)}.  Section~\ref{sec:intrinsic-flow} develops the
argument based on flow decomposition, proves the relative-capacity estimate and
establishes Theorem \ref{thm:weighted-liouville}.
Section~\ref{sec-example} gives examples illustrating the relations between nonexistence of positive solutions and the \(L^q\)-Liouville property.
 Finally,
Sections~\ref{sec7} and \ref{sec-zk} apply these sharp criteria to \(\mathbb Z^d\) and
to its orthant domains.  These applications make explicit the analogy between boundary geometry in manifolds and boundary geometry in discrete domains.

\section{Preliminaries}\label{sec2}

In this section, we collect preliminaries on weighted graphs, random walks,
Green functions, geometric assumptions, intrinsic metrics, flows, and effective
conductance, and fix notation.  
We include the relevant statements and, where
useful, brief proofs in order to fix conventions and make the later arguments
self-contained; no originality is claimed for these preliminaries.  For
background on random walks and Green functions, we refer to
\cite{MBbook,WWbook}; for flows and effective conductance on weighted graphs,
see \cite{FFbook,LPbook}.

\subsection{Weighted graphs and the normalized Laplacian}

Let \((V,E)\) be an infinite, connected, locally finite graph.  We write
\(x\sim y\) if \(x,y\in V\) are joined by an edge.  We only consider loopless graphs, that is, $x\not\sim x$, for all $x\in V$.  A symmetric edge weight
\(\mu:V\times V\to[0,\infty)\) is fixed, and we write
\[
\mu_{xy}:=\mu(x,y).
\]
We assume that \(\mu_{xy}=\mu_{yx}\) and that \(\mu_{xy}>0\) if and only if
\(x\sim y\).  The vertex measure is
\begin{equation*}
	\mu(x):=\sum_{y\sim x}\mu_{xy}.
\end{equation*}
We shall write \((V,\mu)\) for the weighted graph.

For a subset \(A\subset V\) and a vertex weight \(\omega:V\rightarrow[0,\infty)\), set
\[
\omega(A):=\sum_{x\in A}\omega(x).
\]
In particular,
\[
\mu(A)=\sum_{x\in A}\mu(x).
\]
For \(1\le p<\infty\), define
\[
\ell^p(A,\omega)
:=
\left\{
f:A\to\mathbb R:
\sum_{x\in A}|f(x)|^p\omega(x)<\infty
\right\},
\]
with norm
\[
\|f\|_{\ell^p(A,\omega)}
:=
\left(
\sum_{x\in A}|f(x)|^p\omega(x)
\right)^{1/p}.
\]
When \(A=V\), we simply write \(\ell^p(\omega)\).  We also write
\(\ell(A)\) for the space of all real functions on \(A\), \(\ell_0(A)\) for finitely supported functions, and \(\ell^+(A)\) for the nonnegative
functions.

The transition probabilities associated with \((V,\mu)\) are
\begin{equation}\label{eq:transition-probability}
	P(x,y):=
	\begin{cases}
		\dfrac{\mu_{xy}}{\mu(x)}, & x\sim y,\\[1ex]
		0, & x\not\sim y.
	\end{cases}
\end{equation}
The corresponding Markov operator is
\[
Pf(x):=\sum_{y\in V}P(x,y)f(y).
\]
The normalized graph Laplacian is
\begin{equation*}
	\Delta f(x)
	:=
	Pf(x)-f(x)
	=
	\frac1{\mu(x)}
	\sum_{y\sim x}\mu_{xy}\bigl(f(y)-f(x)\bigr).
\end{equation*}
Thus \(-\Delta=I-P\) is the nonnegative Laplacian.  The reversibility identity
\[
\mu(x)P(x,y)=\mu(y)P(y,x)=\mu_{xy}
\]
will be used repeatedly.

\subsection{Random walks and killed kernels}

Let \(\{X_n\}_{n\ge0}\) be the Markov chain/random walk with transition probabilities
\eqref{eq:transition-probability}.  Its \(n\)-step transition probabilities are
\[
P_n(x,y):=\mathbb P_x[X_n=y],
\]
and the corresponding discrete heat kernel is
\[
p_n(x,y):=\frac{P_n(x,y)}{\mu(y)}.
\]
By reversibility,
\[
p_n(x,y)=p_n(y,x),\qquad x,y\in V.
\]

Let \(\Omega\subset V\).  Its outer vertex boundary is
\[
\partial\Omega
:=
\{y\in V\setminus\Omega:\text{ there exists }x\in\Omega
\text{ such that }x\sim y\}.
\]
Functions on \(\Omega\) are always extended by zero to \(V\setminus\Omega\)
when the full graph Laplacian is applied.

Let
\[
\tau_\Omega:=\inf\{n\ge0:X_n\notin\Omega\}
\]
be the first exit time from \(\Omega\).  The killed transition probabilities are
\[
P_n^\Omega(x,y)
:=
\mathbb P_x[X_n=y,\ n<\tau_\Omega],
\qquad x,y\in V.
\]
We put \(P_n^\Omega(x,y)=0\) whenever \(x\notin\Omega\) or \(y\notin\Omega\).
The killed heat kernel is
\[
p_n^\Omega(x,y):=\frac{P_n^\Omega(x,y)}{\mu(y)}.
\]
Then
\[
p_n^\Omega(x,y)=p_n^\Omega(y,x),
\qquad x,y\in V.
\]

The killed Markov operator on \(\Omega\) is
\[
P^\Omega f(x)
:=
\sum_{y\in\Omega}P(x,y)f(y),
\qquad x\in\Omega,
\]
and the Dirichlet Laplacian on \(\Omega\) is
\begin{equation*}
	\Delta_\Omega f(x)
	:=
	P^\Omega f(x)-f(x)
	=
	\sum_{y\in\Omega}\frac{\mu_{xy}}{\mu(x)}f(y)-f(x),
	\qquad x\in\Omega.
\end{equation*}
If \(f\) is extended by zero outside \(\Omega\), then
\[
\Delta_\Omega f(x)=\Delta f(x),
\qquad x\in\Omega.
\]

\subsection{Green functions and Green operators}
Formally, we can define the Green function as
\begin{equation*}
	g(x,y):=\sum_{n=0}^{\infty}p_n(x,y),
	\qquad x,y\in V,
\end{equation*}
which may be identically $+\infty$. It is well known that the finiteness of the Green function, the existence of nonconstant nonnegative superharmonic functions, and the transience of the Markov chain \(\{X_n\}_{n\ge0}\) are equivalent (\cite{MBbook}). Such weighted graphs are also called non-parabolic.

Assuming non-parabolicity, the Green function is the fundamental solution to the Laplacian in the following sense
\[-\Delta_x g(x,y)=\frac1{\mu(y)}\mathbf 1_{\{y\}}(x),
\qquad x,y\in V.\]

Similarly, the Dirichlet Green function in \(\Omega\) is
\[
	g_\Omega(x,y):=\sum_{n=0}^{\infty}p_n^\Omega(x,y),
	\qquad x,y\in\Omega.
\]
It is convenient to view $g(x,y)$ as $g_V(x,y)$.

The Green function is symmetric:
\[
g_\Omega(x,y)=g_\Omega(y,x).
\]
Moreover, whenever \(g_\Omega(\cdot,y)\) is finite,
\begin{equation}\label{eq:green-fundamental-solution}
	-\Delta_x g_\Omega(x,y)
	=
	\frac{1}{\mu(y)}\mathbf 1_{\{y\}}(x),
	\qquad x,y\in\Omega.
\end{equation}

The Green operator on \(\Omega\) is
\[
	G_\Omega f(x)
	:=
	\sum_{y\in\Omega}g_\Omega(x,y)f(y)\mu(y),
	\qquad x\in\Omega.
\]
Thus, formally, for \(f\in \ell^{+}(\Omega)\) such that \(G_\Omega f\) is well defined,
\[
-\Delta_\Omega G_\Omega f=f.
\]
If \(\Omega=V\), we write \(G:=G_V\).

The following standard fact will be used without further comment (see \cite[Theorem 1.31]{MBbook}).

\begin{proposition}\label{gu-lemma}
	Assume that \(\Omega\subset V\) is connected.  If either \(\Omega\neq V\)
	or \((V,\mu)\) is non-parabolic, then
	\[
	0<g_\Omega(x,y)<\infty,
	\qquad x,y\in\Omega.
	\]
\end{proposition}

\subsection{The \texorpdfstring{$(3G)$}{(3G)} condition}\label{subsection:3G}

We say that the Green function $g(\cdot,\cdot)$ satisfies the \textnormal{$(3G)$} inequality
if there exists $\kappa\ge 1$ such that
\begin{equation}\label{3g-pre}
	\frac{1}{g(x,y)}\le \kappa\left(\frac{1}{g(x,z)}+\frac{1}{g(z,y)}\right),
	\qquad x,y,z\in V.
\end{equation}

In the literature, the following equivalent form is also frequently used:
\[
\min\{g(x,z),g(y,z)\}\le C g(x,y), \qquad x,y,z\in V.
\]
It is equivalent to \eqref{3g-pre} up to changing the constant.

We also record a simple comparison between Green functions with different
poles.  This is useful when one states pole-dependent criteria such as
\eqref{e1}.

\begin{lemma}\label{lem_pole}
	Let \(\Omega\subset V\) be connected, and assume that \(g_\Omega\) is finite.
	For any \(o,o'\in\Omega\), there exists a constant \(C=C(o,o',\Omega)>1\)
	such that
	\[
	C^{-1}g_\Omega(o,x)
	\le
	g_\Omega(o',x)
	\le
	Cg_\Omega(o,x),
	\qquad x\in\Omega.
	\]
\end{lemma}

\begin{proof}
	Choose a path in \(\Omega\) from \(o'\) to \(o\).  The killed walk follows
	this path with a positive probability \(a>0\), and hence, by the strong
	Markov property,
	\[
	g_\Omega(o',x)\ge a\,g_\Omega(o,x),
	\qquad x\in\Omega.
	\]
	Interchanging \(o\) and \(o'\) gives the reverse comparison.
\end{proof}

\subsection{Graph distance and geometric assumptions}\label{subsection:conditions}

For any two vertices \(x\ne y \in V\), let \(d(x, y)\) denote the minimal number of edges along a path connecting \(x\) and \(y\). Set $d(x,x)=0$ for any $x\in V$.
This defines a natural metric on \(V\), which is called the graph distance.

For \(o\in V\) and
\(r\ge0\), set
\[
B(o,r):=\{x\in V:d(o,x)\le r\}.
\]

We shall use the following three assumptions.

The graph \((V,\mu)\) satisfies the volume doubling condition if there exists
\(C_D>0\) such that
\begin{equation*}\tag{VD}\label{VD}
	\mu(B(x,2r))\le C_D\mu(B(x,r)),
	\qquad x\in V,\ r>0.
\end{equation*}

The graph \((V,\mu)\) satisfies the scale-invariant Poincar\'e inequality if
there exists \(C_P>0\) such that, for all \(x_0\in V\), \(r>0\), and
\(f\in\ell(V)\),
\begin{equation*}\tag{PI}\label{PI}
	\sum_{x\in B(x_0,r)} |f(x)-f_B|^2\mu(x)
	\le
	C_P r^2
	\sum_{x,y\in B(x_0,2r)}
	\mu_{xy}\bigl(f(y)-f(x)\bigr)^2,
\end{equation*}
where
\[
f_B
:=
\frac{1}{\mu(B(x_0,r))}
\sum_{x\in B(x_0,r)}f(x)\mu(x).
\]

The graph \((V,\mu)\) satisfies \textnormal{(P$_0$)} if there exists
\(p_0\ge1\) such that
\begin{equation*}\tag{P$_0$}\label{p0}
	\frac{\mu_{xy}}{\mu(x)}\ge \frac1{p_0},
	\qquad x\sim y.
\end{equation*}

Delmotte \cite{Delmotte} studied heat kernel and Green-function estimates under conditions \textnormal{(VD)}, \textnormal{(PI)}, and \textnormal{(P$_0$)}. See Lemma \ref{lem:Green-estimate} below for the relevant result for Green functions.

\subsection{Intrinsic metrics adapted to a potential}\label{subsection:intrinsic-metrics}

This subsection introduces the intrinsic metric used in Theorem \ref{thm:weighted-liouville}.

Let \(0<\sigma\in\ell^+(V)\), and set
\[
\nu(x):=\sigma(x)\mu(x).
\]
A positive symmetric edge-length function \(\rho:E\to(0,\infty)\) is called
\(\nu\)-adapted if
\begin{equation}\label{eq:adapted-weight}
	\sum_{y\sim x}\mu_{xy}\rho(x,y)^2\le \nu(x),
	\qquad x\in V.
\end{equation}
\begin{remark}
Such an edge-length function always exists.  For example, one may take
\[
\rho(x,y)=\sqrt{\min\{\sigma(x),\sigma(y)\}}.
\]
\end{remark}

We then define the associated path metric.
Set \(d_\rho(x,x)=0\). For \(x\ne y\), define
\[
d_\rho(x,y)
:=
\inf\left\{
\sum_{i=1}^n\rho(x_{i-1},x_i):
x=x_0\sim x_1\sim\cdots\sim x_n=y
\right\}.
\]
This construction is the so-called intrinsic path metric in the literature, and can be viewed as an analogue of  the geodesic distance function on a Riemannian manifold. In particular, \eqref{eq:adapted-weight} is analogous to
\[\lvert\nabla d(x, \cdot)\rvert\le 1.\]

For \(o\in V\) and \(r>0\), we write
\[
B_{d_{\rho}}(o,r):=\{x\in V:{d_{\rho}}(o,x)\le r\},
\qquad
\nu(B_{d_{\rho}}(o,r))
:=
\sum_{x\in B_{d_{\rho}}(o,r)}\nu(x).
\]
\begin{remark}\label{rem:base-point-independence}
Let \(o,o'\in V\) and set \(D=d_\rho(o,o')\). Then, for every \(r>D\),
\[
B_{d_\rho}(o,r-D)
\subset B_{d_\rho}(o',r)
\subset B_{d_\rho}(o,r+D).
\]
Consequently, the divergence condition
\[
\int_1^\infty
\frac{r^{2q-1}}
{\nu(B_{d_\rho}(o,r))^{q-1}}\,dr
=\infty
\]
is independent of the choice of the base point. Indeed, the two inclusions,
together with a change of variables, show that divergence for one base point
implies divergence for every other base point.
\end{remark}
The following result serves as a discrete Hopf-Rinow type theorem (\cite{HKMW13}, \cite{Keller15}, \cite[Theorem 11.16]{KLWBook}).
\begin{proposition}\label{prop:complete-finite-balls}
Let $(V,\mu)$ be a weighted graph. Let $0<\sigma\in \ell^{+}(V)$, and set $\nu =\sigma\mu$. Suppose that $d_{\rho}$ is the intrinsic path metric associated with a $\nu$-adapted edge-length function $\rho$. Then $(V,d_{\rho})$ is complete if and only if every distance ball $B_{d_{\rho}}(o,r)$ is finite.
\end{proposition}

In the special case that $\sigma \equiv 1$, the unit edge length $\mathbf{1}_E$ is $\mu$-adapted, and the graph metric $d$ is itself an intrinsic path metric. Note that by local finiteness, $(V,d)$ is always complete.

The notion of intrinsic metrics for general Dirichlet forms was introduced in \cite{FLW}, which was motivated by considerations of L\'{e}vy processes (cf. \cite{MU}). For applications of intrinsic metrics to stochastic completeness of weighted graphs, we refer to \cite{Folz14, Huang14}.

\subsection{Flows and effective conductance on finite weighted graphs}
\label{subsection:finite-flows}

We recall the flow-theoretic notions used in
Section~\ref{sec:intrinsic-flow}.  Let \((W,\widetilde\mu)\) be a finite
connected weighted graph.  Thus \(W\) is a finite vertex set,
\(\widetilde\mu_{xy}=\widetilde\mu_{yx}\ge0\), and we write
\(x\sim_W y\) when \(\widetilde\mu_{xy}>0\).  The edge weight
\(\widetilde\mu_{xy}\) is viewed as the conductance of \(\{x,y\}\).  In the
applications below, such a graph is obtained from the ambient weighted graph \((V,\mu)\) by
restriction, wiring, and adjoining finitely many boundary vertices.  Let
\[
\overrightarrow E_W
:=
\{(x,y)\in W\times W:x\sim_W y\}
\]
be the set containing both orientations of every edge.  If \(e=(x,y)\), then
\[
\bar e=(y,x),
\qquad
\tail(e)=x,
\qquad
\head(e)=y,
\qquad
\widetilde\mu_e=\widetilde\mu_{xy}.
\]

A flow on \((W,\widetilde\mu)\) is an antisymmetric function
\[
\Phi:\overrightarrow E_W\longrightarrow\mathbb R,
\qquad
\Phi(\bar e)=-\Phi(e).
\]
Its divergence at \(x\in W\) is
\begin{equation*}\label{eq:flow-divergence}
\Div\Phi(x)
:=
\sum_{y\sim_W x}\Phi(x,y).
\end{equation*}
Antisymmetry gives
\[
\sum_{x\in W}\Div\Phi(x)=0.
\]
Let \(A,K\subset W\) be nonempty and disjoint.  We say that \(\Phi\) is a
flow from \(A\) to \(K\) if
\[
\Div\Phi=0\quad\text{on }W\setminus(A\cup K),
\qquad
\Div\Phi\ge0\quad\text{on }A,
\qquad
\Div\Phi\le0\quad\text{on }K.
\]
Its strength is
\begin{equation*}\label{eq:flow-strength}
I(\Phi)
:=
\sum_{x\in A}\Div\Phi(x)
=
-\sum_{x\in K}\Div\Phi(x).
\end{equation*}
A flow of strength one is called a unit flow.  The directed support of
\(\Phi\) is the set of oriented edges on which \(\Phi(e)>0\).  The flow is
called acyclic if its directed support contains no directed cycle.

The energy of a flow is
\begin{equation*}\label{eq:flow-energy}
\mathcal D_W(\Phi)
:=
\frac12\sum_{e\in\overrightarrow E_W}
\frac{\Phi(e)^2}{\widetilde\mu_e}.
\end{equation*}
For a function \(f:W\to\mathbb R\), put
\[
\nabla_e f=f(\tail(e))-f(\head(e)),
\qquad
\mathcal E_W(f)
:=
\frac12\sum_{e\in\overrightarrow E_W}
\widetilde\mu_e|\nabla_e f|^2.
\]
The current generated by \(f\) is
\begin{equation*}\label{eq:current-from-potential}
i_f(e):=\widetilde\mu_e\nabla_e f.
\end{equation*}
Then
\[
\mathcal D_W(i_f)=\mathcal E_W(f).
\]
If
\[
\widetilde\mu(x):=\sum_{y\sim_W x}\widetilde\mu_{xy}
\]
and \(\Delta_W\) denotes the normalized Laplacian on
\((W,\widetilde\mu)\), then
\begin{equation*}\label{eq:current-divergence-laplacian}
\Div i_f(x)
=
\sum_{y\sim_W x}\widetilde\mu_{xy}\bigl(f(x)-f(y)\bigr)
=
-\widetilde\mu(x)\Delta_W f(x).
\end{equation*}
Thus current conservation is equivalent to harmonicity away from the source
and sink sets.  Moreover, for every flow \(\Phi\) and every function \(f\),
summation by parts gives
\begin{equation}\label{eq:flow-integration-by-parts}
\sum_{x\in W}f(x)\Div\Phi(x)
=
\frac12\sum_{e\in\overrightarrow E_W}\Phi(e)\nabla_e f.
\end{equation}
If the nonzero edges of \(i_f\) are oriented from the larger value of \(f\)
to the smaller one, then \(f\) decreases strictly along every directed edge.
Consequently, the resulting directed current is acyclic.

To wire a set of vertices in a finite weighted graph means to identify them
to a single vertex, delete the resulting loops, and replace parallel edges by
one edge whose weight is the sum of their weights.  For disjoint nonempty sets
\(A,K\subset W\), define their effective conductance by
\begin{equation}\label{eq:effective-conductance-definition}
C_{\mathrm{eff}}^W(A,K)
:=
\inf\left\{
\mathcal E_W(f):
 f|_A=1,\ f|_K=0
\right\}.
\end{equation}
Equivalently, one may wire \(A\) and \(K\) separately before taking the
infimum.  We omit the superscript \(W\) when the finite weighted graph is
clear.  The effective resistance is
\[
R_{\mathrm{eff}}^W(A,K):=C_{\mathrm{eff}}^W(A,K)^{-1}.
\]
The minimizer \(h\) in \eqref{eq:effective-conductance-definition} is the
equilibrium potential: it equals one on \(A\), zero on \(K\), takes values
in \([0,1]\), and is harmonic on \(W\setminus(A\cup K)\).  Its current
\(i_h\) is a flow from \(A\) to \(K\), and its strength and energy both
equal \(C_{\mathrm{eff}}^W(A,K)\).

\begin{proposition}
\label{prop:dirichlet-thomson}
Let \(A,K\subset W\) be nonempty and disjoint. If \(\Phi\) is any
flow from \(A\) to \(K\) of strength \(I\), then
\begin{equation}\label{eq:thomson-principle}
\mathcal D_W(\Phi)
\ge
\frac{I^2}{C_{\mathrm{eff}}^W(A,K)}.
\end{equation}
In particular,
\[
R_{\mathrm{eff}}^W(A,K)
=
\min\left\{
\mathcal D_W(\Phi):
\Phi\text{ is a unit flow from }A\text{ to }K
\right\}.
\]
\end{proposition}

\begin{proof}
Let \(h\) be the equilibrium potential.  The Euler--Lagrange equation for
\eqref{eq:effective-conductance-definition} gives
\(\Div i_h=0\) on \(W\setminus(A\cup K)\).  Since \(0\le h\le1\), the
current leaves \(A\) and enters \(K\).  By
\eqref{eq:flow-integration-by-parts}, its strength is
\[
I(i_h)
=
\sum_{x\in W}h(x)\Div i_h(x)
=
\mathcal E_W(h)
=
C_{\mathrm{eff}}^W(A,K).
\]
For an arbitrary flow \(\Phi\) of strength \(I\), the same identity and
the Cauchy--Schwarz inequality give
\[
I
=
\frac12\sum_{e\in\overrightarrow E_W}\Phi(e)\nabla_e h
\le
\mathcal D_W(\Phi)^{1/2}
C_{\mathrm{eff}}^W(A,K)^{1/2},
\]
which proves \eqref{eq:thomson-principle}.  Finally,
\[
\Phi_h:=\frac{i_h}{C_{\mathrm{eff}}^W(A,K)}
\]
is a unit flow and satisfies
\[
\mathcal D_W(\Phi_h)=R_{\mathrm{eff}}^W(A,K),
\]
so equality is attained.
\end{proof}

For these principles with the same energy normalization, see
\cite[Chapter 2, Section 4]{LPbook}.

\begin{lemma}
\label{lem:finite-acyclic-flow-decomposition}
Let \(\Phi\) be an acyclic flow from \(A\) to \(K\) of strength \(I>0\) on
a finite weighted graph.  After orienting every edge in the direction in which
\(\Phi\) is positive, there exist finitely many directed paths
\(\gamma_1,\ldots,\gamma_N\), each starting in \(A\) and ending in \(K\),
and positive numbers \(a_1,\ldots,a_N\) such that
\begin{equation}\label{eq:finite-path-decomposition}
\Phi(e)
=
\sum_{j:e\in\gamma_j}a_j,
\qquad
\sum_{j=1}^N a_j=I,
\end{equation}
for every edge \(e\) in the directed support.  Consequently, if \(\Phi\) is
a unit flow, the weights \(a_j\) define a probability measure on directed
\(A\)-to-\(K\) paths satisfying
\begin{equation}\label{eq:path-edge-marginal}
\mathbb P(\gamma\text{ uses }e)=\Phi(e).
\end{equation}
For a flow of general strength \(I\), the corresponding probability is
\(\Phi(e)/I\).
\end{lemma}

\begin{proof} We can apply a standard greedy algorithm.
One may adjoin a super-source connected to each vertex of \(A\) with edge
flow \(\Div\Phi(x)\), and a super-sink connected from each vertex of \(K\)
with edge flow \(-\Div\Phi(x)\).  It is therefore enough to consider a single
source and a single sink.  Starting at the source, follow edges carrying
positive flow.  Flow conservation guarantees a positive outgoing edge at
every intermediate vertex.  Since the directed support is finite and
acyclic, the path must reach the sink.  Subtract from every edge of this path
the minimum flow carried by its edges.  At least one edge then disappears
from the directed support.  Repeating the procedure terminates after finitely
many steps and gives \eqref{eq:finite-path-decomposition}.  The sum of the
path weights is the original strength.
\end{proof}

This decomposition is generally highly non-unique: different choices of paths
in the preceding procedure can produce different collections of paths and
different probability laws.  In Section~\ref{sec:intrinsic-flow}, we fix an
arbitrary decomposition and use only that it is supported on directed
source-to-sink paths and satisfies the edge-marginal identity
\eqref{eq:path-edge-marginal}.  If
\(\gamma=(x_0,e_0,x_1,\ldots,e_{m-1},x_m)\) and
\(0\le\tau<\varsigma\le m\), the segment of \(\gamma\) from \(\tau\) to
\(\varsigma\) means
\[
(x_\tau,e_\tau,x_{\tau+1},\ldots,e_{\varsigma-1},x_\varsigma).
\]
For the standard finite flow decomposition, see
\cite[Chapter I, Section 2, Theorem 2.2]{FFbook}; for the random-path
formulation of a finite acyclic unit flow, see
\cite[Chapter 3, Section 1, Proposition 3.2]{LPbook}.

\subsection{Notation}

Throughout the paper, \(C,c,C_1,c_1,\ldots\) denote positive constants whose
values may change from line to line.  Unless otherwise stated, these constants
may depend on fixed parameters, such as \(q\), but not on the variables under
consideration.  We write
\(
f\lesssim g
\)
if \(f\le Cg\).  We write \(f\gtrsim g\) if \(g\lesssim f\), and
\(f\asymp g\) if both inequalities hold.

\section{Proof of Theorem \ref{thm:green-testing-equivalence}}
\label{sec3}

The proof of Theorem \ref{thm:green-testing-equivalence} uses the following auxiliary results.

\begin{lemma}\label{lem_finite_rep}
    Let \(U\subset V\) be a finite subset, and let \(u \in \ell_0(V)\).
    Define
    \[
        h=u-G_U(-\Delta u).
    \]
    Then \(h\) is the unique solution of
    \begin{equation*}
    	\begin{cases}
    		-\Delta h = 0, \quad\text{in } U,\\
    		h|_{\partial U}=u.
    	\end{cases}
    \end{equation*}
    In particular, if \(u=0\) on \(\partial U\), then
    \[
        u(x)=G_U(-\Delta u)(x)\qquad \text{for all }x\in U.
    \]
\end{lemma}

\begin{proof}
	Since \(G_U(-\Delta u)=0\) on \(\partial U\), we have
	\(h|_{\partial U}=u\).
	
    Fixing \(y\in U\), we have
    \[
        -\Delta_U g_U(x,y)=\frac{1}{\mu(y)}\mathbf{1}_{\{y\}}(x)\qquad \text{for }x\in U.
    \]
    Using the fact that \(G_U(-\Delta u)=0\) on \(V\setminus U\) and under the zero-extension convention, we obtain
    \[
        -\Delta G_U(-\Delta u)(x)=-\Delta_U G_U(-\Delta u)(x)=-\Delta u(x)
        \qquad \text{for }x\in U.
    \]
    Hence,
    \[
        -\Delta h=0\qquad \text{in }U.
    \]

    If \(u=0\) on \(\partial U\), then \(h=0\) on \(\partial U\) as well. Since \(U\) is finite,  \(h\equiv 0\) in \(U\). Therefore
    \[
        u(x)=G_U(-\Delta u)(x)\qquad \text{for all }x\in U.
    \]
\end{proof}
\begin{lemma}\label{lem_rev}
	Let \(\Omega\subset V\) be connected and assume that \(g_\Omega\) is finite.
	Let \(u\in\ell^+(\Omega)\), extended by zero to \(V\setminus\Omega\), such that
	\[
	-\Delta u=-\Delta_\Omega u\ge0\quad\text{in }\Omega.
	\]
	Then
	\[
	G_\Omega(-\Delta u)(x)\le u(x),\qquad x\in\Omega.
	\]
\end{lemma}

\begin{proof}
	Recalling that \(-\Delta_{\Omega} = I_{\Omega} - P^{\Omega}\), for any integer $N>1$, we compute
	\begin{align*}
		\sum_{n=0}^{N} P^{\Omega}_n(-\Delta_{\Omega}u) (x)=u(x)-P^{\Omega}_{N+1} u(x)\le u(x).
	\end{align*}
	We can pass to the limit and obtain
	 \[
	 G_\Omega(-\Delta_\Omega u)(x)=\sum_{n=0}^{\infty} P^{\Omega}_n (-\Delta_{\Omega} u)(x)\le u(x)
	 \]
	 via monotone convergence, since \(-\Delta_\Omega u\ge0\). For any $v\in \ell^+(\Omega)$, Fubini's theorem gives
	\begin{align*}
	 	G_{\Omega} v(x)=\sum_{y\in \Omega}g_{\Omega}(x,y)v(y)\mu(y)=\sum_{y\in \Omega}\sum_{n=0}^{\infty}P_{n}^{\Omega}(x,y)v(y)=\sum_{n=0}^{\infty} P^{\Omega}_n v(x).
	 \end{align*}
     Under the zero-extension convention,
     \(\Delta u=\Delta_\Omega u\) on \(\Omega\).  Hence
	\begin{equation*}
			G_{\Omega}(-\Delta u)(x)
			= G_{\Omega}(-\Delta_{\Omega} u)(x)\le u(x), \qquad x\in \Omega.
	\end{equation*}
\end{proof}

\begin{lemma}\label{lem1}
	Assume that \(\Omega\subset V\) is connected and that \(g_\Omega\) is finite.
Let \(0\not\equiv \sigma\in\ell^+(\Omega)\) and
\(0\not\equiv f\in\ell^+(\Omega)\). Let
	\(u\in\ell^+(\Omega)\), extended by zero outside \(\Omega\), satisfy
	\[
	u|_{\partial\Omega}=0,\qquad
	-\Delta u\ge \sigma u^q+f \quad \text{in }\Omega.
	\]
	Set \(h=G_\Omega f\), and assume \(h(x)<\infty\) for all \(x\in\Omega\).
    \begin{enumerate}
        \item[(i)] If \(1<q<\infty\), then necessarily
        \begin{equation}\label{eq-necessity}
 G_{\Omega}(\sigma h^q)(x)
<
\frac{h(x)}{q-1}
\quad \text{for all }x\in\Omega.
        \end{equation}

        Moreover,
        \[
            u(x)\ge h(x)\left[
            1-(q-1)\frac{G_{\Omega}(\sigma h^q)(x)}{h(x)}
            \right]^{-\frac{1}{q-1}}
            \quad \text{for all }x\in\Omega.
        \]

        \item[(ii)] If \(0<q<1\), then
        \[
            G_{\Omega}(\sigma h^q)(x)
            \le
            \frac{h(x)}{1-q}
            \left[
                \left(\frac{u(x)}{h(x)}\right)^{1-q}-1
            \right]
            \quad \text{for all }x\in\Omega.
        \]
        Equivalently,
        \[
            u(x)\ge h(x)\left[
            1+(1-q)\frac{G_{\Omega}(\sigma h^q)(x)}{h(x)}
            \right]^{\frac{1}{1-q}}
            \quad \text{for all }x\in\Omega.
        \]
    \end{enumerate}
\end{lemma}
\begin{remark}
	We include the case $0<q<1$  for completeness, although only
	the case \(q>1\) is used in this paper.
	\end{remark}
\begin{proof}
	Note that $u>0$ in $\Omega$ since $f\not \equiv 0$.
Since $0\not\equiv \sigma\in \ell^+(\Omega) $ and $\Omega$ is connected, it follows that
	\(G_\Omega(\sigma u^q)>0\).
	
    If \(\Omega\) is finite, then Lemma \ref{lem_finite_rep} gives
    \[
        u(x)=G_{\Omega}(-\Delta u)(x)\ge G_{\Omega}(\sigma u^q+f)(x)>G_{\Omega}f(x)=h(x)
        \quad \text{for all }x\in\Omega.
    \]
    If \(\Omega\) is infinite, then Lemma \ref{lem_rev} yields
    \[
        u(x)\ge G_{\Omega}(\sigma u^q+f)(x)>G_{\Omega}f(x)=h(x)
        \quad \text{for all }x\in\Omega.
    \]
    Since \(\Omega\) is connected and \(f\not\equiv 0\), we have
    \(h=G_\Omega f>0\) in \(\Omega\). Hence \(0<h\le u<\infty\).

    Define
    \[
        \phi(s)=\bigl(1+(1-q)s\bigr)^{\frac1{1-q}},
    \]
    where \(s\ge 0\) if \(0<q<1\), and \(0\le s<\frac1{q-1}\) if \(1<q<\infty\).
    Then
    \[
        \phi(0)=1,\qquad \phi'(s)=\phi(s)^q,\qquad \phi''(s)=q\,\phi(s)^{2q-1}\ge 0.
    \]
    Set
    \[
        v(x)=\phi^{-1}\left(\frac{u(x)}{h(x)}\right),\qquad x\in\Omega.
    \]
    Extend \(v\) by zero to \(V\setminus\Omega\).
    Since \(u/h> 1\), we have \(v>0\) in \(\Omega\), and in the case \(q>1\),
    \[
        0<v(x)<\frac1{q-1}.
    \]

    For brevity, write
    \[
        \nabla_{xy}f=f(y)-f(x).
    \]
    We compute
    \begin{align*}
        \Delta\bigl(h\phi(v)\bigr)(x)
        &=\sum_{y\sim x}P(x,y)h(y)\nabla_{xy}\phi(v)+\phi(v(x))\Delta h(x),\\
        \Delta(hv)(x)
        &=\sum_{y\sim x}P(x,y)h(y)\nabla_{xy}v+v(x)\Delta h(x).
    \end{align*}
    By Taylor's formula, for some \(\xi_{xy}\) between \(v(x)\) and \(v(y)\),
    \[
        \nabla_{xy}\phi(v)
        =\phi'(v(x))\nabla_{xy}v+\frac{\phi''(\xi_{xy})}2(\nabla_{xy}v)^2.
    \]
    Hence
    \begin{align*}
        -\Delta(hv)(x)
        &=\frac{-\Delta u(x)}{\phi'(v(x))}
        +\frac{1}{2\phi'(v(x))}
        \sum_{y\sim x}\phi''(\xi_{xy})P(x,y)h(y)(\nabla_{xy}v)^2\\
        &\quad
        +\left(\frac{\phi(v(x))}{\phi'(v(x))}-v(x)\right)\Delta h(x),
    \end{align*}
    where we used \(u=h\phi(v)\).

    Since \(h=G_{\Omega}f\), we have
    \[
        -\Delta h=f\quad \text{in }\Omega.
    \]
    Therefore,
    \begin{align*}
        -\Delta(hv)(x)
        &\ge \frac{\sigma(x)h(x)^q\phi(v(x))^q}{\phi'(v(x))}
        +\frac{1}{2\phi'(v(x))}
        \sum_{y\sim x}\phi''(\xi_{xy})P(x,y)h(y)(\nabla_{xy}v)^2\\
        &\quad
        +\left(
        \frac{\phi(v(x))-1}{\phi'(v(x))}-v(x)
        \right)\Delta h(x).
    \end{align*}

    Since \(\phi\) is convex and \(\phi(0)=1\), for every \(s>0\),
    \[
        \frac{\phi(s)-1}{s}\le \phi'(s),
    \]
    that is,
    \[
        \frac{\phi(s)-1}{\phi'(s)}-s\le 0.
    \]
    Also,
    \[
        \phi'(s)=\phi(s)^q,\qquad \phi''(s)\ge 0.
    \]
    Since \(\Delta h=-f\le 0\), it follows that
    \[
        -\Delta(hv)\ge \sigma h^q\quad \text{in }\Omega.
    \]

    If \(\Omega\) is finite, then Lemma \ref{lem_finite_rep} gives
    \[
        hv(x)=G_{\Omega}(-\Delta(hv))(x)\ge G_{\Omega}(\sigma h^q)(x)
        \quad \text{for all }x\in\Omega.
    \]
    If \(\Omega\) is infinite, then applying Lemma \ref{lem_rev} again, we obtain
    \begin{equation}\label{eq-necessary}
    	 hv(x)\ge G_{\Omega}(\sigma h^q)(x)\quad \text{for all }x\in\Omega.
    \end{equation}

    If \(1<q<\infty\), since $0\le v<\frac{1}{q-1}$,
    together with \eqref{eq-necessary}, this leads to the necessity of \eqref{eq-necessity}. Note that
    \[
        v(x)=\frac{1-\left(\frac{h(x)}{u(x)}\right)^{q-1}}{q-1},
    \]
    which yields the estimate in (i).  If \(0<q<1\), then
    \[
        v(x)=\frac{\left(\frac{u(x)}{h(x)}\right)^{1-q}-1}{1-q},
    \]
    which yields the estimate in (ii). The displayed pointwise bounds are just equivalent reformulations.
\end{proof}

We are now ready to prove Theorem \ref{thm:green-testing-equivalence}.
\begin{proof}[Proof of Theorem \ref{thm:green-testing-equivalence}]
We complete the proof by showing that (I) $\Rightarrow$ (II), (II) $\Rightarrow$ (III), and     (III) $\Rightarrow$ (I).

    (I) $\Rightarrow$ (II).
    Suppose \(u\) is a positive solution to \eqref{di-1}.
     Applying Lemma \ref{lem_rev} to the solution \(u\), we obtain
    \[
        u \ge G_{\Omega}(-\Delta u) \ge G_{\Omega}(\sigma u^q)\quad\text{in } \Omega,
    \]
    so \( u \) is a positive solution to \eqref{ii}.

   \medskip

   \noindent
   \textnormal{(II)} \(\Rightarrow\) \textnormal{(III)}.
   Assume \(u\) is a positive solution to \eqref{ii}.
   Let \(v = \varepsilon u\) with \(\varepsilon \in (0,1)\).
   Then, by \eqref{ii}, we have
   \begin{align*}
   	v
   	&= \varepsilon u
   	\ge \varepsilon\, G_{\Omega}(\sigma u^q)
   	= \varepsilon^{1 - q} G_{\Omega}(\sigma v^q) \\
   	&= G_{\Omega}(\sigma v^q)
   	+(\varepsilon^{1 - q}-1)G_{\Omega}(\sigma v^q).
   \end{align*}

   Since \(0\not\equiv\sigma\), choose \(p\in\Omega\) such that
   \(\sigma(p)>0\). Define
   \[
   \omega(x):=G_{\Omega}(\sigma v^q)(x),
   \]
   and
   \[
   h(x):=(\varepsilon^{1-q}-1)v(p)^q\nu(p)g_\Omega(p,x).
   \]
   Since \(\Omega\) is connected, \(g_\Omega(p,x)>0\) for all
   \(x\in\Omega\), and hence \(h>0\) in \(\Omega\). Moreover,
   \[
   \omega(x)
   =\sum_{y\in\Omega}g_\Omega(x,y)\sigma(y)v(y)^q\mu(y)
   \ge v(p)^q\nu(p)g_\Omega(p,x).
   \]
   Therefore
   \[
   v\ge \omega+(\varepsilon^{1-q}-1)\omega
   \ge \omega+h>0
   \qquad \text{in }\Omega.
   \]

   Using the property
   \(-\Delta G_\Omega(\sigma v^q)=\sigma v^q\) in \(\Omega\), we have
   \[
   -\Delta(\omega+h)
   =-\Delta G_\Omega(\sigma v^q)+f
   =\sigma v^q+f
   \ge \sigma(\omega+h)^q+f,
   \]
   where
   \[
   f(x):=(\varepsilon^{1-q}-1)v(p)^q\sigma(p)\mathbf 1_{\{p\}}(x)
   =-\Delta h(x).
   \]
   Applying Lemma~\ref{lem1} to \(\omega+h\) gives
   \[
   G_\Omega(\sigma h^q)(x)
   <\frac{1}{q-1}h(x),
   \qquad x\in\Omega.
   \]
   Set
   \[
   a:=(\varepsilon^{1-q}-1)v(p)^q\nu(p)>0.
   \]
   Since \(h=a\,g_\Omega(p,\cdot)\), the previous estimate yields
   \[
   G_\Omega\bigl(\sigma g_\Omega(p,\cdot)^q\bigr)(x)
   \le \frac{a^{1-q}}{q-1}g_\Omega(p,x),
   \qquad x\in\Omega.
   \]
   Thus \eqref{e1} holds with \(p\). By Lemma \ref{lem_pole}, \eqref{e1} holds for every \(o\in\Omega\).
   This proves \textnormal{(III)}.

   \medskip

   \noindent
   \textnormal{(III)} \(\Rightarrow\) \textnormal{(I)}.
   Assume that \eqref{e1} holds for some \(o\in\Omega\). Set
   \[
   U(x):=
   C^{\frac{q}{1-q}}
   G_\Omega\bigl(\sigma g_\Omega(o,\cdot)^q\bigr)(x).
   \]
   Since \(0\not\equiv\sigma\), choose \(p\in\Omega\) such that
   \(\sigma(p)>0\). Since \(\Omega\) is connected, we have
   \(g_\Omega(x,p)>0\) and \(g_\Omega(o,p)>0\) for all \(x\in\Omega\).
   Consequently,
   \[
   \begin{aligned}
   	G_\Omega\bigl(\sigma g_\Omega(o,\cdot)^q\bigr)(x)
   	&=\sum_{y\in\Omega}
   	g_\Omega(x,y)\sigma(y)g_\Omega(o,y)^q\mu(y) \\
   	&\ge
   	g_\Omega(x,p)\sigma(p)g_\Omega(o,p)^q\mu(p)
   	>0.
   \end{aligned}
   \]
   Hence \(U(x)>0\) for all \(x\in\Omega\).

   Moreover, by \eqref{e1},
   \[
   U(x)\le C^{\frac{q}{1-q}} C g_\Omega(o,x)
   = C^{\frac{1}{1-q}}g_\Omega(o,x),
   \qquad x\in\Omega.
   \]
   Therefore
   \[
   U(x)^q
   \le C^{\frac{q}{1-q}}g_\Omega(o,x)^q,
   \qquad x\in\Omega.
   \]
   It follows that
   \[
   -\Delta U
   = C^{\frac{q}{1-q}}\sigma g_\Omega(o,\cdot)^q
   \ge \sigma U^q
   \qquad \text{in }\Omega.
   \]
   Therefore, \(U\) is a positive solution to \eqref{di-1}.
   This completes the proof.
\end{proof}

\section{Proof of Theorem \ref{thm:green-level-criterion}}\label{sec4}

We first record the maximum principle that will be used in the proof of
Theorem~\ref{thm:green-level-criterion}.

\begin{lemma}
\label{lem:wmp-graph}
Let \(\Omega\subset V\) be connected and assume that \(g_\Omega\) is finite, and let \(f\in\ell^+(\Omega)\) be supported in a set \(A\subset \Omega\). If \(M\ge0\) and
\[
        G_{\Omega}f(x)\le M,\qquad  x\in A,
\]
then
\[
        G_{\Omega}f(x)\le M,\qquad x\in \Omega.
\]
\end{lemma}

\begin{proof}
If \(f\equiv0\), the conclusion is immediate.  Hence assume that
\(f\not\equiv0\).
We first assume that \(\Omega\) is finite. Let \(u=G_\Omega f\), extended by zero to \(V\setminus\Omega\). Then \(-\Delta u=f\) in \(\Omega\). Suppose, to the contrary, that \(u>M\) somewhere in \(\Omega\), and set
\[
        D:=\{x\in\Omega:u(x)>M\}.
\]
Since \(u\le M\) on \(A\) and \(\operatorname{supp}f\subset A\), we have \(D\subset \Omega\setminus A\). Hence \(u\) is harmonic at every vertex of \(D\).

Choose \(x_0\in D\) where \(u\) attains its maximum on \(D\). Since \(\Delta u(x_0)=0\),
\[
        u(x_0)=\sum_{y\sim x_0}\frac{\mu_{x_0y}}{\mu(x_0)}u(y).
\]
Noting that \(u(y)\le u(x_0)\) for any \(y\sim x_0\) and \(\sum\limits_{y\sim x_0}\mu_{x_0y}=\mu(x_0)\), we obtain \(u(y)=u(x_0)>M\) for any \(y\sim x_0\), thus every neighbor of \(x_0\) belongs to \(D\). Repeating the argument along paths, and using connectedness of \(\Omega\), would force every vertex of \(\Omega\) to lie in \(D\), contradicting \(u\le M\) on \(A\). Therefore \(D=\varnothing\), and \(G_\Omega f\le M\) on \(\Omega\).

If \(\Omega\) is infinite, choose an increasing exhaustion \((U_n)\) of
\(\Omega\) by finite connected sets containing a fixed vertex \(o\in\Omega\).
Put \(f_n=f\mathbf 1_{U_n}\) and \(A_n=A\cap U_n\). Since
\(g_{U_n}\le g_\Omega\),
\[
        G_{U_n}f_n(x)\le G_{\Omega}f(x)\le M,
        \qquad x\in A_n.
\]
By the finite-domain statement,
\[
        G_{U_n}f_n(x)\le M,
        \qquad x\in U_n.
\]
Letting \(n\to\infty\) and using
\(g_{U_n}(x,y)\uparrow g_{\Omega}(x,y)\) gives
\(G_{\Omega}f(x)\le M\) for every \(x\in \Omega\).
\end{proof}

For the proof of the necessary part of Theorem \ref{thm:green-level-criterion},
we shall use the following finite-domain Green-energy estimate.

For a finite connected subset \(\Omega\subset V\) with \(o\in\Omega\), define
\[
	L_\Omega^\nu(o)
	:=
	\sum_{y\in\Omega}g_\Omega(o,y)^q\nu(y).
\]

\begin{lemma}\label{lemma6.1}
	Let \(\Omega\subset V\) be finite and connected, and let \(o\in\Omega\).  If
	\(u\) is a nonnegative solution to \eqref{di-2}, then
	\begin{equation}\label{eq:finite-domain-testing}
		u(o)^q\nu(o)
		\le
		\left(\frac{q}{q-1}\right)^{\frac{q}{q-1}}
		\bigl(L_\Omega^\nu(o)\bigr)^{-\frac{1}{q-1}}.
	\end{equation}
	If \(L_\Omega^\nu(o)=0\), the right-hand side is understood as \(+\infty\).
\end{lemma}

\begin{proof}
	If \(L_\Omega^\nu(o)=0\), there is nothing to prove.  Assume that
	\(L_\Omega^\nu(o)>0\).  Define
	\[
	\psi_\Omega(x)
	:=
	\bigl(L_\Omega^\nu(o)\bigr)^{-1}
	\sum_{y\in\Omega}g_\Omega(x,y)g_\Omega(o,y)^{q-1}\nu(y),
	\qquad x\in\Omega,
	\]
	and extend \(\psi_\Omega\) by zero to \(V\setminus\Omega\).  Then
	\(\psi_\Omega\ge0\) and
	\[
	\psi_\Omega(o)
	=
	\bigl(L_\Omega^\nu(o)\bigr)^{-1}
	\sum_{y\in\Omega}g_\Omega(o,y)^q\nu(y)=1.
	\]
	Moreover, by \eqref{eq:green-fundamental-solution},
	\[
		-\Delta\psi_\Omega(x)
		=
		\bigl(L_\Omega^\nu(o)\bigr)^{-1}
		g_\Omega(o,x)^{q-1}\sigma(x),
		\qquad x\in\Omega.
	\]
	
	Multiplying \eqref{di-2} by
	\(\psi_\Omega^{\frac{q}{q-1}}\) and summing over \(V\) with respect to
	\(\mu\), we get
	\[
	\sum_{x\in V}\Delta u(x)\psi_\Omega(x)^{\frac{q}{q-1}}\mu(x)
	+
	\sum_{x\in V}u(x)^q\psi_\Omega(x)^{\frac{q}{q-1}}\nu(x)
	\le0.
	\]
	Hence, by summation by parts and the convexity inequality
	\[
	\Delta\left(\psi_\Omega^{\frac{q}{q-1}}\right)
	\ge
	\frac{q}{q-1}\psi_\Omega^{\frac1{q-1}}\Delta\psi_\Omega,
	\]
	we obtain
	\begin{align*}
		&\sum_{x\in\Omega}u(x)^q\psi_\Omega(x)^{\frac{q}{q-1}}\nu(x)\\
		&\le
		-\sum_{x\in V}u(x)
		\Delta\left(\psi_\Omega^{\frac{q}{q-1}}\right)(x)\mu(x)\\
		&\le
		-\frac{q}{q-1}\sum_{x\in V}u(x)\psi_\Omega(x)^{\frac1{q-1}}
		\Delta\psi_\Omega(x)\mu(x)\\
		&\le
		\frac{q}{q-1}\bigl(L_\Omega^\nu(o)\bigr)^{-1}
		\sum_{x\in\Omega}u(x)\psi_\Omega(x)^{\frac1{q-1}}
		g_\Omega(o,x)^{q-1}\nu(x).
	\end{align*}
	Applying H\"older's inequality to the last sum gives
	\begin{align*}
		&\sum_{x\in\Omega}u(x)^q\psi_\Omega(x)^{\frac{q}{q-1}}\nu(x)\\
		&\le
		\frac{q}{q-1}\bigl(L_\Omega^\nu(o)\bigr)^{-1}
		\left(\sum_{x\in\Omega}u(x)^q\psi_\Omega(x)^{\frac{q}{q-1}}\nu(x)\right)^{1/q}
		\left(\sum_{x\in\Omega}g_\Omega(o,x)^q\nu(x)\right)^{\frac{q-1}{q}}\\
		&=
		\frac{q}{q-1}\bigl(L_\Omega^\nu(o)\bigr)^{-1/q}
		\left(\sum_{x\in\Omega}u(x)^q\psi_\Omega(x)^{\frac{q}{q-1}}\nu(x)\right)^{1/q}.
	\end{align*}
	Therefore
	\[
	\sum_{x\in\Omega}u(x)^q\psi_\Omega(x)^{\frac{q}{q-1}}\nu(x)
	\le
	\left(\frac{q}{q-1}\right)^{\frac{q}{q-1}}
	\bigl(L_\Omega^\nu(o)\bigr)^{-\frac1{q-1}}.
	\]
	Since \(\psi_\Omega(o)=1\), this implies \eqref{eq:finite-domain-testing}.
\end{proof}

\begin{proof}[Proof of Theorem \ref{thm:green-level-criterion}]
	Assume first that \eqref{et1} fails for some $o\in V$, namely
	\[
	\sum_{y\in V}g(o,y)^q\nu(y)=\infty.
	\]
	Since \(0\not\equiv\sigma\), choose \(z\in V\) with \(\sigma(z)>0\).
	By Lemma~\ref{lem_pole},
	\[
	\sum_{y\in V} g(z,y)^q\nu(y)=\infty.
	\]
	Let $u$ be a positive solution to \eqref{di-2}.  For $R\ge1$, set
	$B_R=B(z,R)$.  Since
	$g_{B_R}(z,y)\uparrow g(z,y)$ as $R\to\infty$, the monotone convergence theorem
	gives
	\[
	L_{B_R}^\nu(z)=\sum_{y\in B_R}g_{B_R}(z,y)^q\nu(y)\longrightarrow\infty.
	\]
	Applying Lemma \ref{lemma6.1} with $\Omega=B_R$ and letting $R\to\infty$, we
	obtain $u(z)=0$.

	Then
	\[
	0\le -\Delta u(z)
	=\frac1{\mu(z)}\sum_{y\sim z}\mu_{zy}\bigl(u(z)-u(y)\bigr)
	=-\frac1{\mu(z)}\sum_{y\sim z}\mu_{zy}u(y)\le0.
	\]
	Thus $u(y)=0$ for every $y\sim z$.  Repeating the same argument along paths and
	using connectedness gives $u\equiv0$, contradicting the positivity of $u$.

	Next we prove \eqref{et2}. Fix \(o\in V\), set
\[
        h(y):=g(o,y),\qquad y\in V,
\]
and write \(A_r=A_r(o)=\{y\in V:h(y)>r^{-1}\}\). By Theorem~\ref{thm:green-testing-equivalence}, the existence of a positive solution implies
\begin{equation}\label{eq:green-shell-inline-input}
        G(\sigma h^q)(x)\le C h(x),
        \qquad x\in V.
\end{equation}
Fix \(r>0\) and put
\(t=r^{-1}\). Since \(h\le h(o)<\infty\), only finitely many of the shells
\[
        E_j:=\{y\in V:2^jt<h(y)\le 2^{j+1}t\},
        \qquad j=0,1,2,\ldots,
\]
are nonempty. For each \(j\), set
\[
        u_j(x):=G(\sigma\mathbf 1_{E_j})(x).
\]
On \(E_j\) we have \(h\ge 2^jt\). Therefore, by
\eqref{eq:green-shell-inline-input}, for every \(x\in V\),
\[
        u_j(x)
        \le (2^jt)^{-q}G(\sigma h^q\mathbf 1_{E_j})(x)
        \le C(2^jt)^{-q}h(x).
\]
In particular, if \(x\in E_j\), then \(h(x)\le 2^{j+1}t\), and hence
\[
        u_j(x)\le 2C(2^jt)^{1-q},
        \qquad x\in E_j.
\]
Since \(\sigma\mathbf 1_{E_j}\) is supported on \(E_j\), Lemma~\ref{lem:wmp-graph}
gives
\[
        \sup_{x\in V}u_j(x)\le 2C(2^jt)^{1-q}.
\]
Summing over \(j\), we obtain
\begin{align*}
        \sup_{x\in V}\sum_{y\in A_r(o)}g(x,y)\nu(y)
        &=\sup_{x\in V}G\bigl(\sigma\mathbf 1_{\{h>t\}}\bigr)(x)  \\
        &\le \sum_{j\ge0}\sup_{x\in V}u_j(x)  \\
        &\le 2C\sum_{j\ge0}(2^jt)^{1-q}
         =\frac{2C}{1-2^{1-q}}\,r^{q-1}.
\end{align*}
This proves \eqref{et2}.
	
	Now assume that \eqref{et1}, \eqref{et2}, and \((3G)\) hold. Fix \(x\in V\), and split the vertex set \(V\) into two parts:
	\begin{align*}
		A_1=\{y\in V: g(x,y)\le 2\kappa g(x,o)\}, \quad
		A_2=\{y\in V: g(x,y)> 2\kappa g(x,o)\},
	\end{align*}
	where \(\kappa\) is the constant in \eqref{3g-pre}. Let \(r_*\) be such that
	\eqref{et2} holds for all \(r>r_*\). Enlarging \(\kappa\), if necessary,
	we may assume that
	\[
	\left(\frac12+\kappa\right)g(o,o)^{-1}>r_*.
	\]
	From \eqref{et1},
	\begin{align}\label{t1-1}
		\sum_{y \in A_1} g(x,y)g(o,y)^q \sigma(y)\mu(y)\le 2\kappa g(x,o) \sum_{y \in A_1}g(o,y)^q \sigma(y)\mu(y)=Cg(x,o).
	\end{align}
	For the second part, since \(g(x,y)> 2\kappa g(x,o)\), using \eqref{3g-pre} we have
	\begin{align*}
		\frac{1}{g(x,o)}&\le \kappa\left(\frac{1}{g(x,y)}+\frac{1}{g(y,o)}\right)
		\le \frac{1}{2g(x,o)}+\frac{\kappa}{g(y,o)},
	\end{align*}
	which implies
	\begin{align}\label{t1-2}
		g(y,o)\le 2 \kappa g(x,o).
	\end{align}
	On the other hand,
	\begin{align*}
		\frac{1}{g(y,o)}&\le \kappa\left(\frac{1}{g(y,x)}+\frac{1}{g(x,o)}\right)
		= \kappa\left(\frac{1}{g(x,y)}+\frac{1}{g(x,o)}\right)
		\le \left(\frac{1}{2}+\kappa \right)\frac{1}{g(x,o)},
	\end{align*}
	which gives
	\begin{align}\label{t1-3}
		g(y,o)\ge \left(\frac{1}{2}+\kappa \right)^{-1}g(x,o).
	\end{align}
	By the maximum principle, \(g(x,o)\le g(o,o)\). Hence
	\(\left(\frac{1}{2}+\kappa \right)g(x,o)^{-1}>r_*\). Using
	\eqref{t1-2}, \eqref{t1-3}, and \eqref{et2}, we obtain
	\begin{align}\label{t1-4}
		\sum_{y\in A_2} g(x,y)g(o,y)^q\,\sigma(y)\mu(y)
		&\le (2 \kappa g(x,o))^{q}\sum_{y\in A_2} g(x,y)\sigma(y)\mu(y)
		\nonumber\\
		&\le (2 \kappa g(x,o))^{q}\sum_{y\in A_3} g(x,y)\sigma(y)\mu(y)
		\nonumber\\
		&\le C\, (2 \kappa g(x,o))^{q}(\tfrac{1}{2}+\kappa)^{q-1}g(x,o)^{-(q-1)}
		\nonumber\\
		&= C g(x,o),
	\end{align}
	where \(A_3:=\{y \in V: g(o,y) > (\tfrac{1}{2}+\kappa )^{-1}g(x,o)\}\).
	As established above, under condition \eqref{3g-pre} we have \(A_2\subset A_3\). Combining \eqref{t1-1} with \eqref{t1-4} yields \eqref{e1}. Hence, by
	Theorem~\ref{thm:green-testing-equivalence}, \eqref{di-2} admits a positive
	solution. This completes the proof.
\end{proof}

\begin{remark}
The necessity part of Theorem~\ref{thm:green-level-criterion} can be viewed as
a discrete counterpart of the Green level estimates in \cite{GSV20}, but the
proof is structurally simpler. In the manifold setting, the analogues of
\eqref{et1} and \eqref{et2} are obtained through weighted norm inequalities
and nonlinear iteration arguments for the Green operator. In the present graph
setting, once the Green test \eqref{e1} is available from
Theorem~\ref{thm:green-testing-equivalence}, the two necessary estimates follow
from elementary discrete ingredients: the energy condition \eqref{et1} follows
from the finite-domain Green-energy estimate Lemma~\ref{lemma6.1}, while the
localization estimate \eqref{et2} follows from a dyadic decomposition of the
level set \(A_r(o)\) and the weak maximum principle
Lemma~\ref{lem:wmp-graph}.
\end{remark}

\section{Proof of Theorem \ref{thm:vd-pi-p0-potential-criterion}}\label{sec5}

We first recall the Green-function estimates needed to translate
Theorem~\ref{thm:green-level-criterion} into volume conditions.

\begin{lemma}[{\cite[Proposition~4.2]{Delmotte}, \cite[Theorem~1.10]{HS2025}}]\label{lem:Green-estimate}	
    Assume that the weighted graph $(V,\mu)$ satisfies \hyperref[VD]{$(\mathrm{VD})$}, \hyperref[PI]{$(\mathrm{PI})$}, and \hyperref[p0]{$(\mathrm{P_0})$}. Then $(V,\mu)$ is non-parabolic if and only if
\[\sum_{n=1}^{\infty}\frac{n}{\mu(B(o,n))}<\infty\] for some $o\in V$. Moreover, for all $x,y\in V$,
	\[
  g(x,y)\asymp \sum\limits_{n=d(x,y)}^{\infty}\frac{n}{\mu(B(x,n))}.
	\]
\end{lemma}

\begin{lemma}\label{lem3g}
    Under the assumptions of Lemma~\ref{lem:Green-estimate}, further assuming that $(V,\mu)$ is non-parabolic,  define
    \begin{align*}
       l (x,y)= \left[\sum\limits_{n=d(x,y)}^{\infty}\frac{n}{\mu(B(x,n))}\right]^{-1},\quad \text{for all } x,y\in V.
    \end{align*}
    Then \(l(x,y)\) is finite and positive for all \(x,y\in V\), and there exists \(\kappa>0\) such that
    \begin{align*}
        l (x,y)\le \kappa[l (x,z)+l (z,y)], \text{ for all }x,y,z\in V.
    \end{align*}
\end{lemma}
\begin{proof}
Set
\[
R_x(r):=\sum_{n=r}^{\infty}\frac{n}{\mu(B(x,n))},\qquad r\ge0.
\]
Then \(l(x,y)=R_x(d(x,y))^{-1}\). By \hyperref[VD]{$(\mathrm{VD})$}, for every \(r\ge0\),
\[
R_x(r)\asymp R_x(2r).
\]
Moreover, if \(d(x,z)\le r\), then
\[
B(x,r)\subset B(z,2r)\qquad\text{and}\qquad B(z,r)\subset B(x,2r),
\]
and hence \(\mu(B(x,r))\asymp \mu(B(z,r))\). Consequently,
\[
R_x(r)\asymp R_z(r)\qquad\text{whenever } d(x,z)\le r.
\]
Let \(R=d(x,y)\) and \(k=\lfloor R/2\rfloor\). Then at least one of
\[
d(x,z)\ge k \qquad\text{or}\qquad d(z,y)\ge k
\]
holds. If \(d(x,z)\ge k\), then we have
\[
R_x(d(x,y))=R_x(R)\gtrsim R_x(k)\gtrsim R_x(d(x,z)),
\]
and therefore
\[
l(x,y)\lesssim l(x,z).
\]
If \(d(z,y)\ge k\), then \(R\le 2d(z,y)+1\) and
\[
d(x,z)\le d(x,y)+d(y,z)\le 4d(z,y).
\]
Hence for every \(n\ge d(z,y)\),
\[
B(x,n)\subset B(z,n+d(x,z))\subset B(z,5n).
\]
By \hyperref[VD]{$(\mathrm{VD})$}, this implies \(\mu(B(x,n))\lesssim \mu(B(z,n))\), so
\[
R_x(d(x,y))=R_x(R)\gtrsim R_x(2d(z,y))\gtrsim R_x(d(z,y))\gtrsim R_z(d(z,y)).
\]
Thus
\[
l(x,y)\lesssim l(z,y).
\]
Combining the above two cases, we conclude that there exists \(\kappa>0\) such that
\[
l(x,y)\le \kappa\bigl(l(x,z)+l(z,y)\bigr)
\]
for all \(x,y,z\in V\).
\end{proof}

\begin{proof}[Proof of Theorem \ref{thm:vd-pi-p0-potential-criterion}]
    Suppose that \eqref{di-2} admits a positive solution. In particular, $(V,\mu)$ is non-parabolic.
    By Theorem~\ref{thm:green-level-criterion}, we have
    \begin{align*}
        \sum_{x \in V} g(o, x)^q \sigma(x)\mu(x) < \infty,
    \end{align*}
     and for all \(r>r_0\),
    \begin{align}\label{et2_tm}
       \sup_{x\in V}\sum_{\substack{y \in V \\ g(o,y) > r^{-1}}} g(x,y)\,\sigma(y)\mu(y) \lesssim r^{q-1}.
    \end{align}

    From Lemma \ref{lem:Green-estimate}, we have
    \[
        g(o, x) \gtrsim \sum_{n = d(o, x)}^{\infty} \frac{n}{\mu(B(o, n))}.
    \]
    Hence, using Fubini's theorem, we obtain
    \begin{align*}
        \infty
        &> \sum_{x \in V} g(o, x)^q \sigma(x)\mu(x) \\
        &\gtrsim \sum_{x \in V}
            \Bigg[\sum_{m = d(o, x)}^{\infty}
            \frac{m}{\mu(B(o, m))}\Bigg]^q \sigma(x)\mu(x) \\
        &= \sum_{x \in V}
            \Bigg(\sum_{m = d(o, x)}^{\infty}
            \frac{m}{\mu(B(o, m))}\Bigg)
            \Bigg[\sum_{m = d(o, x)}^{\infty}
            \frac{m}{\mu(B(o, m))}\Bigg]^{q - 1} \sigma(x)\mu(x) \\
        &= \sum\limits_{i=0}^{\infty}\sum_{x\in V: ~d(o,x)=i}
            \Bigg(\sum_{n = i}^{\infty}
            \frac{n}{\mu(B(o, n))}\Bigg)
            \Bigg[\sum_{m = i}^{\infty}
            \frac{m}{\mu(B(o, m))}\Bigg]^{q - 1} \sigma(x)\mu(x) \\
       &\ge \sum_{n = 1}^{\infty} \sum_{i = 0}^{n}
            \frac{n}{\mu(B(o, n))}
            \Bigg[\sum_{m = n}^{\infty}
            \frac{m}{\mu(B(o, m))}\Bigg]^{q - 1}
            \nu(\{x\in V:d(o,x)=i\}) \nonumber\\
       &= \sum_{n = 1}^{\infty}
            \Bigg[\sum_{m = n}^{\infty}
            \frac{m}{\mu(B(o, m))}\Bigg]^{q - 1}
            \frac{n\, \nu(B(o, n))}{\mu(B(o, n))},
    \end{align*}
    where  \(\nu = \sigma \mu\). 
    This proves \eqref{thm_e1}.

    Next, the lower Green-function estimate in
    Lemma~\ref{lem:Green-estimate} gives
    \begin{align*}
        g(o, y)\ge C\sum_{m = d(o, y)}^{\infty}
            \frac{m}{\mu(B(o, m))}.
    \end{align*}

    Set
    \[
        T_j:=\sum_{m=j}^{\infty}\frac{m}{\mu(B(o,m))}.
    \]
    By Lemma~\ref{lem:Green-estimate}, the tail sums \(T_j\) are finite. Moreover,
    \(T_j\downarrow0\) and \(T_j-T_{j+1}=j/\mu(B(o,j))>0\) for \(j\ge1\).
    Choose \(N\) large enough such that \(C T_N<r_0^{-1}\).
    Then, for every integer \(n\ge N\), we may choose \(r>0\) such that

    \begin{align}\label{thmmain_1a}
    C T_n> \frac1r > C T_{n+1}.
    \end{align}
    For such \(n\), we have \(r>r_0\), and hence \eqref{et2_tm} applies.

    For any \(y \in V\) with \(d(o, y) \le n\), we have
    \begin{align*}
        g(o, y)
        &\ge C\sum_{m = d(o, y)}^{\infty}
            \frac{m}{\mu(B(o, m))} \\
        &\ge C\sum_{m = n}^{\infty}
            \frac{m}{\mu(B(o, m))} > \frac{1}{r}.
    \end{align*}
    Define
    \begin{align*}
        A_r = \{y \in V : g(o, y) > r^{-1}\}.
    \end{align*}
    Clearly, \(B(o, n) \subset A_r\).
    Then
    \begin{align*}
        \sum_{y \in A_r} g(x, y)\sigma(y)\mu(y)
        &\gtrsim \sum_{y \in B(o, n)}
            \sum_{m = d(x, y)}^{\infty}
            \frac{m}{\mu(B(x, m))}\sigma(y)\mu(y)
            \\
        &= \sum_{i = 0}^{\infty} \sum_{m = i}^{\infty}
            \frac{m}{\mu(B(x, m))}
            \sum_{\substack{y \in B(o, n) \\ d(x, y) = i}}
            \sigma(y)\mu(y)
            \\
        &= \sum_{m = 1}^{\infty}
            \frac{m}{\mu(B(x, m))}
            \nu(B(x, m) \cap B(o, n)),
    \end{align*}
    where again \(\nu = \sigma \mu\).

    Combining this with \eqref{et2_tm} and \eqref{thmmain_1a}, we deduce
    \begin{align*}
        \sum_{m = 1}^{\infty}
            \frac{m\nu(B(x, m) \cap B(o, n))}{\mu(B(x, m))}\lesssim r^{q - 1}
        < \Bigg[\sum_{m = n + 1}^{\infty}
        \frac{m}{\mu(B(o, m))}\Bigg]^{-(q - 1)}.
    \end{align*}
    From \hyperref[VD]{\((\mathrm{VD})\)}, it follows that
    \begin{align*}
        \sum_{m = n}^{\infty}
        \frac{m}{\mu(B(o, m))}
        \lesssim \sum_{m = n + 1}^{\infty}
        \frac{m}{\mu(B(o, m))}.
    \end{align*}
    Therefore,
    \begin{align*}
        \sum_{m = 1}^{\infty}
            \frac{m\nu(B(x, m) \cap B(o, n))}{\mu(B(x, m))}
        \Bigg[\sum_{i = n}^{\infty}
        \frac{i}{\mu(B(o, i))}\Bigg]^{q - 1}
        \lesssim 1,
    \end{align*}
    uniformly in \(x\in V\) and in all \(n\ge N\).
    It remains only to handle integers \(1\le n<N\).
    For fixed \(n\), using Lemma~\ref{lem:Green-estimate},
    \[
    \begin{aligned}
    	\sum_{m=1}^{\infty}
    	\frac{m\,\nu(B(x,m)\cap B(o,n))}{\mu(B(x,m))}
    	&=
    	\sum_{y\in B(o,n)}\nu(y)
    	\sum_{m\ge d(x,y)}
    	\frac{m}{\mu(B(x,m))}  \\
    	&\lesssim
    	\sum_{y\in B(o,n)}\nu(y)g(x,y)
    	\le
    	\sum_{y\in B(o,N)}\nu(y)g(y,y)<\infty.
    \end{aligned}
    \]
    It follows that, for \(1\le n<N\),
     \begin{align*}
    	\sum_{m = 1}^{\infty}
    	\frac{m\nu(B(x, m) \cap B(o, n))}{\mu(B(x, m))}
    	\Bigg[\sum_{i = n}^{\infty}
    	\frac{i}{\mu(B(o, i))}\Bigg]^{q - 1}
    	\lesssim C,
    \end{align*}
    where \(C=g(o,o)^{q-1}\sum_{y\in B(o,N)}\nu(y)g(y,y)\).

    This establishes
    \eqref{thm_e2} for all \(n\ge1\).

    Conversely, assume that \eqref{thm_e1} and \eqref{thm_e2} hold. Since $\sigma\not \equiv  0$, the finiteness of
    \eqref{thm_e1} implies that
    \[
    \sum_{m=n}^{\infty}\frac{m}{\mu(B(o,m))}<\infty
    \]
    for all sufficiently large \(n\).   By  Lemma \ref{lem:Green-estimate} and Lemma \ref{lem3g}, $(V,\mu)$ is non-parabolic, and \eqref{3g-pre} also holds.  It is sufficient to verify
     \begin{align*}
        \sum_{y \in V} g(o,y)^q \sigma(y)\mu(y) < \infty,
    \end{align*}
    and
    \begin{align*}
       \sup_{x\in V}\sum_{y\in A_{r}} g(x,y)\,\sigma(y)\mu(y) \lesssim r^{q-1},
    \end{align*}
    for all \(r>r_0\), where \(r_0=\kappa g(o,o)^{-1}\) and
    \(A_r:=\{y\in V:g(o,y)>r^{-1}\}\).

    Note that
    \begin{align*}
        g(o, x) \lesssim \sum_{n = d(o, x)}^{\infty} \frac{n}{\mu(B(o, n))}.
    \end{align*}
    Set
    \[
        a_n:=\frac{n}{\mu(B(o,n))},
        \qquad
        A_i:=\sum_{m=i}^{\infty}a_m.
    \]
    By Lemma~\ref{lem:Green-estimate}, \(A_i<\infty\) for all \(i\), \(A_i\downarrow0\),
    and \(A_i-A_{i+1}=a_i>0\) for \(i\ge1\).
    Since
    \[
        A_i^q
        =\sum_{n=i}^{\infty}(A_n^q-A_{n+1}^q)
        \le q\sum_{n=i}^{\infty}a_nA_n^{q-1},
    \]
    we obtain, by Fubini's theorem,
    \begin{align*}
        \sum_{x \in V} g(o, x)^q \sigma(x)\mu(x)
        &\lesssim \sum_{i=0}^{\infty}A_i^q\nu(\{x\in V:d(o,x)=i\}) \\
        &\le q\sum_{n = 1}^{\infty}a_nA_n^{q - 1}\nu(B(o,n)) \\
        &=q\sum_{n = 1}^{\infty}
            \Bigg[\sum_{m = n}^{\infty}
            \frac{m}{\mu(B(o, m))}\Bigg]^{q - 1}
            \frac{n\, \nu(B(o, n))}{\mu(B(o, n))}<\infty.
    \end{align*}

    To verify \eqref{et2}, let \(C>0\) be such that
    \[
        g(o,y)\le C\sum_{m=d(o,y)}^{\infty}\frac{m}{\mu(B(o,m))}
        \qquad\text{for all } y\in V.
    \]
    Enlarging \(r_0\), if necessary, we may assume that
    \(r^{-1}< C T_1\) for every \(r\ge r_0\).
    For such
    \(r\), choose \(n\) such that
    \begin{align*}
        C\sum_{m=n}^{\infty}\frac{m}{\mu(B(o,m))} > \frac1r
        \ge C\sum_{m=n+1}^{\infty}\frac{m}{\mu(B(o,m))}.
    \end{align*}
    If \(y\in A_r\), then \(g(o,y)>r^{-1}\), and hence
    \[
        C\sum_{m=d(o,y)}^{\infty}\frac{m}{\mu(B(o,m))}
        \ge g(o,y)> \frac1r
        \ge C\sum_{m=n+1}^{\infty}\frac{m}{\mu(B(o,m))}.
    \]
    Since the tail sum is decreasing in \(d(o,y)\), it follows that \(d(o,y)\le n\), that is,
    \[
        A_r\subset B(o,n).
    \]
    Therefore,
    \begin{align*}
        \sum_{y \in A_r} g(x, y)\sigma(y)\mu(y)
        &\lesssim \sum_{y \in B(o, n)}
            \sum_{m = d(x, y)}^{\infty}
            \frac{m}{\mu(B(x, m))}\sigma(y)\mu(y) \\
        &= \sum_{m = 1}^{\infty}
            \frac{m}{\mu(B(x, m))}
            \nu(B(x, m) \cap B(o, n)).
    \end{align*}
    By \eqref{thm_e2}, we obtain
    \begin{align*}
        \sum_{y \in A_r} g(x, y)\sigma(y)\mu(y)
        &\lesssim
        \Bigg[
            \sum_{m = n}^{\infty}
            \frac{m}{\mu(B(o, m))}
        \Bigg]^{1 - q}
        \lesssim r^{q - 1}.
    \end{align*}
    This completes the proof.
\end{proof}

An integral counterpart of the following equivalence was proved in
\cite[proof of Corollary~1.2]{GSV20}.  We give a direct discrete argument,
valid for every \(p>0\), and also record the resulting quantitative
comparison.
\begin{lemma}\label{lem:unweighted-tail-equivalence}
Assume \textnormal{(VD)}, fix \(o\in V\), and write
\[
V_n:=\mu(B(o,n)),
\qquad
T_n:=\sum_{m=n}^{\infty}\frac{m}{V_m},
\qquad n\ge1.
\]
Then, for every \(p>0\),
\begin{equation}\label{eq:unweighted-tail-equivalence}
\sum_{n=1}^{\infty}nT_n^p<\infty
\quad\Longleftrightarrow\quad
\sum_{n=1}^{\infty}n\left(\frac{n^2}{V_n}\right)^p<\infty.
\end{equation}
More precisely, the two quantities in \eqref{eq:unweighted-tail-equivalence}
control one another up to constants depending only on \(p\) and the doubling
constant.
\end{lemma}

\begin{proof}
Set
\[
b_n:=\frac{n^2}{V_n}.
\]
By \textnormal{(VD)},
\[
T_n
\ge
\sum_{m=n}^{2n-1}\frac{m}{V_m}
\ge
\frac{n^2}{V_{2n}}
\gtrsim
\frac{n^2}{V_n}
=b_n.
\]
Consequently,
\begin{equation*}
\sum_{n=1}^{\infty}n b_n^p
\lesssim
\sum_{n=1}^{\infty}nT_n^p.
\end{equation*}

For the converse estimate, decompose the tail into dyadic blocks.  Since
\(V_m\ge V_{2^jn}\) for \(2^jn\le m<2^{j+1}n\),
\begin{align*}
T_n
&=
\sum_{j=0}^{\infty}
\sum_{m=2^jn}^{2^{j+1}n-1}\frac{m}{V_m} \\
&\le
2\sum_{j=0}^{\infty}
\frac{(2^jn)^2}{V_{2^jn}}
=
2\sum_{j=0}^{\infty}b_{2^jn}.
\end{align*}
All the following estimates may first be applied to truncated tails and then
passed to the limit by monotone convergence.

If \(0<p\le1\), subadditivity of \(t\mapsto t^p\) gives
\begin{align*}
\sum_{n=1}^{\infty}nT_n^p
&\le
2^p\sum_{j=0}^{\infty}
\sum_{n=1}^{\infty}n b_{2^jn}^p \\
&=
2^p\sum_{j=0}^{\infty}2^{-j}
\sum_{n=1}^{\infty}(2^jn)b_{2^jn}^p \\
&\le
2^p\sum_{j=0}^{\infty}2^{-j}
\sum_{m=1}^{\infty}m b_m^p
\lesssim_p
\sum_{m=1}^{\infty}m b_m^p.
\end{align*}
If \(p>1\), Minkowski's inequality yields
\begin{align*}
\left(\sum_{n=1}^{\infty}nT_n^p\right)^{1/p}
&\le
2\sum_{j=0}^{\infty}
\left(\sum_{n=1}^{\infty}n b_{2^jn}^p\right)^{1/p} \\
&\le
2\left(\sum_{m=1}^{\infty}m b_m^p\right)^{1/p}
\sum_{j=0}^{\infty}2^{-j/p}.
\end{align*}
This proves the reverse inequality and hence
\eqref{eq:unweighted-tail-equivalence}.
\end{proof}

In the unweighted manifold setting, the analogous implication from the
first condition to the second was observed in
\cite[discussion preceding Corollary~1.2]{GSV20}. The next lemma gives
the corresponding discrete statement, together with a proof based on
Green-function estimates and the weak maximum principle.
\begin{lemma}\label{lem:unweighted-testing-automatic}
Assume \textnormal{(VD)}, \textnormal{(PI)}, and \textnormal{(P$_0$)}.  Fix
\(o\in V\), use the notation of
Lemma~\ref{lem:unweighted-tail-equivalence}, and let \(p>0\).  If
\begin{equation}\label{eq:unweighted-first-condition-finite}
\sum_{n=1}^{\infty}nT_n^p
<\infty,
\end{equation}
then
\begin{equation}\label{eq:unweighted-second-condition-automatic}
\sup_{n\ge1}\sup_{x\in V}
\left[
\sum_{m=1}^{\infty}
 m\frac{\mu(B(x,m)\cap B(o,n))}{\mu(B(x,m))}
\right]T_n^p
<\infty.
\end{equation}
Thus, in the unweighted case, the second condition in
Theorem~\ref{thm:vd-pi-p0-potential-criterion} follows from the first one.
\end{lemma}

\begin{proof}
Condition \eqref{eq:unweighted-first-condition-finite} implies \(T_1<\infty\),
so the graph is non-parabolic by Lemma~\ref{lem:Green-estimate}.  Set
\[
B_n:=B(o,n),
\qquad
F_n(x):=G\mathbf 1_{B_n}(x).
\]
The lower Green estimate in Lemma~\ref{lem:Green-estimate} and Tonelli's
theorem give
\begin{equation}\label{eq:testing-sum-below-potential}
\sum_{m=1}^{\infty}
 m\frac{\mu(B(x,m)\cap B_n)}{\mu(B(x,m))}
\lesssim
F_n(x).
\end{equation}
Since \(\mathbf 1_{B_n}\) is supported in \(B_n\), the weak maximum principle,
Lemma~\ref{lem:wmp-graph}, gives
\[
F_n(x)\le \sup_{z\in B_n}F_n(z),
\qquad x\in V.
\]
Fix \(z\in B_n\).  By the upper Green estimate and Tonelli's theorem,
\begin{align*}
F_n(z)
&\lesssim
\sum_{m=1}^{\infty}
 m\frac{\mu(B(z,m)\cap B_n)}{\mu(B(z,m))} \\
&\le
\sum_{m\le2n}m
+
V_n\sum_{m>2n}\frac{m}{\mu(B(z,m))}.
\end{align*}
For \(m>2n\), one has
\[
B(o,m-n)\subset B(z,m),
\qquad m-n\ge \frac m2.
\]
By \textnormal{(VD)}, this implies
\[
\mu(B(z,m))\ge V_{m-n}\gtrsim V_m.
\]
Therefore
\begin{equation}\label{eq:ball-potential-estimate}
F_n(x)
\lesssim
n^2+V_nT_n,
\qquad x\in V.
\end{equation}

It remains to control the right-hand side after multiplication by \(T_n^p\).
Since \(T_m\) is decreasing,
\begin{equation}\label{eq:n2-tail-control}
n^2T_n^p
\lesssim
\sum_{m=\lceil n/2\rceil}^{n}mT_m^p
\le \sum_{m=1}^{\infty}mT_m^p.
\end{equation}
Also, \(T_m\downarrow0\) and
\(T_m-T_{m+1}=m/V_m\).  Hence, by the convexity inequality
\(a^{p+1}-b^{p+1}\le(p+1)a^p(a-b)\) for \(a\ge b\ge0\),
\begin{align}
V_nT_n^{p+1}
&=
V_n\sum_{m=n}^{\infty}
\bigl(T_m^{p+1}-T_{m+1}^{p+1}\bigr) \nonumber\\
&\le
(p+1)V_n\sum_{m=n}^{\infty}
T_m^p\frac{m}{V_m} \nonumber\\
&\le
(p+1)\sum_{m=n}^{\infty}mT_m^p
\le
(p+1)\sum_{m=1}^{\infty}mT_m^p.
\label{eq:volume-tail-control}
\end{align}
Combining \eqref{eq:testing-sum-below-potential},
\eqref{eq:ball-potential-estimate}, \eqref{eq:n2-tail-control}, and
\eqref{eq:volume-tail-control} proves
\eqref{eq:unweighted-second-condition-automatic}.
\end{proof}

\begin{proof}[Proof of Corollary~\ref{cor:unweighted-one-condition}]
Take \(\sigma\equiv1\), so that \(\nu=\mu\), and put \(p=q-1\).  Then the
first condition \eqref{thm_e1} in
Theorem~\ref{thm:vd-pi-p0-potential-criterion} becomes
\[
\sum_{n=1}^{\infty}nT_n^{q-1}<\infty.
\]
By Lemma~\ref{lem:unweighted-tail-equivalence}, this is equivalent to
\[
\sum_{n=1}^{\infty}
\frac{n^{2q-1}}{\mu(B(o,n))^{q-1}}<\infty.
\]
Moreover, whenever this condition holds,
Lemma~\ref{lem:unweighted-testing-automatic} shows that the second condition
\eqref{thm_e2} is automatic.  The conclusion now follows directly from
Theorem~\ref{thm:vd-pi-p0-potential-criterion}.
\end{proof}

\section{Flow decomposition and capacity estimates}
\label{sec:intrinsic-flow}
In this section, we prove Theorem~\ref{thm:weighted-liouville}.  The argument
based on flow decomposition first yields a finite-domain relative capacity
estimate; criteria based on annular conductance, capacity to infinity, and
volume growth then follow as consequences.  Throughout this section,
\(q>1\), \(0<\sigma\in\ell^+(V)\), \(\nu=\sigma\mu\), and \(\rho\) is a
\(\nu\)-adapted edge-length function such that \((V,d_\rho)\) is complete.
All balls and radial variables are defined using \(d_\rho\), while \(d\)
continues to denote the graph distance.  Corollary~\ref{cor:unweighted-liouville}
follows by taking \(\sigma\equiv1\) and \(\rho\equiv1\).


Fix a base point \(o\in V\), set \(\nu=\sigma\mu\), and, for \(r\ge0\), write
\begin{equation*}\label{eq:section6-ball-volume-notation}
B_r:=B_{d_\rho}(o,r).
\end{equation*}
By Proposition~\ref{prop:complete-finite-balls}, every ball \(B_r\) is finite.
It is also connected.  Indeed, if \(x\in B_r\), an approximating path from
\(o\) to \(x\) of length less than \(r+1\) lies in the finite ball
\(B_{r+1}\).  After loops are removed, only finitely many such simple paths
remain, so one of them realizes \(d_\rho(o,x)\).  Every initial segment of a
minimizing path has length at most \(d_\rho(o,x)\le r\), and therefore stays in
\(B_r\).

For \(r>0\), let
\begin{equation*}\label{eq:annular-conductance-profile}
\mathcal C_o(r)
:=
C_{\mathrm{eff}}\bigl(B_r,V\setminus B_{2r}\bigr).
\end{equation*}
Here effective conductance has the meaning fixed in
Subsection~\ref{subsection:finite-flows}: since \(B_{2r}\) is finite, one
retains all edges incident with \(B_{2r}\) and wires \(B_r\) and
\(V\setminus B_{2r}\) separately. 
Clearly, we have
$0<\mathcal C_o(r)<\infty$ for $r>0$.

For \(R>0\), define
\begin{equation*}\label{eq:section6-green-energy-notation}
g_R(x)
:=
\begin{cases}
g_{B_R}(o,x),&x\in B_R,\\
0,&x\notin B_R,
\end{cases}
\qquad
L_R^\nu
:=
\sum_{x\in B_R}g_R(x)^q\nu(x).
\end{equation*}
For \(f\in\ell_0(V)\), set
\begin{equation*}\label{eq:ambient-dirichlet-energy}
\mathcal E(f)
:=
\frac12\sum_{x\in V}\sum_{y\sim x}
\mu_{xy}\bigl(f(x)-f(y)\bigr)^2.
\end{equation*}
For \(0<r<R\), define the capacity of \(B_r\) relative to the killed domain
\(B_R\) by
\begin{equation}\label{eq:relative-capacity-profile}
\capacity_{B_R}(B_r)
:=
\inf\left\{
\mathcal E(f):
 f=1\text{ on }B_r,\quad
 f=0\text{ on }V\setminus B_R
\right\}.
\end{equation}
The central estimate of the section is
\begin{equation*}
L_R^\nu
\ge c_q\int_0^R r\,\capacity_{B_R}(B_r)^{1-q}\,dr,
\qquad R>0,
\end{equation*}
where \(c_q>0\) depends only on \(q\).  Its annular consequence is
\begin{equation*}
L_R^\nu
\ge c_q\int_0^{R/2}r\,\mathcal C_o(r)^{1-q}\,dr,
\qquad R>0.
\end{equation*}
The latter can be used directly as a conductance criterion.  By estimating
\(\mathcal C_o\) with an intrinsic radial cutoff, it also yields
\begin{equation}\label{eq:intrinsic-green-volume-lower}
L_R^\nu
\ge c_q\int_0^R
\frac{r^{2q-1}}{\nu(B_r)^{q-1}}\,dr,
\qquad R>0.
\end{equation}

\subsection{Unit current and flow decomposition}

Let
\[
\partial_E B_R
:=
\bigl\{e=\{x,y\}\in E:x\in B_R,\ y\notin B_R\bigr\}.
\]
For each \(e=\{x,y\}\in\partial_E B_R\), with \(x\in B_R\) and
\(y\notin B_R\), introduce a terminal vertex \(b_e\).  Let
\[
\mathcal T_R:=\{b_e:e\in\partial_E B_R\}.
\]
We form a finite weighted graph \((\widehat V_R,\widehat\mu_R)\) as
follows.  Its vertex set is
\[
\widehat V_R:=B_R\cup\mathcal T_R.
\]
We retain every edge with both endpoints in \(B_R\) and replace each boundary
edge \(e=\{x,y\}\) by a copied edge joining \(x\) to \(b_e\).  An internal
edge \(\{x,z\}\subset B_R\) retains the weight \(\mu_{xz}\), while the copied
edge \(\{x,b_e\}\), corresponding to \(e=\{x,y\}\), receives the weight
\(\mu_{xy}\).  Every retained or copied edge is labelled by its unique
original edge \(e\); we denote its weight by \(\mu_e\) and its original
intrinsic length by \(\rho_e\).  Thus \(\rho_e=\rho(x,y)\) when the original
edge is \(e=\{x,y\}\).

Define a radial label on the vertices of \(\widehat V_R\) by
\begin{equation*}\label{eq:terminal-radial-label}
\mathfrak r_R(z)
:=
\begin{cases}
d_\rho(o,z),&z\in B_R,\\
d_\rho(o,y),&z=b_e,\ e=\{x,y\}\in\partial_E B_R.
\end{cases}
\end{equation*}
In particular, \(\mathfrak r_R(b_e)>R\).  Thus the terminal copies retain the
edge weight, intrinsic length, and radial position of every original boundary
edge.  Extend \(g_R\) by zero to \(\mathcal T_R\).  By symmetry of the Green
function and \eqref{eq:green-fundamental-solution}, the zero extension of
\(g_R\) to \(V\setminus B_R\) satisfies
\begin{equation}\label{eq:intrinsic-current-equation}
\sum_{y\sim x}
\mu_{xy}\bigl(g_R(x)-g_R(y)\bigr)
=
\mathbf 1_{\{o\}}(x),
\qquad x\in B_R.
\end{equation}
Under the correspondence between original and copied edges, this is precisely
the current-conservation equation on the finite weighted graph
\((\widehat V_R,\widehat\mu_R)\).

Orient every edge with nonzero voltage drop from the larger value of \(g_R\)
to the smaller one, discard the edges with zero voltage drop, and set
\begin{equation*}\label{eq:green-current-definition}
\theta_e
:=
\mu_e\bigl(g_R(\tail(e))-g_R(\head(e))\bigr)
\end{equation*}
for every retained directed edge, with the antisymmetric extension to the
reverse orientation.  Since \(g_R>0\) on \(B_R\), every retained terminal
edge is directed into \(\mathcal T_R\).  Equation~\eqref{eq:intrinsic-current-equation} then
shows that \(\theta\) is a unit flow from \(\{o\}\) to \(\mathcal T_R\).
Since the voltage decreases strictly along its directed support, the flow is
acyclic.
Lemma~\ref{lem:finite-acyclic-flow-decomposition} therefore gives a probability
measure on directed paths
\[
\gamma=(x_0=o,e_0,x_1,\ldots,e_{m-1},x_m),
\qquad x_m\in\mathcal T_R,
\]
such that
\begin{equation}\label{eq:intrinsic-edge-marginal}
\mathbb P(\gamma\text{ uses }e)=\theta_e
\end{equation}
for every retained directed edge \(e\). Fix one such probability measure.

For a sampled path, write
\begin{equation*}\label{eq:intrinsic-path-notation}
\begin{aligned}
V_i&=g_R(x_i),&&0\le i\le m,\\
\ell_i&=\rho_{e_i},\qquad
\delta_i=V_i-V_{i+1}>0,&&0\le i\le m-1,\\
s_i&=\sum_{j=0}^{i-1}\ell_j,&&0\le i\le m.
\end{aligned}
\end{equation*}
Thus \(V_m=0\), \(s_0=0\), and the total path length is
\(L_\gamma=s_m\).  The path in the finite weighted graph
\((\widehat V_R,\widehat\mu_R)\) corresponds to an original graph path ending at the exterior vertex represented by \(x_m\).
Consequently,
\begin{equation}\label{eq:radial-label-prefix-bound}
\mathfrak r_R(x_i)\le s_i,
\qquad 0\le i\le m,
\end{equation}
and, in particular,
\begin{equation}\label{eq:path-length-exceeds-radius}
L_\gamma\ge\mathfrak r_R(x_m)>R.
\end{equation}
Define the path energy
\begin{equation*}\label{eq:intrinsic-path-energy}
H_\gamma
:=
\sum_{i=0}^{m-1}\frac{\ell_i^2V_i^q}{\delta_i}.
\end{equation*}

\begin{lemma}\label{lem:intrinsic-path-energy}
For every \(R>0\),
\begin{equation*}\label{eq:intrinsic-path-energy-expectation}
\mathbb E_\gamma H_\gamma\le L_R^\nu.
\end{equation*}
\end{lemma}

\begin{proof}
Using \eqref{eq:intrinsic-edge-marginal}, we obtain
\begin{align*}
\mathbb E_\gamma H_\gamma
&=
\sum_{e=(x,z)}
\theta_e\frac{\rho_e^2g_R(x)^q}{g_R(x)-g_R(z)}\\
&=
\sum_{e=(x,z)}\mu_e\rho_e^2g_R(x)^q,
\end{align*}
where the sums are over retained directed edges and \(x=\tail(e)\).  Every
such tail lies in \(B_R\).  For each \(x\in B_R\), the copied boundary edges
are in one-to-one correspondence with the original boundary edges incident
with \(x\).  Therefore, by \eqref{eq:adapted-weight},
\begin{align*}
\mathbb E_\gamma H_\gamma
&\le
\sum_{x\in B_R}g_R(x)^q
\sum_{y\sim x}\mu_{xy}\rho(x,y)^2\\
&\le
\sum_{x\in B_R}g_R(x)^q\nu(x)
=
L_R^\nu.
\end{align*}
\end{proof}

\subsection{A deterministic estimate along one path}

Let \(v_\gamma:[0,L_\gamma]\to[0,\infty)\) be the piecewise affine function
satisfying \(v_\gamma(s_i)=V_i\).

\begin{lemma}\label{lem:weighted-path-hardy}
Let \(q>1\).  Suppose that \(v:[0,L]\to[0,\infty)\) is continuous,
strictly decreasing, and affine on each interval of a finite partition.  Then
\begin{equation*}\label{eq:weighted-path-hardy}
\int_0^L\frac{v(s)^q}{-v'(s)}\,ds
\ge
\frac{q-1}{4\log 2}\int_0^L s\,v(s)^{q-1}\,ds.
\end{equation*}
\end{lemma}

\begin{proof}
Set
\[
a(s)=\frac{-v'(s)}{v(s)^q},
\qquad
A(t)=\int_0^t a(s)\,ds.
\]
For \(0<t<L\), the Cauchy--Schwarz inequality gives
\[
\int_{t/2}^t\frac{ds}{a(s)}
\ge
\frac{t^2}{4A(t)}.
\]
After division by \(t\), integration in \(t\), and a change of order,
\[
\int_0^L\frac{t}{A(t)}\,dt
\le
4\log 2\int_0^L\frac{ds}{a(s)}.
\]
Moreover,
\[
A(t)
=
\int_0^t\frac{-v'(s)}{v(s)^q}\,ds
\le
\frac{v(t)^{1-q}}{q-1}.
\]
Hence \(A(t)^{-1}\ge(q-1)v(t)^{q-1}\), which proves the result.
\end{proof}

On \((s_i,s_{i+1})\), one has
\(-v_\gamma'=\delta_i/\ell_i\) and \(v_\gamma\le V_i\).  Therefore,
\begin{equation}\label{eq:H-controls-path-integral}
H_\gamma
\ge
\int_0^{L_\gamma}\frac{v_\gamma(s)^q}{-v_\gamma'(s)}\,ds
\ge
\frac{q-1}{4\log 2}
\int_0^{L_\gamma}s\,v_\gamma(s)^{q-1}\,ds.
\end{equation}

For \(0<r<R\), define
\begin{equation*}\label{eq:first-exit-index}
\kappa_r
:=
\min\bigl\{0\le i\le m-1:\mathfrak r_R(x_{i+1})>r\bigr\},
\qquad
Z_r(\gamma):=V_{\kappa_r}.
\end{equation*}
The index is well-defined because \(\mathfrak r_R(x_0)=0\), whereas
\(\mathfrak r_R(x_m)>R\).  By minimality,
\begin{equation}\label{eq:first-exit-location}
d_\rho(o,x_{\kappa_r})
=
\mathfrak r_R(x_{\kappa_r})
\le r
<
\mathfrak r_R(x_{\kappa_r+1}).
\end{equation}
Thus \(Z_r>0\) is the voltage at the tail of the first edge by which the path
leaves the intrinsic ball \(B_r\).  The radial sequence
\(i\mapsto\mathfrak r_R(x_i)\) need not be monotone; only this first exit
is used.  If the head is a terminal copy, the last quantity in
\eqref{eq:first-exit-location} is the radial distance of the original exterior
endpoint.

\begin{lemma}\label{lem:radial-first-exit-hardy}
There exists \(c_q>0\), depending only on \(q\), such that
\begin{equation}\label{eq:radial-first-exit-hardy}
H_\gamma
\ge
c_q\int_0^R r\,Z_r(\gamma)^{q-1}\,dr
\end{equation}
for every sampled path.
\end{lemma}

\begin{proof}
Set
\[
J_\gamma
:=
\int_0^{L_\gamma}s\,v_\gamma(s)^{q-1}\,ds.
\]
By \eqref{eq:H-controls-path-integral},
\[
H_\gamma\ge\frac{q-1}{4\log 2}J_\gamma.
\]
Split the radii into
\[
E_1=\{0<r<R:s_{\kappa_r}\ge r/2\},
\qquad
E_2=\{0<r<R:s_{\kappa_r}<r/2\}.
\]
If \(r\in E_1\), monotonicity of \(v_\gamma\) gives
\(Z_r=v_\gamma(s_{\kappa_r})\le v_\gamma(r/2)\).  Since
\(L_\gamma>R\) by \eqref{eq:path-length-exceeds-radius},
\begin{equation}\label{eq:first-exit-short-radii}
\int_{E_1}rZ_r^{q-1}\,dr
\le
\int_0^Rr\,v_\gamma(r/2)^{q-1}\,dr
=
4\int_0^{R/2}s\,v_\gamma(s)^{q-1}\,ds
\le
4J_\gamma.
\end{equation}

If \(r\in E_2\) and \(\kappa_r=i\), then
\eqref{eq:radial-label-prefix-bound} gives
\[
r
<
\mathfrak r_R(x_{i+1})
\le
s_i+\ell_i
<
\frac r2+\ell_i,
\]
and hence \(r<2\ell_i\).  Let
\(E_{2,i}=\{r\in E_2:\kappa_r=i\}\).  Since \(\delta_i\le V_i\),
\begin{equation}\label{eq:first-exit-long-radii}
\int_{E_{2,i}}rZ_r^{q-1}\,dr
\le
V_i^{q-1}\int_0^{2\ell_i}r\,dr
=
2\ell_i^2V_i^{q-1}
\le
2\frac{\ell_i^2V_i^q}{\delta_i}.
\end{equation}
Summing \eqref{eq:first-exit-long-radii} over \(i\) and using
\eqref{eq:first-exit-short-radii},
\[
\int_0^R rZ_r^{q-1}\,dr
\le
4J_\gamma+2H_\gamma
\le
\left(2+\frac{16\log 2}{q-1}\right)H_\gamma.
\]
This proves \eqref{eq:radial-first-exit-hardy}.
\end{proof}

\subsection{Capacity bounds from first-exit paths}

We first prove a capacity bound that will be applied to path segments starting
at the first exits from intrinsic balls.  All notions of flow, strength, energy, and effective
conductance are those of Subsection~\ref{subsection:finite-flows}.

\begin{lemma}\label{lem:reciprocal-stopped-flow}
Let \((W,\widetilde\mu)\) be a finite connected weighted graph.  Let
\(g: W\rightarrow [0, \infty)\), let \(\theta=i_g\) be an acyclic unit current, and let
\(\mathbb P\) be the probability measure induced by a flow decomposition of \(\theta\) satisfying the
edge-marginal identity \eqref{eq:path-edge-marginal}.  Let \(A,K\subset W\)
be nonempty and disjoint.  For each path
\[
\gamma=(x_0,e_0,x_1,\ldots,e_{m-1},x_m),
\]
choose indices
\[
0\le\tau(\gamma)<\varsigma(\gamma)\le m
\]
such that
\[
x_{\tau(\gamma)}\in A,
\qquad
x_{\varsigma(\gamma)}\in K,
\qquad
g(x_{\tau(\gamma)})>0.
\]
Then
\begin{equation}\label{eq:reciprocal-stopped-flow}
\mathbb E_\gamma\frac1{g(x_{\tau})}
\le
C_{\mathrm{eff}}^W(A,K).
\end{equation}
\end{lemma}

\begin{proof}
For a path \(\gamma\), let \(\chi_\gamma\) be the unit flow carried by its
selected segment from \(\tau\) to \(\varsigma\): on the directed support of
\(\theta\),
\[
\chi_\gamma(e)
=
\mathbf 1_{\{e=e_i\text{ for some }\tau\le i<\varsigma\}},
\]
and extend \(\chi_\gamma\) antisymmetrically to the reverse orientations.
Thus
\[
\Div\chi_\gamma
=
\mathbf 1_{\{x_\tau\}}-\mathbf 1_{\{x_\varsigma\}}.
\]
Define the reweighted superposition
\begin{equation}\label{eq:reweighted-stopped-flow-definition}
\Phi(e)
:=
\int
\frac{\chi_\gamma(e)}{g(x_\tau)}\,d\mathbb P(\gamma).
\end{equation}
It follows that
\begin{equation*}\label{eq:reweighted-stopped-flow-divergence}
\Div\Phi(x)
=
\int
\frac{\mathbf 1_{\{x_\tau=x\}}-\mathbf 1_{\{x_\varsigma=x\}}}
{g(x_\tau)}\,d\mathbb P(\gamma).
\end{equation*}
Hence \(\Phi\) is a flow from \(A\) to \(K\), and its strength is
\begin{equation*}\label{eq:reweighted-stopped-flow-strength}
I(\Phi)
=
\mathbb E_\gamma\frac1{g(x_\tau)}
=:I.
\end{equation*}

Orient each edge in the directed support of \(\theta\), and sum each
unoriented edge once in that orientation.  Put
\[
\delta_e=g(\tail(e))-g(\head(e))>0,
\qquad
\theta_e=\widetilde\mu_e\delta_e.
\]
Let \(\Gamma_e\) be the set of paths whose selected segment uses \(e\).
By \eqref{eq:reweighted-stopped-flow-definition}, the Cauchy--Schwarz inequality, and
the edge-marginal identity,
\begin{equation*}
\Phi(e)^2=
\left(\int_{\Gamma_e}\frac{d\mathbb P}{g(x_\tau)}\right)^2
\le
\mathbb P(\Gamma_e)
\int_{\Gamma_e}\frac{d\mathbb P}{g(x_\tau)^2}
\le
\theta_e
\int_{\Gamma_e}\frac{d\mathbb P}{g(x_\tau)^2}.
\end{equation*}
Consequently, by Tonelli's theorem and telescoping along the selected segment,
\begin{equation*}
\mathcal D_W(\Phi)
=
\sum_e\frac{\Phi(e)^2}{\widetilde\mu_e}
\le
\sum_e\delta_e
\int_{\Gamma_e}\frac{d\mathbb P}{g(x_\tau)^2}=
\mathbb E_\gamma\left[
\frac{g(x_\tau)-g(x_\varsigma)}{g(x_\tau)^2}
\right]\\
\le
\mathbb E_\gamma\frac1{g(x_\tau)}
=I.
\end{equation*}
Proposition~\ref{prop:dirichlet-thomson} gives
\[
\frac{I^2}{C_{\mathrm{eff}}^W(A,K)}
\le
\mathcal D_W(\Phi)
\le I,
\]
and \eqref{eq:reciprocal-stopped-flow} follows.
\end{proof}

We first apply the lemma to the terminal-copy weighted graph
\((\widehat V_R,\widehat\mu_R)\) with the terminal set as the outer plate.
For \(0<r<R\), the variational definition
\eqref{eq:relative-capacity-profile} agrees with the corresponding effective
conductance:
\begin{equation}\label{eq:relative-capacity-terminal-identification}
\capacity_{B_R}(B_r)
=
C_{\mathrm{eff}}^{\widehat V_R}(B_r,\mathcal T_R).
\end{equation}

\begin{lemma}\label{lem:relative-capacity-reciprocal}
For every \(0<r<R\),
\begin{equation}\label{eq:relative-capacity-reciprocal}
\mathbb E_\gamma Z_r^{-1}
\le
\capacity_{B_R}(B_r),
\end{equation}
and consequently
\begin{equation}\label{eq:relative-capacity-positive-moment}
\mathbb E_\gamma Z_r^{q-1}
\ge
\capacity_{B_R}(B_r)^{1-q}.
\end{equation}
\end{lemma}

\begin{proof}
For each path, choose \(\tau=\kappa_r\) and \(\varsigma=m\).  By
\eqref{eq:first-exit-location}, the initial vertex lies in \(B_r\) and has
voltage \(Z_r>0\), while \(x_m\in\mathcal T_R\).  Applying
Lemma~\ref{lem:reciprocal-stopped-flow} with
\[
A=B_r,
\qquad
K=\mathcal T_R,
\]
and using \eqref{eq:relative-capacity-terminal-identification} gives
\eqref{eq:relative-capacity-reciprocal}.

H\"older's inequality gives
\[
1
=
\mathbb E_\gamma
\left[Z_r^{(q-1)/q}Z_r^{-(q-1)/q}\right]
\le
\bigl(\mathbb E_\gamma Z_r^{q-1}\bigr)^{1/q}
\bigl(\mathbb E_\gamma Z_r^{-1}\bigr)^{(q-1)/q}.
\]
Since \(1-q<0\), \eqref{eq:relative-capacity-reciprocal} implies
\eqref{eq:relative-capacity-positive-moment}.
\end{proof}

\begin{remark}\label{rem:returning-relative-capacity}
After the selected segment begins at the first exit from \(B_r\), the path may
revisit \(B_r\).  This causes no difficulty: every intermediate visit has zero
divergence, so the superposition is still a flow from \(B_r\) to
\(\mathcal T_R\).
\end{remark}

We next recover the annular estimate.  For \(0<r<R/2\), set
\begin{equation*}\label{eq:annular-outer-plate-terminal-graph}
K_{r,R}
:=
\bigl\{z\in B_R\cup\mathcal T_R:\mathfrak r_R(z)>2r\bigr\}.
\end{equation*}
After wiring \(B_r\) and \(K_{r,R}\) separately, the resulting
two-terminal weighted graph consists of \(B_{2r}\), all original edges
incident with \(B_{2r}\), and one wired exterior vertex.  Therefore
\begin{equation*}
C_{\mathrm{eff}}^{\widehat V_R}(B_r,K_{r,R})
=
\mathcal C_o(r).
\end{equation*}
Since every terminal vertex has radial label larger than \(R>2r\), one has
\(\mathcal T_R\subset K_{r,R}\).  Enlarging the zero plate in Dirichlet's
principle and using \eqref{eq:relative-capacity-terminal-identification} gives
\begin{equation}\label{eq:relative-annular-comparison}
\capacity_{B_R}(B_r)
\le
\mathcal C_o(r),
\qquad 0<r<R/2.
\end{equation}

\subsection{The relative-capacity estimate and its consequences}
We first record the following consequence of Lemma~\ref{lemma6.1}, which will be used repeatedly to convert divergence of the Green energy into a Liouville conclusion.
\begin{lemma}\label{lem:green-energy-blowup}
Assume \(q>1\), \(0<\sigma\in\ell^+(V)\), and let \(d_\rho\) be
a complete \(\nu\)-adapted path metric. Fix \(o\in V\), and define
\(L_R^\nu\) as above. Suppose that
\(L_{R_j}^\nu\to\infty\) for some \(R_j\uparrow\infty\).
Then every nonnegative solution of \eqref{di-2} is identically zero.
\end{lemma}

\begin{proof}
By Lemma~\ref{lemma6.1},
\[
\nu(o)u(o)^q
\le
\left(\frac q{q-1}\right)^{q/(q-1)}
\bigl(L_{R_j}^\nu\bigr)^{-1/(q-1)}.
\]
Hence \(u(o)=0\). Since \(u\ge0\) and \(-\Delta u(o)\ge0\),
every neighbor of \(o\) also vanishes. Connectedness gives \(u\equiv0\).
\end{proof}
\begin{proposition}\label{prop:relative-annular-estimates}
There exists \(c_q>0\), depending only on \(q\), such that, for every
\(R>0\),
\begin{equation}\label{eq:relative-capacity-estimate}
L_R^\nu
\ge
c_q\int_0^R
r\,\capacity_{B_R}(B_r)^{1-q}\,dr.
\end{equation}
Consequently,
\begin{equation}\label{eq:annular-conductance-estimate}
L_R^\nu
\ge
c_q\int_0^{R/2}
r\,\mathcal C_o(r)^{1-q}\,dr.
\end{equation}
\end{proposition}

\begin{proof}
By Lemmas~\ref{lem:intrinsic-path-energy} and
\ref{lem:radial-first-exit-hardy}, followed by Tonelli's theorem,
\[
L_R^\nu
\ge
\mathbb E_\gamma H_\gamma
\ge
c_q\,
\mathbb E_\gamma
\int_0^R r\,Z_r(\gamma)^{q-1}\,dr
=
c_q\int_0^R
r\,\mathbb E_\gamma\!\left[Z_r^{q-1}\right]\,dr.
\]
For every \(0<r<R\), Lemma~\ref{lem:relative-capacity-reciprocal} gives
\[
\mathbb E_\gamma\!\left[Z_r^{q-1}\right]
\ge
\capacity_{B_R}(B_r)^{1-q}.
\]
Substituting this estimate into the preceding inequality proves
\eqref{eq:relative-capacity-estimate}.

Now let \(0<r<R/2\). By \eqref{eq:relative-annular-comparison},
\[
\capacity_{B_R}(B_r)\le \mathcal C_o(r).
\]
Since \(1-q<0\), it follows that
\[
\capacity_{B_R}(B_r)^{1-q}
\ge
\mathcal C_o(r)^{1-q}.
\]
Restricting the integral in
\eqref{eq:relative-capacity-estimate} to \(0<r<R/2\) therefore yields
\[
L_R^\nu
\ge
c_q\int_0^{R/2}
r\,\capacity_{B_R}(B_r)^{1-q}\,dr
\ge
c_q\int_0^{R/2}
r\,\mathcal C_o(r)^{1-q}\,dr,
\]
which proves \eqref{eq:annular-conductance-estimate}.
\end{proof}

The first consequence is a direct conductance criterion.

\begin{corollary}\label{cor:annular-conductance-liouville}
Let \(q>1\), let \(0<\sigma\in\ell^+(V)\), and set \(\nu=\sigma\mu\).  Let
\(\rho\) be a \(\nu\)-adapted edge-length function and assume that
\((V,d_\rho)\) is complete.  For \(o\in V\), let
\[
\mathcal C_o(r)
=
C_{\mathrm{eff}}\bigl(B_{d_\rho}(o,r),
V\setminus B_{d_\rho}(o,2r)\bigr).
\]
If
\begin{equation*}\label{eq:annular-conductance-divergence}
\int_1^\infty r\,\mathcal C_o(r)^{1-q}\,dr
=\infty,
\end{equation*}
then every nonnegative solution of \eqref{di-2} is identically zero.
\end{corollary}

\begin{proof}
For \(R>2\), Proposition~\ref{prop:relative-annular-estimates} gives
\[
L_R^\nu
\ge
c_q\int_1^{R/2}r\,\mathcal C_o(r)^{1-q}\,dr
\longrightarrow\infty
\qquad\text{as }R\to\infty.
\]
The conclusion follows from Lemma~\ref{lem:green-energy-blowup}.
\end{proof}

The second consequence is obtained by estimating the annular conductance in
terms of intrinsic volume.

\begin{lemma}\label{lem:intrinsic-annular-cutoff}
For every \(r>0\),
\begin{equation*}\label{eq:intrinsic-annular-cutoff}
\mathcal C_o(r)
\le
\frac{\nu(B_{2r})}{r^2}.
\end{equation*}
\end{lemma}

\begin{proof}
Use the radial cutoff
\[
\eta_r(x)=
\begin{cases}
1, & d_\rho(o,x)\le r,\\[1mm]
\dfrac{2r-d_\rho(o,x)}r, & r<d_\rho(o,x)<2r,\\[2mm]
0, & d_\rho(o,x)\ge2r.
\end{cases}
\]
The scalar cutoff is \(r^{-1}\)-Lipschitz, and the path-metric property gives
\[
|\eta_r(x)-\eta_r(y)|
\le
\frac{d_\rho(x,y)}r
\le
\frac{\rho(x,y)}r,
\qquad x\sim y.
\]
By Dirichlet's principle, only edges having at least one endpoint in
\(B_{2r}\) can contribute.  Hence, using \eqref{eq:adapted-weight},
\begin{align*}
\mathcal C_o(r)
&\le
\sum_{\{x,y\}\in E}\mu_{xy}
\bigl(\eta_r(x)-\eta_r(y)\bigr)^2\\
&\le
\frac1{r^2}
\sum_{x\in B_{2r}}\sum_{y\sim x}\mu_{xy}\rho(x,y)^2\\
&\le
\frac1{r^2}\sum_{x\in B_{2r}}\nu(x)
=
\frac{\nu(B_{2r})}{r^2}.
\end{align*}
\end{proof}

\begin{proposition}\label{prop:intrinsic-green-volume-lower}
There exists \(c_q>0\), depending only on \(q\), such that
\eqref{eq:intrinsic-green-volume-lower} holds for every \(R>0\).
\end{proposition}

\begin{proof}
Since \(\sigma(o)>0\), one has \(\nu(B_r)\ge\nu(o)>0\) for every \(r>0\).
Proposition~\ref{prop:relative-annular-estimates} and
Lemma~\ref{lem:intrinsic-annular-cutoff} give
\[
L_R^\nu
\ge
c_q\int_0^{R/2}
 r\left(\frac{r^2}{\nu(B_{2r})}\right)^{q-1}\,dr
=
c_q\int_0^{R/2}
\frac{r^{2q-1}}{\nu(B_{2r})^{q-1}}\,dr.
\]
The change of variables \(s=2r\), followed by a change of the constant,
proves \eqref{eq:intrinsic-green-volume-lower}.
\end{proof}

We are ready to establish the integral volume-growth criterion.
\begin{proof}[Proof of Theorem~\ref{thm:weighted-liouville}]
Let \(u\ge0\) satisfy \eqref{di-2}.  For \(R>1\),
Proposition~\ref{prop:intrinsic-green-volume-lower} gives
\[
L_R^\nu
\ge
c_q\int_1^R
\frac{r^{2q-1}}{\nu(B_r)^{q-1}}\,dr.
\]
By \eqref{eq:weighted-volume-divergence}, the right-hand side tends to
infinity as \(R\to\infty\). The preceding estimate implies \(L_R^\nu\to\infty\);
the conclusion follows from Lemma~\ref{lem:green-energy-blowup}. 
\end{proof}

\begin{proof}[Proof of Corollary~\ref{cor:unweighted-liouville}]
Take \(\sigma\equiv1\) and \(\rho\equiv1\).  Then \(\nu=\mu\), and
\[
\sum_{y\sim x}\mu_{xy}\rho(x,y)^2
=
\mu(x)
=
\nu(x),
\]
so \(\rho\) is \(\mu\)-adapted.  Moreover, \(d_\rho\) is the graph distance,
which is complete.

For every integer \(n\ge1\) and every \(r\in[n,n+1)\),
\(B_{d_\rho}(o,r)=B(o,n)\).  Therefore
\[
\frac{n^{2q-1}}{\mu(B(o,n))^{q-1}}
\le
\int_n^{n+1}
\frac{r^{2q-1}}{\mu(B_{d_\rho}(o,r))^{q-1}}\,dr
\le
2^{2q-1}\frac{n^{2q-1}}{\mu(B(o,n))^{q-1}}.
\]
It follows that the integral in
\eqref{eq:weighted-volume-divergence} diverges if and only if the series in
\eqref{eq:unweighted-volume-divergence} diverges.  The conclusion follows
from Theorem~\ref{thm:weighted-liouville}.
\end{proof}

\subsection{Capacity to infinity and  Nash--Williams type criteria}

For a finite set \(A\subset V\), define its capacity to infinity by
\begin{equation*}\label{eq:global-capacity-definition}
\capacity_V(A)
:=
\inf\left\{
\mathcal E(f):
 f\in\ell_0(V),\quad f\ge1\text{ on }A
\right\}.
\end{equation*}
By truncation, one may equivalently require \(f=1\) on \(A\) and
\(0\le f\le1\).  Thus \(\capacity_V(B_r)\) may be interpreted as the effective conductance from \(B_r\) to infinity.

\begin{lemma}\label{lem:capacity-exhaustion}
Fix \(r>0\).  As \(R\to\infty\) through any increasing sequence of radii
with \(R>r\),
\begin{equation*}\label{eq:capacity-exhaustion}
\capacity_{B_R}(B_r)
\downarrow
\capacity_V(B_r).
\end{equation*}
\end{lemma}

\begin{proof}
If \(R_2>R_1>r\), every function admissible for
\(\capacity_{B_{R_1}}(B_r)\), extended by zero, is admissible for
\(\capacity_{B_{R_2}}(B_r)\).  Hence the relative capacities decrease in
\(R\).  Every relative admissible function is finitely supported and
admissible for \(\capacity_V(B_r)\), so
\[
\capacity_{B_R}(B_r)\ge\capacity_V(B_r).
\]

Conversely, let \(\varepsilon>0\).  Choose \(f\in\ell_0(V)\), with
\(f\ge1\) on \(B_r\), such that
\[
\mathcal E(f)\le\capacity_V(B_r)+\varepsilon.
\]
After truncation, we may assume that \(f=1\) on \(B_r\) and
\(0\le f\le1\).  For all sufficiently large \(R\), one has
\(\operatorname{supp}f\subset B_R\), and therefore
\[
\capacity_{B_R}(B_r)
\le
\mathcal E(f)
\le
\capacity_V(B_r)+\varepsilon.
\]
Letting \(\varepsilon\downarrow0\) proves the assertion.
\end{proof}

Let
\[
g(o,x):=\lim_{R\to\infty}g_R(x)\in(0,\infty]
\]
and define the extended Green energy
\begin{equation*}\label{eq:global-green-energy}
L^\nu(o)
:=
\sum_{x\in V}g(o,x)^q\nu(x)
\in[0,\infty].
\end{equation*}
On a non-parabolic graph this is the ordinary whole-graph Green energy.  On a
parabolic graph it is understood in the extended sense.

\begin{proposition}\label{prop:global-green-capacity-lower}
With the convention \(0^{1-q}=+\infty\),
\begin{equation*}\label{eq:global-green-capacity-lower}
L^\nu(o)
\ge
c_q\int_0^\infty r\,\capacity_V(B_r)^{1-q}\,dr.
\end{equation*}
\end{proposition}

\begin{proof}
Choose \(R_j\uparrow\infty\).  The killed Green functions increase pointwise,
\[
g_{B_{R_j}}(o,x)\uparrow g(o,x),
\]
and hence \(L_{R_j}^\nu\uparrow L^\nu(o)\) by monotone convergence.  For
\(r>0\), set
\[
F_j(r)
:=
\mathbf 1_{\{r<R_j\}}\,
r\,\capacity_{B_{R_j}}(B_r)^{1-q}.
\]
Lemma~\ref{lem:capacity-exhaustion} and \(1-q<0\) show that \(F_j(r)\)
increases pointwise to
\[
r\,\capacity_V(B_r)^{1-q}.
\]
Apply Proposition~\ref{prop:relative-annular-estimates} and then
monotone convergence in \(r\).
\end{proof}

\begin{corollary}\label{cor:global-capacity-liouville}
If
\begin{equation*}\label{eq:global-capacity-divergence}
\int_1^\infty r\,\capacity_V(B_r)^{1-q}\,dr
=\infty,
\end{equation*}
then every nonnegative solution of \eqref{di-2} is identically zero.
\end{corollary}

\begin{proof}
Proposition~\ref{prop:global-green-capacity-lower} gives
\(L^\nu(o)=\infty\), and therefore \(L_R^\nu\to\infty\).
Then the conclusion follows from Lemma~\ref{lem:green-energy-blowup}. 
\end{proof}

\begin{remark}\label{rem:capacity-comparison-warning}
For \(0<2r<R\),
\begin{equation*}\label{eq:global-relative-annular-chain}
\capacity_V(B_r)
\le
\capacity_{B_R}(B_r)
\le
\mathcal C_o(r).
\end{equation*}
Thus the global capacity criterion is formally stronger than the annular one
and may detect bottlenecks far beyond the scale \(r\). 
\end{remark}

We finally specialize to the graph distance.  Assume \(\sigma\equiv1\) and
\(\rho\equiv1\), write \(B_n=B(o,n)\), and set
\begin{equation*}\label{eq:metric-cut-conductance}
b_k
:=
\sum_{\substack{x\in B_k,\ y\notin B_k\\x\sim y}}\mu_{xy}.
\end{equation*}
We have the following Nash--Williams estimate.
\begin{lemma}\label{lem:nash-williams-capacity}
For integers \(1\le n< R\),
\begin{equation}\label{eq:relative-capacity-nash-williams}
\capacity_{B_R}(B_n)^{-1}
\ge
\sum_{k=n}^R\frac1{b_k}.
\end{equation}
Consequently,
\begin{equation}\label{eq:global-capacity-nash-williams}
\capacity_V(B_n)^{-1}
\ge
\sum_{k=n}^\infty\frac1{b_k}.
\end{equation}
\end{lemma}

\begin{proof}
For \(k\ge n\), we define the cutset
\[
\Pi_k
:=
\bigl\{\{x,y\}\in E:x\in B_k,\ y\notin B_k\bigr\}.
\]
It is clear that \(\Pi_k\) are pairwise disjoint.
Let \(\Phi\) be any unit flow from \(B_n\) to \(\mathcal T_R\), and orient
each edge of \(\Pi_k\) outward from \(B_k\).  Summing \(\Div\Phi\) over
\(B_k\) shows that the net flux through \(\Pi_k\) is one.  Hence the Cauchy--Schwarz inequality gives
\[
\sum_{e\in\Pi_k}\frac{\Phi(e)^2}{\mu_e}
\ge
\frac{\bigl(\sum_{e\in\Pi_k}\Phi(e)\bigr)^2}
{\sum_{e\in\Pi_k}\mu_e}
=
\frac1{b_k}.
\]
Since the cutsets are disjoint,
\[
\mathcal D_{\widehat V_R}(\Phi)
\ge
\sum_{k=n}^R\frac1{b_k}.
\]
Taking the infimum over unit flows and using Thomson's principle (Proposition \ref{prop:dirichlet-thomson}) together with
\eqref{eq:relative-capacity-terminal-identification} proves
\eqref{eq:relative-capacity-nash-williams}.  Letting \(R\to\infty\) and using
Lemma~\ref{lem:capacity-exhaustion} proves
\eqref{eq:global-capacity-nash-williams}.
\end{proof}

\begin{corollary}\label{cor:nash-williams-cut-liouville}
If
\begin{equation*}\label{eq:nash-williams-cut-tail-divergence}
\sum_{n=1}^\infty
n\left(\sum_{k=n}^\infty\frac1{b_k}\right)^{q-1}
=\infty,
\end{equation*}
then every nonnegative solution of \eqref{sdi} is identically zero.
\end{corollary}

\begin{proof}
For \(r\in[n,n+1)\), one has \(B_r=B_n\).  Therefore
\eqref{eq:global-capacity-nash-williams} gives
\begin{equation*}
\int_1^\infty r\,\capacity_V(B_r)^{1-q}\,dr
\ge
\sum_{n=1}^\infty n\,\capacity_V(B_n)^{1-q}
\ge
\sum_{n=1}^\infty
n\left(\sum_{k=n}^\infty\frac1{b_k}\right)^{q-1}.
\end{equation*}
The conclusion follows from Corollary~\ref{cor:global-capacity-liouville}.
\end{proof}

The strict positivity assumption in
Theorem~\ref{thm:weighted-liouville} is tied to its symmetric
intrinsic-metric formulation. The following oriented formulation treats degenerate
potentials without forcing the geometry to collapse.

\begin{remark}\label{rem:degenerate-potentials}
Let \(0\not\equiv\sigma:V\to[0,\infty)\), set \(\nu=\sigma\mu\), and assign a
possibly asymmetric length \(\lambda(x,y)\ge0\) to every ordered adjacent pair,
subject to
\begin{equation*}\label{eq:oriented-adapted-summary}
\sum_{y\sim x}\mu_{xy}\lambda(x,y)^2\le\nu(x),
\qquad x\in V.
\end{equation*}
Define the forward path distance and its balls by
\[
d_\lambda^+(x,y)
:=
\inf_{x=x_0\sim\cdots\sim x_m=y}
\sum_{i=0}^{m-1}\lambda(x_i,x_{i+1}),
\qquad
B_\lambda^+(o,r):=\{x:d_\lambda^+(o,x)\le r\}.
\]
Assume that all forward balls are finite and choose \(o\in\supp\sigma\).
It can be shown that \[
\int_1^\infty
\frac{r^{2q-1}}
{\nu(B_\lambda^+(o,r))^{q-1}}\,dr=\infty
\]
forces every nonnegative solution of \eqref{di-2} to vanish identically.
 More
basically, with
\[
\mathcal C_{o,\lambda}^+(r)
:=C_{\mathrm{eff}}\bigl(B_\lambda^+(o,r),
V\setminus B_\lambda^+(o,2r)\bigr),
\]
the divergence condition
\[
\int_1^\infty r\,\mathcal C_{o,\lambda}^+(r)^{1-q}\,dr=\infty
\]
has the same consequence.  A canonical exactly adapted choice is
\[
\lambda_\sigma(x,y)=\sqrt{\sigma(x)}.
\]
The proofs follow the same scheme of this section and we omit the details.
\end{remark}

\section[Green energies and \texorpdfstring{$L^q$}{Lq}-Liouville properties]{Green energies and \texorpdfstring{$L^q$}{Lq}-Liouville properties}\label{sec-example}

In this section, we explain how the Green quantities used in the preceding
sections are related to the \(L^q\)-Liouville property.  We also point out a
basic distinction: the \(L^q\)-Liouville property is controlled by a diagonal
Green energy, whereas the existence of positive solutions to the semilinear
inequality is controlled by a pointwise testing condition.

Let \(p=q/(q-1)\).  In the terminology of Hao and Sun \cite{HS2025}, a graph
has the \(L^q\)-Liouville property if every nonnegative superharmonic function
in \(\ell^q(V,\mu)\) is identically zero.  They introduce the \(L^q\)-Green
function
\[
	g_q(x,y):=
	\sum_{z\in V}g(x,z)g(z,y)^{q-1}\mu(z).
\]
In particular,
\[
	g_q(o,o)=\sum_{z\in V}g(o,z)^q\mu(z).
\]
By \cite[Theorem~4.2]{HS2025}, for \(1<p<\infty\), the following conditions are
equivalent: \(L^p\)-parabolicity, the \(L^q\)-Liouville property, and
\(g_q(x,y)=\infty\) for some, equivalently all, \(x,y\in V\).

In the unweighted specialization of Section~\ref{sec:intrinsic-flow}, the
finite-domain Green energy has the following potential-theoretic interpretation.
Indeed, when \(\sigma\equiv1\) and \(B_R=B(o,R)\),
\begin{equation}\label{eq:LR-gq-approx}
	L_R^\mu
	:=
	\sum_{x\in B_R}g_{B_R}(o,x)^q\mu(x)
\end{equation}
is the finite-domain approximation of the diagonal energy \(g_q(o,o)\).
Allowing the value \(+\infty\), one has
\(g_{B_R}(o,x)\uparrow g(o,x)\) for every \(x\in V\), and hence
\(L_R^\mu\uparrow g_q(o,o)\) by monotone convergence.

As a consequence, Corollary~\ref{cor:unweighted-liouville} has the following interpretation.

\begin{corollary}\label{cor:Lq-Liouville-volume}
	Let \(1<q<\infty\), and set \(p=q/(q-1)\).  If, for some \(o\in V\),
	\begin{equation}\label{eq:Lq-Liouville-volume-condition}
		\sum_{n=1}^{\infty}
		\frac{n^{2q-1}}{\mu(B(o,n))^{q-1}}
		=\infty,
	\end{equation}
	then \((V,\mu)\) has the \(L^q\)-Liouville property.  Equivalently, in the
	terminology of \cite{HS2025}, \((V,\mu)\) is \(L^p\)-parabolic.
\end{corollary}

\begin{proof}
	By Proposition~\ref{prop:intrinsic-green-volume-lower}, applied with
	\(\sigma\equiv1\) and \(\rho\equiv1\), condition
	\eqref{eq:Lq-Liouville-volume-condition} implies that
\(L_R^\mu\to\infty\).  By
	\eqref{eq:LR-gq-approx}, this gives \(g_q(o,o)=\infty\).  The conclusion follows
	from \cite[Theorem~4.2]{HS2025}.
\end{proof}

We next compare the \(L^q\)-Liouville property with the nonexistence of positive
solutions to the semilinear inequality
\begin{equation}\label{eq:unweighted-Lane-Emden-for-section}
	-\Delta u\ge u^q.
\end{equation}
For the unweighted inequality \eqref{eq:unweighted-Lane-Emden-for-section}, the
implication from the \(L^q\)-Liouville property to semilinear nonexistence is
already contained in Theorem \ref{thm:green-level-criterion}.  Indeed, by the
equivalence quoted above, the \(L^q\)-Liouville property gives
\(g_q(o,o)=\infty\) for every \(o\in V\), that is,
\[
\sum_{x\in V}g(o,x)^q\mu(x)=\infty.
\]
On the other hand, if \eqref{eq:unweighted-Lane-Emden-for-section} admitted a
positive solution, then Theorem \ref{thm:green-level-criterion}, applied with
\(\sigma\equiv1\) and \(\nu=\mu\), would force this same diagonal Green energy to
be finite.  This contradiction rules out positive solutions.

The reverse implication is false  for the unweighted inequality
\eqref{eq:unweighted-Lane-Emden-for-section}. See the following example.

\begin{example}
	\label{ex:unweighted-nonexistence-not-Lq}
	Fix $q>1$ and choose a number
	\[
	\frac{q-1}{2}<\beta<q.
	\]
	Let $b\ge 2$, and let $T_b$ be the rooted regular tree with branching number $b$.
	Let $o$ be its root, and let $(z_n)_{n\ge1}$ be a sequence with $o\sim z_1\sim \cdots\sim z_n\sim z_{n+1}\sim \cdots$, so that $d(o,z_n)=n$.  For each $n$, attach to $z_n$ a path
	\[
	z_n=w_{n,0}\sim w_{n,1}\sim\cdots\sim w_{n,\ell_n},
	\qquad
	\ell_n:=\lfloor b^{\beta n}\rfloor.
	\]
	All edges have weight one.  Denote the resulting graph by $(V,\mu)$. In this case
	$\mu(x)$ is the degree of $x$.
	
	Let
	\[
	h(x):=g(o,x).
	\]
	Note that $h$ is constant
	on each attached path by harmonicity, and on the tree part it coincides with the usual Green
	function of the regular tree. In particular,
	\[
		h(w_{n,j})=h(z_n)\asymp b^{-n},
		\qquad 0\le j\le \ell_n.
	\]
	It follows that the diagonal $L^q$-Green energy is finite.  Indeed,
	\[
	\sum_{x\in T_b}h(x)^q\mu(x)
	\lesssim
	\sum_{m=0}^\infty b^m b^{-qm}<\infty,
	\]
	and the contribution of the attached paths is bounded by
	\[
2	\sum_{n=1}^\infty \ell_n b^{-qn}
	\lesssim
	\sum_{n=1}^\infty b^{-(q-\beta)n}<\infty.
	\]
	Thus
	\[
		\sum_{x\in V}g(o,x)^q\mu(x)<\infty.
	\]
	By the equivalence between the $L^q$-Liouville property and divergence of the
	$L^q$-Green function, this graph does not have the $L^q$-Liouville property.
	
	On the other hand, the pointwise testing condition fails.  Set
	$x_n:=w_{n,\ell_n}$.  Since $h$ is constant on the attached path at $z_n$, we have
	$h(y)=h(z_n)$ for all $y\in\{w_{n,1},\ldots,w_{n,\ell_n}\}$.  Therefore
	\[
	G(h^q)(x_n)
	\ge
	h(z_n)^q
	\sum_{j=1}^{\ell_n}g(x_n,w_{n,j})\mu(w_{n,j}).
	\]
	The last sum is the expected time spent on the attached path by the random walk
	started at $x_n$.  Before hitting $z_n$, this is just the simple random walk on
	$\{0,1,\ldots,\ell_n\}$, started at $\ell_n$, with $0$ absorbing and $\ell_n$
	reflecting.  Its expected hitting time of $0$ is $\ell_n^2$.  Hence
	\[
	G(h^q)(x_n)
	\gtrsim
	b^{-qn}\ell_n^2.
	\]
	Dividing by $h(x_n)\asymp b^{-n}$ gives
	\[
	\frac{G(h^q)(x_n)}{h(x_n)}
	\gtrsim
	b^{-(q-1)n}\ell_n^2
	\asymp
	b^{(2\beta-q+1)n}\longrightarrow\infty,
	\]
	by the choice of $\beta$.  Therefore there is no constant $C$ such that
	\[
	G(g(o,\cdot)^q)(x)\le Cg(o,x),\qquad x\in V.
	\]
	By Theorem \ref{thm:green-testing-equivalence}, the inequality
	\[
	-\Delta u\ge u^q
	\]
	has no positive solution on $V$.
\end{example}

\begin{remark}
	The regular tree part already has exponential
	volume growth and fails \textnormal{(VD)}.  Thus there is no conflict with Corollary \ref{cor:unweighted-one-condition}.
\end{remark}

As mentioned in Remark \ref{rem:relation-Green-energy}, it is also natural to ask whether, under \textnormal{$(3G)$}, the Green energy
condition \[
\sum_{x\in V}g(o,x)^q\nu(x)<\infty
\] implies the full testing condition
\begin{equation}\label{eq:section-pointwise-testing}
        G\big(\sigma g(o,\cdot)^q\big)(x)
        \lesssim g(o,x),
        \qquad x\in V.
\end{equation}
The following example shows that this is not the case even if we assume \textnormal{(VD)},
\textnormal{(PI)}, and \textnormal{(P$_0$)}.

\begin{example}\label{ex:diagonal-energy-not-testing}
	Let \(d\ge3\), and consider \(\mathbb Z^d\) with the standard nearest-neighbor
	weights. Clearly \textnormal{(VD)},
	\textnormal{(PI)}, and \textnormal{(P$_0$)} all hold. Fix
	\[
	q>\frac{d}{d-2}.
	\]
	Note that 	\(g(0,x)\asymp (1+|x|)^{2-d}\).

	Choose vertices \(z_k\to\infty\) so rapidly that
	\[
	\sum_k g(0,z_k)^{1/2}<\infty.
	\]
	Define a measure \(\lambda\) by
	\[
	\lambda:=\mathbf{1}_{\{0\}}+
	\sum_k g(0,z_k)^{1/2}\mathbf{1}_{\{z_k\}}.
	\]
	Let \(\nu\) be defined by
	\[
\nu(x)=\frac{\lambda(x)}{	g(0,x)^q},
	\]
and	\(\sigma(x)=\nu(x)/\mu(x)\).  Then \(\sigma\in\ell^+(V)\), \(\sigma(0)>0\), and
	\[
	\sum_{x\in V}g(0,x)^q\nu(x)
	=\lambda(V)<\infty.
	\]
	However, the pointwise testing condition fails.  Indeed, at \(x=z_k\),
	\[
	G\bigl(\sigma g(0,\cdot)^q\bigr)(z_k)
	=\sum_{y\in V}g(z_k,y)\lambda(y)
	\ge g(z_k,z_k)\lambda(z_k)
	\asymp g(0,z_k)^{1/2}.
	\]
	Since \(g(0,z_k)\to0\), we have
	\[
	\frac{G\bigl(\sigma g(0,\cdot)^q\bigr)(z_k)}{g(0,z_k)}
	\gtrsim g(0,z_k)^{-1/2}\to\infty.
	\]
	Thus \eqref{eq:section-pointwise-testing} fails.
\end{example}

The example isolates the main obstruction.  A diagonal Green energy condition such as
\(\sum_x g(o,x)^q\nu(x)<\infty\)  does not prevent the potential from concentrating near sparse points.
  To
deduce condition~\eqref{e1}, one needs an additional assumption controlling
such concentration.  Condition~\eqref{et2} provides one such assumption: it
rules out the concentration exhibited in
Example~\ref{ex:diagonal-energy-not-testing}.  A systematic study of other
hypotheses with this property lies beyond the scope of this paper.


\section[Existence and Nonexistence in Zd]{Existence and Nonexistence in \texorpdfstring{$\mathbb{Z}^d$}{Zd}}\label{sec7}
Throughout this section and the next, \(|\cdot|\) and \(\|\cdot\|_1\) denote the Euclidean norm and $\ell^1$-norm on $\mathbb{Z}^d$, respectively. We equip \(\mathbb Z^d\) with the standard
nearest-neighbor weights
\[
\mu_{xy}=
\begin{cases}
	1, & \|x-y\|_1=1,\\
	0, & \text{otherwise}.
\end{cases}
\]
The graph distance 
\(
d(x,y)\) coincides with the $\ell^1$-norm \(\|x-y\|_1
\). 

In this section, we study positive solutions on \(\mathbb Z^d\) of
\begin{equation}\label{eq-z}
	-\Delta u\ge \frac{1}{(1+|x|)^\alpha}u^q.
\end{equation}


A two-sided estimate for the Green function in $\mathbb{Z}^d$ is well known; see, for example, \cite{U98}. For completeness, we recall how it follows from the
general estimates above.
\begin{lemma}\label{lem:green-Zd}
In $\mathbb{Z}^d$ with $d\ge  3$, we have for any $x, y\in \mathbb{Z}^d$
\begin{align}\label{eq:green-Zd}
g(x,y)\asymp (1+d(x,y))^{2-d}.
\end{align}
\end{lemma}
\begin{proof}
From \cite{CG98}, we know that \textnormal{(PI)} holds in $\mathbb{Z}^d$.
Since in $\mathbb{Z}^d$ we have $\mu(B(x,n))\asymp (1+n)^d$, condition \textnormal{(VD)} is also satisfied. Lemma \ref{lem:Green-estimate} completes the proof.
\end{proof}
\begin{remark}
Since every nonnegative superharmonic function on $\mathbb{Z}^1$ or $\mathbb{Z}^2$ is constant, there is no positive solution to
(\ref{eq-z}) in $\mathbb{Z}^1$ and $\mathbb{Z}^2$. In what follows, we only consider the case $d\ge 3$.
\end{remark}

For completeness, we record the following observation outside the standing
superlinear range; it is not used below.
\begin{lemma}
If $\alpha\le 0$ and $0<q\le  1$, there is no positive solution to (\ref{eq-z}) in $\mathbb{Z}^d$.
\end{lemma}
\begin{proof}
This follows from \cite[Theorem 1.1]{GHS23}.
\end{proof}

\begin{theorem}\label{th-z-1}
	Let \(d\ge3\) and \(\alpha<2\).  If
	\[
	1<q\le \frac{d-\alpha}{d-2},
	\]
	then \eqref{eq-z} has no positive solution on \(\mathbb Z^d\).
\end{theorem}

\begin{proof}
In $\mathbb{Z}^d$, for all $x\in\mathbb{Z}^d$ and $n\in\mathbb{N}_{+}$, we have
\[\mu(B(x,n))\asymp (n+1)^d,\quad \mu(x)=2d.\]
Let $o$ be the origin. Using \eqref{eq:green-Zd} and \(1+|y|\asymp 1+d(y,o)\),
 we obtain
\begin{align}\label{est-I}
G(\sigma g(o,\cdot)^q)(x)&=\sum_{y\in \mathbb{Z}^d}g(x,y)\sigma(y)g(y,o)^q\mu(y)\nonumber\\
&=2d\sum_{y\in \mathbb{Z}^d}g(x,y)\frac{1}{(1+d(y,o))^{\alpha}}g(y,o)^q\nonumber\\
&\asymp \sum_{y\in \mathbb{Z}^d}\frac{1}{(1+d(x,y))^{d-2}}\frac{1}{(1+d(y,o))^{\alpha+q(d-2)}}:=I.
\end{align}
Note that
\begin{align*}
I&=\sum_{y\in\mathbb{Z}^d,\ 2d(y,o)> d(x,o)}
\frac{1}{(1+d(x,y))^{d-2}}\frac{1}{(1+d(y,o))^{\alpha+q(d-2)}}\nonumber\\
&\quad+\sum_{y\in\mathbb{Z}^d,\ 2d(y,o)\le  d(x,o)}
\frac{1}{(1+d(x,y))^{d-2}}\frac{1}{(1+d(y,o))^{\alpha+q(d-2)}}\nonumber\\
&:=I_1+I_2.\nonumber
\end{align*}
For $I_1$, since $d(x,y)\le  3d(y,o)$, we have
\begin{align}\label{I1-est-1}
I_1&\gtrsim \sum _{y\in\mathbb{Z}^d, 2d(y,o)> d(x,o)}\frac{1}{(1+d(y,o))^{\alpha+(d-2)(q+1)}}\nonumber\\
&\gtrsim  \sum_{n=d(x,o)+1}^{\infty}\frac{1}{(1+n)^{\alpha+(d-2)(q+1)-d+1}},
\end{align}
For $I_2$, since $d(x,y)\le  2d(x,o)$ when $d(x,o)>0$, we have
\begin{align}\label{I2-est-1}
I_2&\gtrsim \frac{1}{(1+d(x,o))^{d-2}}\sum_{y\in\mathbb{Z}^d, 2d(y,o)\le  d(x,o)}\frac{1}{(1+d(y,o))^{\alpha+q(d-2)}}\nonumber\\
&\gtrsim \frac{1}{(1+d(x,o))^{d-2}}\sum_{n=0}^{\frac{d(x,o)}{2}}\frac{1}{(1+n)^{\alpha+q(d-2)-d+1}}.
\end{align}

On the other hand, by Theorem \ref{thm:green-testing-equivalence}, if there exists a positive solution to (\ref{eq-z}), we must have
\begin{equation}\label{ne-1}
G(\sigma g(o,\cdot)^q)(x)\lesssim g(x,o).
\end{equation}

We now split the argument into three cases:  $1<q\le  \frac{2-\alpha}{d-2}$, $\frac{2-\alpha}{d-2}<q<\frac{d-\alpha}{d-2}$
and $q=\frac{d-\alpha}{d-2}$.

\textbf{Case 1.} If $1<q\le  \frac{2-\alpha}{d-2}$, the series in (\ref{I1-est-1}) diverges,
hence $I_1\equiv+\infty$, which contradicts (\ref{ne-1}).

\textbf{Case 2.} If $\frac{2-\alpha}{d-2}<q<\frac{d-\alpha}{d-2}$, it follows from (\ref{I1-est-1}) that
\[
I_1\gtrsim (1+d(x,o))^{-\alpha-q(d-2)+2},
\]
Since $g(x,o)\asymp (1+d(x,o))^{2-d}$, it follows from (\ref{ne-1})
that, necessarily, for all $x\in\mathbb{Z}^d$,
\[
(1+d(x,o))^{-\alpha-q(d-2)+2}\lesssim (1+d(x,o))^{2-d}
\]
This contradicts $q<\frac{d-\alpha}{d-2}$ as $d(x,o)\to\infty$.

\textbf{Case 3.} If $q=\frac{d-\alpha}{d-2}$, it follows from (\ref{I2-est-1}) that
\[
I_2\gtrsim \frac{1}{(1+d(x,o))^{d-2}}\ln (1+d(x,o)).
\]
This contradicts (\ref{ne-1}) as $d(x,o)\to\infty$.
This completes the proof.
\end{proof}

\begin{remark}
	The preceding proof uses Green testing.  Alternatively, flow decomposition
	gives the same nonexistence result through
	Theorem~\ref{thm:weighted-liouville}.  Let $o=0$, let
$\sigma(x)=(1+|x|)^{-\alpha}$, and let $\nu=\sigma\mu$.  For adjacent vertices
$x\sim y$, choose
\[
        \rho(x,y):=\min\{\sigma(x),\sigma(y)\}^{1/2}.
\]
Then
\[
        \sum_{y\sim x}\mu_{xy}\rho(x,y)^2
        \le 2d\,\sigma(x)=\nu(x),
\]
so $\rho$ is $\nu$-adapted.  Since $\rho(x,y)\asymp (1+d(o,x))^{-\alpha/2}$ for
$x\sim y$, one has
\[
        d_\rho(o,x)\asymp (1+d(o,x))^{1-\alpha/2},
        \qquad \alpha<2.
\]
Consequently $(\mathbb Z^d,d_\rho)$ is complete and
\[
        \nu(B_{d_\rho}(o,r))
        \asymp r^{\frac{2(d-\alpha)}{2-\alpha}}.
\]
The intrinsic volume integral in Theorem \ref{thm:weighted-liouville} is
therefore comparable to
\[
        \int_1^\infty
        r^{2q-1-\frac{2(d-\alpha)(q-1)}{2-\alpha}}\,dr,
\]
which diverges exactly when $q\le (d-\alpha)/(d-2)$.  Thus Theorem
\ref{thm:weighted-liouville} yields Theorem \ref{th-z-1}.
\end{remark}

\begin{theorem}\label{th-z-2}
	Suppose \(d\ge3\).  If either
	\[
	\alpha<2
	\qquad\text{and}\qquad
	q>\frac{d-\alpha}{d-2},
	\]
	or \(\alpha\ge2\) and \(q>1\), then \eqref{eq-z} admits a positive solution
	on \(\mathbb Z^d\).
\end{theorem}

\begin{proof}
By Theorem \ref{thm:green-testing-equivalence}, it suffices to prove (\ref{ne-1}). As in the proof of Theorem \ref{th-z-1}, we begin with (\ref{est-I})
\begin{align*}
I&=\sum_{y\in \mathbb{Z}^d,\ d(y,o)>2d(x,o)}
\frac{1}{(1+d(x,y))^{d-2}}\frac{1}{(1+d(y,o))^{\alpha+q(d-2)}}\\
&\quad+\sum_{y\in \mathbb{Z}^d,\ \frac{d(x,o)}{2}<d(y,o)\le 2d(x,o)}
\frac{1}{(1+d(x,y))^{d-2}}\frac{1}{(1+d(y,o))^{\alpha+q(d-2)}}\\
&\quad+\sum_{y\in \mathbb{Z}^d,\ d(y,o)\le \frac{d(x,o)}{2}}
\frac{1}{(1+d(x,y))^{d-2}}\frac{1}{(1+d(y,o))^{\alpha+q(d-2)}}\\
&:=I_3+I_4+I_5.
\end{align*}
For $I_3$, since $d(x,y)\ge \frac{d(y,o)}{2}$ and $q>\frac{d-\alpha}{d-2}$, we obtain
\begin{align}\label{I3}
I_3&\lesssim \sum_{n=2d(x,o)}^{\infty}\frac{(1+n)^{d-1}}{(1+n)^{\alpha+(q+1)(d-2)}}\nonumber\\
&\lesssim\frac{1}{(1+d(x,o))^{(q+1)(d-2)-d+\alpha}}\nonumber\\
&\lesssim \frac{1}{(1+d(x,o))^{d-2}}.
\end{align}
For $I_4$, since $d(y,o)\asymp d(x,o)$, we have
\begin{align}\label{I4}
I_4&\lesssim \frac{1}{(1+d(x,o))^{\alpha+q(d-2)}}
\sum_{d(x,y)\le  3d(x,o)}\frac{1}{(1+d(x,y))^{d-2}}\nonumber\\
&\lesssim\frac{1}{(1+d(x,o))^{\alpha+q(d-2)-2}}\nonumber\\
&\lesssim \frac{1}{(1+d(x,o))^{d-2}}.
\end{align}
For $I_5$, since $d(x,y)\ge \frac{d(x,o)}{2}$ and $q>\frac{d-\alpha}{d-2}$, we derive
\begin{align}\label{I5}
I_5&\lesssim \frac{1}{(1+d(x,o))^{d-2}}\sum_{y\in \mathbb{Z}^d,  d(y,o)\le \frac{d(x,o)}{2}}\frac{1}{(1+d(y,o))^{\alpha+q(d-2)}}\nonumber\\
&\lesssim \frac{1}{(1+d(x,o))^{d-2}}\sum_{n=0}^{\frac{d(x,o)}{2}}\frac{(1+n)^{d-1}}{(1+n)^{\alpha+q(d-2)}}\nonumber\\
&\lesssim \frac{1}{(1+d(x,o))^{d-2}}.
\end{align}
Combining (\ref{I3})-(\ref{I5}) with \eqref{eq:green-Zd}, we obtain (\ref{ne-1}). By Theorem \ref{thm:green-testing-equivalence}, there exist positive solutions to (\ref{eq-z}) in $\mathbb{Z}^d$.
\end{proof}

In summary, if \(d\ge3\) and \(q>1\), then
\[
-\Delta u\ge(1+|x|)^{-\alpha}u^q
\]
has a positive solution on \(\mathbb Z^d\) if and only if
\[
\alpha+q(d-2)>d.
\]

\section[Existence and Nonexistence in k-orthant]{Existence and Nonexistence in the \texorpdfstring{$k$}{k}-orthant}\label{sec-zk}

Fix integers \(d\ge3\) and \(1\le k\le d\). Here \(d\) is the ambient
dimension of the lattice and \(k\) is the number of constrained coordinates.
We consider the discrete \(k\)-orthant
\[
A_{d,k}:=\{x\in\mathbb Z^d:x_1\ge1,\dots,x_k\ge1\},
\]
and study the Lane--Emden type inequality \eqref{di-1} with Dirichlet boundary condition (cf. \cite{CHZ25} for related work on equations, where Serrin-type indices are also studied).
The case \(k=1\) is the discrete half-space, while \(k=d\) is the positive
orthant. When $d$ and $k$ are fixed, we write $A=A_{d,k}$ for short, and denote the corresponding Dirichlet Green function by $g_A(x,y)$.
We first establish the two-sided Green function estimate, then derive its
fixed-pole form, and finally apply Green testing to determine the critical
exponent.

First, we give a two-sided estimate of $g_A(x,y)$ when $|x-y|$ is large.
\begin{theorem}\label{thm:large}
There exists a constant $R_0 > 0$, depending only on $d$ and $k$, such that for all $x,y \in A$ with $|x-y| \ge R_0$,
\[
g_A(x,y)
\asymp |x-y|^{2-d}\prod_{i=1}^k\left(1\wedge \frac{x_i y_i}{|x-y|^2}\right).
\]
\end{theorem}

\begin{proof}
     Let $\iota_i:\mathbb{Z}^d \rightarrow \mathbb{Z}^d$ denote the reflection mapping in the $i$-th coordinate:
    \[
      \iota_i y = (y_1, \dots, -y_i, \dots, y_d).
   \]
  Set $\Gamma(z)=g(0, z)$ for $z\in \mathbb{Z}^d$. We use the asymptotic expansion of $\Gamma(z)$ of order $2k$ from \cite[Theorem 2]{U98}:
    \begin{align*}
    \Gamma(z)&= \left( \frac{c_d}{|z|^{d-2}} + \sum_{m=1}^{2k} \frac{U_m(z/|z|)}{|z|^{d-2+m}} \right) + o\left(\frac{1}{|z|^{d-2+2k}}\right)\\
    &:=F_{2k}(z) + E_{2k}(z),\qquad |z| \to \infty,
    \end{align*}
where $U_m$ are polynomials.

Since \(\Gamma\) is invariant under each coordinate reflection \(\iota_i\), we have
\[
    \Gamma(z)-\Gamma(\iota_i z)\equiv0.
\]
Comparing the asymptotic expansion term by term, using the uniqueness of
asymptotic expansions in powers of \(|z|^{-1}\), we obtain
\[
    U_m(z/|z|)=U_m(\iota_i z/|z|)
    \qquad 1\le m\le 2k.
\]
Hence  \(F_{2k}\in C^\infty(\mathbb R^d\setminus\{0\})\) is even in each of the first \(k\) variables. By the smooth even-function lemma (see \cite{Whit1943}), it follows
that there exists a function
\[\Phi\in C^\infty\!\left(
\bigl([0,\infty)^k\times\mathbb R^{d-k}\bigr)\setminus\{0\}
\right)\]
such that
\begin{align}\label{fandphi}
F_{2k}(z) = \Phi(z_1^2,\dots,z_k^2,z_{k+1},\dots,z_d).
\end{align}
Define
\[
\theta := \frac{d-2}{2}.
\]
The leading term of $F_{2k}$ as $|z|\to\infty$ is given by
\[
\frac{c_d}{|z|^{d-2}} = \frac{c_d}{(|z|^2)^{\theta}}=\frac{c_d}{(|w|^2+s_1+\cdots+s_k)^{\theta}}.
\]
Since \(F_{2k}\) is a finite sum of smooth homogeneous terms, its
representation in \eqref{fandphi} may be differentiated term by term; the
leading term gives the principal contribution, while the remaining terms have
the stated lower order. Consequently,
\begin{align}\label{eq:mixed-deriv-asym}
(-1)^k\partial_{s_1}\cdots\partial_{s_k}\Phi(s,w)
=& \frac{c_{d,k}}{\Bigl(|w|^2+s_1+\cdots+s_k\Bigr)^{\theta+k}}\nonumber\\
&+ O\Bigl(\frac{1}{(|w|^2+s_1+\cdots+s_k)^{\theta+k+\frac12}}\Bigr),
\end{align}
uniformly as $|w|^2+s_1+\cdots+s_k \to \infty$, where
\[
c_{d,k} = c_d\times(\theta)_k > 0,
\quad\text{and}\quad
(\theta)_k := \theta(\theta+1)\cdots(\theta+k-1).
\]

Next, for $1\le i\le k$, we introduce the following notation
\[
\ell_i := 4x_i y_i,
\qquad
a_i := (x_i-y_i)^2,
\]
and
\[
r := |x-y|,
\qquad w := (x_{k+1}-y_{k+1},\dots,x_d-y_d) \in \mathbb{R}^{d-k}.
\]

Let \(\mathcal R_k\) denote the abelian reflection group generated by $\{\iota_1, \dots, \iota_k\}$. Note that the reflections commute and $\iota_i^2 = \mathrm{id}$; hence every subset of indices $E \subseteq \{1, \dots, k\}$ uniquely determines an element $\iota \in \mathcal R_k$.
We denote this element by
\[
\iota_E = \prod_{i \in E} \iota_i,
\]
where $\iota_\emptyset = \mathrm{id}$. 

Applying the one-coordinate reflection identity successively in the \(k\)
commuting coordinates gives the \(k\)-orthant Green function
\[
    g_A(x,y) = \sum_{\iota_{E} \in \mathcal R_k} (-1)^{|E|} \Gamma(x- \iota_{E} y).
\]
Indeed, for \(k=1\), recall that \(2d\,g_A(x,y)\) is the expected number of
visits to \(y\) by the random walk starting from \(x\) before it exits \(A\).
By the symmetry of \(\mathbb Z^d\), we have
\begin{align*}
	g_A(x,y)
	=& g(x,y)-\mathbb E_x[g(X_{\tau_A},y)]\\
	=& g(x,y)-\mathbb E_x[g(X_{\tau_A},\iota_{\{1\}}y)]\\
	=& g(x,y)-g(x,\iota_{\{1\}}y).
\end{align*}
Here the second equality uses reflection symmetry on the boundary, whereas
the third follows from the strong Markov property. Thus the formula is
the probabilistic reflection principle for the Green function.

Furthermore, let us write $g_A(x,y)$ as
\[
g_A(x,y) = I_1 + I_2,
\]
where
\[
I_1 := \sum_{\iota_{E} \in \mathcal R_k}(-1)^{|E|}F_{2k}(x-\iota_Ey),
\qquad
I_2 := \sum_{\iota_{E} \in \mathcal R_k}(-1)^{|E|}E_{2k}(x-\iota_Ey).
\]

Note that the squared distance between $x$ and the reflected point $\iota_{E}y$ is
\[
|x-\iota_{E}y|^2 = r^2 + \sum_{i\in E}\ell_i.
\]
For $1 \le i \le k$, let us define the difference operator $\nabla_i$ by
\[
\nabla_i  F_{2k}(x-y) = F_{2k}(x-y)-F_{2k}(x-\iota_iy).
\]
Clearly,
\[
I_1 = \left( \prod_{i=1}^k \nabla_i \right)  F_{2k}(x-y).
\]

Using \eqref{fandphi} and applying the Newton-Leibniz formula, we rewrite $I_1$ as
\begin{equation}\label{eq:I1-integral}
I_1
= \int_{[0,\ell_1]\times\cdots\times[0,\ell_k]}
(-1)^k\partial_{s_1}\cdots\partial_{s_k}
\Phi(a_1+t_1,\dots,a_k+t_k,w)
\,dt_1\cdots dt_k.
\end{equation}
Observing that
\[
|w|^2+(a_1+t_1)+\cdots+(a_k+t_k) = r^2+t_1+\cdots+t_k,
\]
we combine \eqref{eq:mixed-deriv-asym} and \eqref{eq:I1-integral} to obtain
\begin{equation}\label{eq:I1-main-error}
I_1
= c_{d,k}\,J_\theta(r^2;\ell_1,\dots,\ell_k)
+ O\bigl(r^{-1}J_\theta(r^2;\ell_1,\dots,\ell_k)\bigr),
\end{equation}
where the function $J_\theta$ is defined by
\begin{align}\label{def-ja}
J_\theta(a;\ell_1,\dots,\ell_k)
:= \int_{[0,\ell_1]\times\cdots\times[0,\ell_k]}
(a+t_1+\cdots+t_k)^{-\theta-k}
\,dt_1\cdots dt_k.
\end{align}
The term $r^{-1}$ in the error term follows from the observation that
\[
(r^2+t_1+\cdots+t_k)^{-\theta-k-\frac12}
\le r^{-1}(r^2+t_1+\cdots+t_k)^{-\theta-k}.
\]
Thus the bound for $I_1$ reduces to obtaining a good estimate of $J_\theta$, which is given in Lemma \ref{lem:Jalpha} below.

Applying Lemma~\ref{lem:Jalpha} below with \eqref{eq:I1-main-error} and $\ell_i = 4x_i y_i$, we obtain that there exists $R_0\ge 1$ such that for $r\ge  R_0$
\begin{equation}\label{eq:I1-good}
I_1 \asymp r^{2-d}\prod_{i=1}^k\left(1\wedge \frac{x_i y_i}{r^2}\right),
\end{equation}

We now turn to the error term $I_2$. Since $|x-\iota_Ey| \ge |x-y| = r$, we have
\[
|I_2|
\le \sum_{E\subset\{1,\dots,k\}}|E_{2k}(x-\iota_Ey)|
\le 2^k\sup_{|z|\ge r}|E_{2k}(z)|
= o(r^{2-d-2k}).
\]
Conversely, because $x, y\in A$, we have $x_i, y_i \ge 1$ for $1\le  i\le  k$, and hence
\[
1\wedge \frac{x_i y_i}{r^2} \ge \frac{1}{r^2}
\qquad (r \ge 1).
\]
This guarantees that the right-hand side of \eqref{eq:I1-good} is bounded strictly below by a constant multiple of $r^{2-d-2k}$. Consequently,
\[
|I_2| = o(I_1)
\qquad (r \to \infty).
\]
Therefore, we can choose $R_0$ large enough such that
\[
\frac12 I_1 \le g_A(x,y) = I_1 + I_2 \le \frac32 I_1
\qquad \text{when } |x-y| \ge R_0,
\]
and the theorem follows directly from the estimates of $I_1$ established in \eqref{eq:I1-good}.
\end{proof}

Now we give the estimate of $J_\theta(a;\ell_1,\dots,\ell_k)$ defined in (\ref{def-ja}).
\begin{lemma}\label{lem:Jalpha}
Let $\beta > 0$ and $\ell_1,\dots,\ell_k \ge 0$. There exist positive constants $c_{\beta,k}$ and $C_{\beta,k}$ such that for every $a > 0$,
\[
c_{\beta,k}a^{-\beta}\prod_{i=1}^k\left(1\wedge \frac{\ell_i}{a}\right)
\le J_\beta(a;\ell_1,\dots,\ell_k)
\le C_{\beta,k}a^{-\beta}\prod_{i=1}^k\left(1\wedge \frac{\ell_i}{a}\right).
\]
\end{lemma}

\begin{proof}
Applying the change of variables $t_i = au_i$, we can scale the integral to
\[
J_\beta(a;\ell_1,\dots,\ell_k)
= a^{-\beta}
\int_{[0,\eta_1]\times\cdots\times[0,\eta_k]}
(1+u_1+\cdots+u_k)^{-\beta-k}
\,du_1\cdots du_k,
\]
where $\eta_i := \ell_i/a$. Partition the index set for the integration by defining
\[
M := \{i : \eta_i \le 1\}.
\]
For the lower bound, we restrict the domain of integration to the hypercube
\[
0 \le u_i \le \eta_i\wedge 1,
\qquad 1 \le i \le k.
\]
On this restricted hypercube, the integrand satisfies $1+u_1+\cdots+u_k \le 1+k$. Therefore,
\[
J_\beta(a;\ell_1,\dots,\ell_k)
\ge (1+k)^{-\beta-k}a^{-\beta}
\prod_{i=1}^k(\eta_i\wedge 1).
\]
For the upper bound, we integrate over exactly $[0,\eta_i]$ for indices $i \in M$, and extend the integration domain to $[0,\infty)$ for indices $i \notin M$. This yields
\[
J_\beta(a;\ell_1,\dots,\ell_k)
\le a^{-\beta}\Bigl(\prod_{i\in M}\eta_i\Bigr)
\int_{[0,\infty)^{k-|M|}}(1+v_1+\cdots+v_{k-|M|})^{-\beta-k}\,dv.
\]
Because $\beta > 0$, this remaining integral converges to a finite constant, which concludes the proof of the lemma.
\end{proof}

\begin{theorem}\label{thm:global}
For all $x,y \in A$, we have
\begin{align}\label{est-ga}
g_A(x,y)
\asymp (1+|x-y|)^{2-d}\prod_{i=1}^k\left(1\wedge \frac{x_i y_i}{(1+|x-y|)^2}\right).
\end{align}
\end{theorem}

\begin{proof}
For the large-distance regime $|x-y| \ge R_0$, this result follows immediately from Theorem~\ref{thm:large}, since
\[
1+|x-y| \asymp |x-y|
\qquad (|x-y| \ge R_0).
\]
It remains to verify the bound in the short-distance regime $|x-y| \le R_0$.

For the upper bound, we use the trivial inequality that the orthant Green function cannot exceed the full-space Green function evaluated at the origin:
\[
0 < g_A(x,y) \le g(x,y) \le g(0,0).
\]
For the corresponding lower bound, let $m := d(x,y) = \|x-y\|_1$, where $\|\cdot\|_1$ denotes the $\ell^1$-distance on $\mathbb{Z}^d$. Then the probability of the simple random walk reaching $y$ without exiting $A$ provides a lower bound:
\[
g_A(x,y)
= \frac1{2d}\sum_{n=0}^\infty \mathbb{P}_x(X_n=y,\ n<\tau_{A})
\ge \frac1{2d}\,\mathbb{P}_x(X_{m+2}=y,\ m+2<\tau_{A})
\ge (2d)^{-(m+3)},
\]
where we used the fact that for $m\ge0$, if $d(x,y)=m$, there is a path in $A$ of length $m+2$ connecting $x$ and $y$.

Because \(m\) is comparable to the Euclidean distance,
\[
|x-y| \le m \le \sqrt{d}\,|x-y| \le \sqrt{d}\,R_0,
\]
we obtain a uniform local lower bound:
\[
g_A(x,y) \ge (2d)^{-(\sqrt{d} R_0+3)}> 0
\qquad (|x-y| \le R_0).
\]
Finally, define the comparison function
\[
\Theta(x,y) := (1+|x-y|)^{2-d}\prod_{i=1}^k\left(1\wedge \frac{x_i y_i}{(1+|x-y|)^2}\right).
\]
If $|x-y| \le R_0$, then because $x_i y_i \ge 1$ for all $1 \le i \le k$, we know that
\[
(1+R_0)^{2-d-2k} \le \Theta(x,y) \le 1.
\]
This shows $\Theta(x,y)$ is uniformly bounded between positive constants when $|x-y| \le R_0$. Combining this with the uniform upper and lower bounds for $g_A(x,y)$ concludes the proof.
\end{proof}

Fix a reference point $o\in A$ and define the coordinate product function
\[
\HH(x) := \prod_{i=1}^k x_i, \qquad x\in A.
\]

\begin{proposition}\label{prop:fixed-pole}
For all $y\in A$, we have
\begin{equation}
\label{eq:fixed-pole}
 g_A(o,y)
 \asymp \HH(y)\,|y|^{-(d+2k-2)}.
\end{equation}
\end{proposition}

\begin{proof}
By Theorem \ref{thm:global},
\[
g_A(o,y)\asymp (1+|o-y|)^{2-d}
\prod_{i=1}^k \left(1\wedge \frac{o_i y_i}{(1+|o-y|)^2}\right).
\]
Since $o$ is a fixed point, we clearly have $1+|o-y| \asymp |y|$ for $|y|>2|o|$. Also, each factor inside the product satisfies
\[
1\wedge \frac{o_i y_i}{(1+|o-y|)^2}\asymp \frac{y_i}{|y|^2},
\qquad 1\le i\le k,
\]
because $o_i \ge 1$ is a fixed constant and $y_i \le |y|$. Hence for $|y|>2|o|$,
\[
g_A(o,y)\asymp |y|^{2-d}\prod_{i=1}^k \frac{y_i}{|y|^2}
= \HH(y)\,|y|^{-(d+2k-2)}.
\]

The case $|y|\le 2|o|$ involves only finitely many vertices and is absorbed into the constants. This proves \eqref{eq:fixed-pole}.
\end{proof}
The power \(d+k-2\) appearing below comes from \eqref{eq:fixed-pole}.
When $y_i\asymp|y|$ for each $1\le i \le k$, one has \(\HH(y)\asymp |y|^k\), so the pole Green
function decays like \(|y|^{-(d+k-2)}\), as in a space of effective dimension
\(d+k\).


\begin{theorem}\label{thm:orthant-critical-exponent}
Fix $d\ge 3$, $1\le k\le d$, $o\in A$, $q>1$, and $\alpha\ge0$. Then the following are equivalent:
\begin{itemize}
\item[(i)] there exists $C>0$ such that
\begin{equation}
	\label{eq:goal}
\sum_{y\in A} g_A(x,y)\,|y|^{-\alpha}\,\bigl(g_A(o,y)\bigr)^q
\le C\,g_A(o,x)
\end{equation}
for all $x\in A$;
\item[(ii)]
\begin{equation}
\label{eq:criterion}
\alpha+q(d+k-2)>d+k.
\end{equation}
\end{itemize}
Equivalently,
\[
q>\frac{d+k-\alpha}{d+k-2}.
\]
\end{theorem}

\begin{remark}
	When \(k=1\), the domain \(A_{d,1}\) is the discrete half-space in
	\(\mathbb Z^d\). For \(\alpha=0\), the criterion becomes
	\[
	q>\frac{d+1}{d-1},
	\]
	which agrees with the classical Euclidean half-space exponent in
	\(\mathbb R^d_+\).
\end{remark}

\begin{proof}
We divide the proof into necessity and sufficiency.

\medskip
\noindent
\textbf{Necessity.}
Assume \eqref{eq:goal} holds for some $C>0$.
We prove \eqref{eq:criterion}.

For $R\ge 1$ define
\[
x^{(R)}:=(\underbrace{R,\dots,R}_{k},0,\dots,0)\in A.
\]
Choose $R=2^{M+4}$ with $M\in \mathbb N$ large.
By Proposition \ref{prop:fixed-pole},
\begin{equation}
\label{eq:KxR}
g_A(o,x^{(R)})\asymp \HH(x^{(R)})\,R^{-(d+2k-2)}=R^kR^{-(d+2k-2)}=R^{2-d-k}.
\end{equation}

For each $m=0,1,\dots,M$, define the dyadic box
\[
B_m:=\Bigl\{y\in A:\ 2^m\le y_i<2^{m+1}\ (1\le i\le k),\quad |y_j|\le 2^m\ (k+1\le j\le d)\Bigr\}.
\]
Then, uniformly in $m$,
\begin{equation}
\label{eq:box-size}
\#B_m\asymp 2^{md},
\qquad
\HH(y)\asymp 2^{mk},
\qquad
|y|\asymp 2^m,
\qquad y\in B_m.
\end{equation}
Moreover, since $2^{m+1}\le 2^{M+1}=R/8$, every $y\in B_m$ satisfies
\[
|x^{(R)}-y|\asymp R.
\]
Also, for $1\le i\le k$,
\[
\frac{x^{(R)}_i y_i}{|x^{(R)}-y|^2}
\asymp \frac{R\,y_i}{R^2}
\asymp \frac{y_i}{R}
\le \frac18<1.
\]
Hence the global Green estimate \eqref{est-ga} gives, for $y\in B_m$,
\[
g_A(x^{(R)},y)
\gtrsim R^{2-d}\prod_{i=1}^k \frac{R y_i}{R^2}
= \HH(y)\,R^{2-d-k}.
\]
Combining this with \eqref{eq:KxR}, we obtain
\[
g_A(x^{(R)},y)\gtrsim g_A(o,x^{(R)})\,\HH(y),
\qquad y\in B_m,\ 0\le m\le M.
\]

Therefore, from \eqref{eq:goal},
\begin{align*}
C
&\ge \frac1{g_A(o,x^{(R)})}
\sum_{y\in A} g_A(x^{(R)},y)\,|y|^{-\alpha}\,\bigl(g_A(o,y)\bigr)^q \\
&\ge c\sum_{m=0}^{M}\sum_{y\in B_m}
\HH(y)\,|y|^{-\alpha}\,\bigl(g_A(o,y)\bigr)^q.
\end{align*}
Now Proposition \ref{prop:fixed-pole} and \eqref{eq:box-size} yield, for $y\in B_m$,
\[
g_A(o,y)\asymp \HH(y)\,|y|^{-(d+2k-2)}
\asymp 2^{mk}2^{-m(d+2k-2)}
=2^{-m(d+k-2)}.
\]
Hence
\[
\HH(y)\,|y|^{-\alpha}\,\bigl(g_A(o,y)\bigr)^q
\asymp 2^{mk}\,2^{-m\alpha}\,2^{-mq(d+k-2)}.
\]
Summing over $B_m$ and using \eqref{eq:box-size}, we obtain
\[
\sum_{y\in B_m}
\HH(y)\,|y|^{-\alpha}\,\bigl(g_A(o,y)\bigr)^q
\asymp 2^{m[d+k-\alpha-q(d+k-2)]}.
\]
Therefore
\[
C\ge c\sum_{m=0}^{M}2^{m[d+k-\alpha-q(d+k-2)]}
\qquad \text{for all }M\ge 0.
\]
The left-hand side is bounded independently of $M$, so the geometric sum on the right must remain bounded as $M\to\infty$.
This is possible only if
\[
d+k-\alpha-q(d+k-2)<0,
\]
which is exactly \eqref{eq:criterion}.
This proves necessity.

\medskip
\noindent
\textbf{Sufficiency.}
Assume now that \eqref{eq:criterion} holds.
We prove \eqref{eq:goal}.

For the sufficiency proof only, set
\[
W(y):=|y|^{-\alpha}\,\bigl(g_A(o,y)\bigr)^q,
\qquad y\in A,
\]
and
\[
L(x):=\sum_{y\in A}g_A(x,y)W(y).
\]
It remains to show that $L(x)\lesssim g_A(o,x)$ for all $x\in A$.

\smallskip
\noindent
\emph{Step 1: a summability estimate.}
We claim that
\begin{equation}
\label{eq:J-def}
\sum_{y\in A}\HH(y)\,W(y)<\infty.
\end{equation}
Indeed, by Proposition \ref{prop:fixed-pole},
\[
\HH(y)W(y)
\lesssim \HH(y)^{q+1}|y|^{-\alpha-q(d+2k-2)}.
\]
Since $\HH(y)\le |y|^k$, we obtain
\[
\HH(y)W(y)
\lesssim |y|^{k(q+1)-\alpha-q(d+2k-2)}
=|y|^{k-\alpha-q(d+k-2)}.
\]
Therefore, by shell counting,
\[
\sum_{y\in A}\HH(y)\,W(y)\lesssim \sum_{r=1}^{\infty} r^{d-1}r^{k-\alpha-q(d+k-2)}
=\sum_{r=1}^{\infty} r^{d+k-1-\alpha-q(d+k-2)}.
\]
This series converges because \eqref{eq:criterion} is equivalent to
\[
d+k-1-\alpha-q(d+k-2)<-1.
\]
So \eqref{eq:J-def} holds.

\smallskip
\noindent
\emph{Step 2: reduction to \(|x|>10\).}
For each fixed $x$, the series $L(x)$ converges.
Indeed, for large $|y|$, Theorem \ref{thm:global} gives
\[
g_A(x,y)\lesssim \HH(x)\HH(y)|y|^{-d-2k+2},
\]
and Proposition \ref{prop:fixed-pole} gives
\[
W(y)\lesssim \HH(y)^q|y|^{-\alpha-q(d+2k-2)}.
\]
Hence
\[
g_A(x,y)W(y)
\lesssim \HH(x)\,\HH(y)^{q+1}|y|^{-\alpha-(q+1)(d+2k-2)}
\lesssim \HH(x)|y|^{-d-k+2-\alpha-q(d+k-2)}.
\]
The shell exponent is
\[
d-1+\bigl(-d-k+2-\alpha-q(d+k-2)\bigr)=1-k-\alpha-q(d+k-2),
\]
which is strictly less than $-1$ under \eqref{eq:criterion}. Thus the tail is summable.
Since there are only finitely many $x\in A$ with $|x|\le 10$, it suffices to prove the desired estimate for $|x|>10$.

Fix now $x\in A$ with $|x|>10$, and set
\[
R:=1+|x|.
\]
We split $A$ into four regions:
\begin{align*}
\Omega_1&:=\{y\in A:\ |y|\le R/2\},\\
\Omega_2&:=\{y\in A:\ |y|>R/2,\ |x-y|\le R/2\},\\
\Omega_3&:=\{y\in A:\ R/2<|y|\le 2R,\ |x-y|>R/2\},\\
\Omega_4&:=\{y\in A:\ |y|>2R\}.
\end{align*}
Write accordingly
\[
L(x)=I_1+I_2+I_3+I_4,
\qquad
I_j:=\sum_{y\in \Omega_j}g_A(x,y)W(y).
\]

\smallskip
\noindent
\emph{Estimate on $\Omega_1$.}
If $y\in \Omega_1$, then $|x-y|\asymp R$. By Theorem \ref{thm:global},
\[
g_A(x,y)
\lesssim R^{2-d}\prod_{i=1}^k \frac{x_i y_i}{R^2}
= \HH(x)\HH(y)R^{-d-2k+2}.
\]
By Proposition \ref{prop:fixed-pole},
\[
g_A(o,x)\asymp \HH(x)R^{-d-2k+2}.
\]
Hence
\[
g_A(x,y)\lesssim g_A(o,x)\,\HH(y),
\qquad y\in \Omega_1.
\]
Therefore, using \eqref{eq:J-def},
\[
I_1\lesssim g_A(o,x)\sum_{y\in \Omega_1}\HH(y)W(y)
\le g_A(o,x)\sum_{y\in A}\HH(y)W(y)
\lesssim g_A(o,x).
\]

\smallskip
\noindent
\emph{Estimate on $\Omega_2$.}
If $y\in \Omega_2$, then $|y|\asymp R$. Hence Proposition \ref{prop:fixed-pole} gives
\[
W(y)
\lesssim R^{-\alpha}\,\HH(y)^q R^{-q(d+2k-2)}.
\]
Since $\HH(y)\le (CR)^k$, we have
\[
\HH(y)^q\le \HH(y)(CR^k)^{q-1}\lesssim \HH(y)R^{k(q-1)}.
\]
Thus
\begin{equation}
\label{eq:W-near-diagonal}
W(y)\lesssim \HH(y)\,R^{-\alpha-q(d+k-2)-k},
\qquad y\in \Omega_2.
\end{equation}

We now claim that if
\[
r:=1+|x-y|,
\]
then for every $1\le i\le k$,
\begin{equation}
\label{eq:1d-claim}
y_i\left(1\wedge \frac{x_i y_i}{r^2}\right)\lesssim x_i.
\end{equation}
Indeed, there are two cases.

If $x_i\ge r/2$, then
\[
y_i\le x_i+|x_i-y_i|\le x_i+(r-1)\le 3x_i,
\]
so the left-hand side of \eqref{eq:1d-claim} is at most $y_i\le 3x_i$.

If $x_i<r/2$, then
\[
y_i\le x_i+|x_i-y_i|<\frac r2 +(r-1)<\frac{3r}{2},
\]
and also
\[
\frac{x_i y_i}{r^2}\le \frac{x_i(3r/2)}{r^2}<\frac34<1.
\]
Hence the minimum is attained by the second term and
\[
y_i\left(1\wedge \frac{x_i y_i}{r^2}\right)
=\frac{x_i y_i^2}{r^2}
\le \frac{x_i(3r/2)^2}{r^2}
=\frac94 x_i.
\]
This proves \eqref{eq:1d-claim}.

Multiplying \eqref{eq:1d-claim} for $i=1,\dots,k$ and using Theorem \ref{thm:global}, we obtain
\begin{equation}
\label{eq:G-times-H}
g_A(x,y)\,\HH(y)\lesssim \HH(x)\,r^{2-d},
\qquad y\in \Omega_2.
\end{equation}
Combining \eqref{eq:W-near-diagonal} and \eqref{eq:G-times-H},
\begin{align*}
I_2
&\lesssim R^{-\alpha-q(d+k-2)-k}
\sum_{y\in \Omega_2} g_A(x,y)\HH(y) \\
&\lesssim \HH(x)R^{-\alpha-q(d+k-2)-k}
\sum_{|x-y|\le R/2}(1+|x-y|)^{2-d}.
\end{align*}
By shell counting,
\[
\sum_{|z|\le R/2}(1+|z|)^{2-d}
\lesssim \sum_{m=0}^{\lfloor R/2\rfloor}(1+m)^{d-1}(1+m)^{2-d}
\lesssim \sum_{m=0}^{\lfloor R/2\rfloor}(1+m)
\lesssim R^2.
\]
Therefore
\[
I_2\lesssim \HH(x)R^{2-k-\alpha-q(d+k-2)}.
\]
Since \eqref{eq:criterion} implies
\[
2-k-\alpha-q(d+k-2)\le 2-d-2k,
\]
we conclude, using Proposition \ref{prop:fixed-pole},
\[
I_2\lesssim \HH(x)R^{-d-2k+2}\asymp g_A(o,x).
\]

\smallskip
\noindent
\emph{Estimate on $\Omega_3$.}
If $y\in \Omega_3$, then $|y|\asymp R$ and $|x-y|\asymp R$.
For each $1\le i\le k$, one has $y_i\le |y|\lesssim R$, hence
\[
1\wedge \frac{x_i y_i}{|x-y|^2}\lesssim \frac{x_i y_i}{R^2}\lesssim \frac{x_i}{R}.
\]
Using Theorem \ref{thm:global}, we get
\[
g_A(x,y)\lesssim R^{2-d}\prod_{i=1}^k \frac{x_i}{R}=\HH(x)R^{2-d-k}.
\]
Also, because $|y|\asymp R$,
\[
W(y)
\lesssim R^{-\alpha}\,\HH(y)^q R^{-q(d+2k-2)}
\lesssim R^{-\alpha}R^{kq}R^{-q(d+2k-2)}
=R^{-\alpha-q(d+k-2)}.
\]
Since $\#\Omega_3\lesssim R^d$, it follows that
\[
I_3\lesssim \HH(x)R^{2-d-k}\cdot R^{-\alpha-q(d+k-2)}\cdot R^d
=\HH(x)R^{2-k-\alpha-q(d+k-2)}.
\]
As above, \eqref{eq:criterion} implies
\[
I_3\lesssim \HH(x)R^{-d-2k+2}\asymp g_A(o,x).
\]

\smallskip
\noindent
\emph{Estimate on $\Omega_4$.}
If $y\in \Omega_4$, then $|y|>2R>2|x|$, and so $|x-y|\asymp |y|$.
Using Theorem \ref{thm:global},
\[
g_A(x,y)
\lesssim |y|^{2-d}\prod_{i=1}^k \frac{x_i y_i}{|y|^2}
=\HH(x)\HH(y)|y|^{-d-2k+2}.
\]
Therefore
\[
I_4
\lesssim \HH(x)\sum_{|y|>2R} \HH(y)^{q+1}|y|^{-\alpha-(q+1)(d+2k-2)}.
\]
Since $\HH(y)\le |y|^k$, shell counting gives
\begin{align*}
I_4
&\lesssim \HH(x)\sum_{r>2R} r^{d-1}r^{k(q+1)}r^{-\alpha-(q+1)(d+2k-2)} \\
&= \HH(x)\sum_{r>2R} r^{1-k-\alpha-q(d+k-2)}.
\end{align*}
Under \eqref{eq:criterion}, the exponent $1-k-\alpha-q(d+k-2)$ is strictly less than $-1$, so the tail sum is bounded by
\[
\sum_{r>2R}r^{1-k-\alpha-q(d+k-2)}
\lesssim R^{2-k-\alpha-q(d+k-2)}.
\]
Hence
\[
I_4\lesssim \HH(x)R^{2-k-\alpha-q(d+k-2)}
\lesssim \HH(x)R^{-d-2k+2}
\asymp g_A(o,x).
\]

Combining the estimates on $I_1,I_2,I_3,I_4$, we obtain
\[
L(x)\lesssim g_A(o,x)
\qquad (|x|>10).
\]
As explained earlier, after enlarging the constant, this remains true for the finitely many $x\in A$ with $|x|\le 10$.
Thus \eqref{eq:goal} holds for all $x\in A$.
This proves sufficiency.
\end{proof}

\begin{corollary}\label{cor:k-orthant-existence}
Fix $d\ge 3$, $1\le k\le d$, $q>1$, and $\alpha\ge0$. Let
\[
A := \{x \in \mathbb{Z}^d : x_1 \ge 1, \dots, x_k \ge 1\}.
\]
The inequality
\[
\begin{cases}
-\Delta u\ge |x|^{-\alpha}u^q, & x\in A,\\
u=0, & x\in A^c,
\end{cases}
\]
admits a positive solution if and only if
\[
\alpha+q(d+k-2)>d+k.
\]
\end{corollary}

\begin{proof}
This follows from Theorem~\ref{thm:orthant-critical-exponent} and Theorem~\ref{thm:green-testing-equivalence}, since
\eqref{eq:goal} is equivalent, up to the harmless constant $2d$, to the condition \eqref{e1}   with \(\Omega=A\) and
\(\sigma(x)=|x|^{-\alpha}\).
\end{proof}

\medskip

	\paragraph{\textbf{Acknowledgments}.}
Part of this work was carried out while Y. Sun was visiting the University of Bielefeld and the University of Perugia. Y. Sun would like to thank both institutions for their hospitality, and Professors Grigor'yan and Filippucci for their invitations. The authors acknowledge the use of OpenAI's ChatGPT (GPT-5.5 Pro) as an interactive tool during the exploratory stage of this project. Examples \ref{ex:unweighted-nonexistence-not-Lq} and \ref{ex:diagonal-energy-not-testing} were found with the help of GPT-5.5 Pro. The authors take full responsibility for all mathematical statements, proofs,
and final wording.

\noindent\textbf{Data availability.} Data sharing is not applicable to this article, as no datasets were generated or analyzed during
the current study.

\noindent\textbf{Declaration.} On behalf of all authors, the corresponding author states that there is no conflict of interest.

\end{document}